\documentclass[10pt]{article}

\usepackage{enumerate,xspace}
\usepackage{amsmath,amssymb,wasysym}
\usepackage[all]{xy}
\usepackage{proof}
\usepackage[svgnames]{xcolor}
\usepackage{pict2e} 
\usepackage{tikz}
\usepackage{stmaryrd} 
\usepackage{mathtools}
\usepackage{latexsym}
\usepackage{dsfont}
\usepackage{multicol}
\usepackage{cmll}
\usepackage{multirow}
\usepackage{longtable}
\usepackage[sort,nocompress]{cite} 

\usepackage{stackrel}

\usepackage{lscape}
\usepackage{array}

\usepackage{hyperref} 
\hypersetup{
    colorlinks,
    citecolor=red,
    filecolor=red,
    linkcolor=blue,
    urlcolor=red
}

\usepackage{shuffle}

\newcommand{\ox}{\otimes}

\newcommand{\X}{\ensuremath{\mathbb X}\xspace}

\newcommand{\N}{{\ensuremath{\mathbb N}}\xspace}

\newtheorem{observation}{Remark}[section]
\newtheorem{lemma}[observation]{Lemma}  
\newtheorem{theorem}[observation]{Theorem}
\newtheorem{definition}[observation]{Definition}

\newtheorem{proposition}[observation]{Proposition} 
\newtheorem{corollary}[observation]{Corollary}

\newcommand{\wand}{\ensuremath{
  \mathrel{\vbox{\offinterlineskip\ialign{
    \hfil##\hfil\cr
    $\star$\cr
    \noalign{\kern-1ex}
    $\vert$\cr
}}}}}

\makeatletter

\newdimen\w@dth

\def\setw@dth#1#2{\setbox\z@\hbox{\scriptsize $#1$}\w@dth=\wd\z@
\setbox\@ne\hbox{\scriptsize $#2$}\ifnum\w@dth<\wd\@ne \w@dth=\wd\@ne \fi
\advance\w@dth by 1.2em}

\def\t@^#1_#2{\allowbreak\def\n@one{#1}\def\n@two{#2}\mathrel
{\setw@dth{#1}{#2}
\mathop{\hbox to \w@dth{\rightarrowfill}}\limits
\ifx\n@one\empty\else ^{\box\z@}\fi
\ifx\n@two\empty\else _{\box\@ne}\fi}}
\def\t@@^#1{\@ifnextchar_ {\t@^{#1}}{\t@^{#1}_{}}}

\def\t@left^#1_#2{\def\n@one{#1}\def\n@two{#2}\mathrel{\setw@dth{#1}{#2}
\mathop{\hbox to \w@dth{\leftarrowfill}}\limits
\ifx\n@one\empty\else ^{\box\z@}\fi
\ifx\n@two\empty\else _{\box\@ne}\fi}}
\def\t@@left^#1{\@ifnextchar_ {\t@left^{#1}}{\t@left^{#1}_{}}}

\def\two@^#1_#2{\def\n@one{#1}\def\n@two{#2}\mathrel{\setw@dth{#1}{#2}
\mathop{\vcenter{\hbox to \w@dth{\rightarrowfill}\kern-1.7ex
                 \hbox to \w@dth{\rightarrowfill}}%
       }\limits
\ifx\n@one\empty\else ^{\box\z@}\fi
\ifx\n@two\empty\else _{\box\@ne}\fi}}
\def\tw@@^#1{\@ifnextchar_ {\two@^{#1}}{\two@^{#1}_{}}}

\def\tofr@^#1_#2{\def\n@one{#1}\def\n@two{#2}\mathrel{\setw@dth{#1}{#2}
\mathop{\vcenter{\hbox to \w@dth{\rightarrowfill}\kern-1.7ex
                 \hbox to \w@dth{\leftarrowfill}}%
       }\limits
\ifx\n@one\empty\else ^{\box\z@}\fi
\ifx\n@two\empty\else _{\box\@ne}\fi}}
\def\t@fr@^#1{\@ifnextchar_ {\tofr@^{#1}}{\tofr@^{#1}_{}}}

\newdimen\W@dth
\def\setW@dth#1#2{\setbox\z@\hbox{$#1$}\W@dth=\wd\z@
\setbox\@ne\hbox{$#2$}\ifnum\W@dth<\wd\@ne \W@dth=\wd\@ne \fi
\advance\W@dth by 1.2em}

\def\T@^#1_#2{\allowbreak\def\N@one{#1}\def\N@two{#2}\mathrel
{\setW@dth{#1}{#2}
\mathop{\hbox to \W@dth{\rightarrowfill}}\limits
\ifx\N@one\empty\else ^{\box\z@}\fi
\ifx\N@two\empty\else _{\box\@ne}\fi}}
\def\T@@^#1{\@ifnextchar_ {\T@^{#1}}{\T@^{#1}_{}}}

\def\T@left^#1_#2{\def\N@one{#1}\def\N@two{#2}\mathrel{\setW@dth{#1}{#2}
\mathop{\hbox to \W@dth{\leftarrowfill}}\limits
\ifx\N@one\empty\else ^{\box\z@}\fi
\ifx\N@two\empty\else _{\box\@ne}\fi}}
\def\T@@left^#1{\@ifnextchar_ {\T@left^{#1}}{\T@left^{#1}_{}}}

\def\Tofr@^#1_#2{\def\N@one{#1}\def\N@two{#2}\mathrel{\setW@dth{#1}{#2}
\mathop{\vcenter{\hbox to \W@dth{\rightarrowfill}\kern-1.7ex
                 \hbox to \W@dth{\leftarrowfill}}%
       }\limits
\ifx\N@one\empty\else ^{\box\z@}\fi
\ifx\N@two\empty\else _{\box\@ne}\fi}}
\def\T@fr@^#1{\@ifnextchar_ {\Tofr@^{#1}}{\Tofr@^{#1}_{}}}

\def\Two@^#1_#2{\def\N@one{#1}\def\N@two{#2}\mathrel{\setW@dth{#1}{#2}
\mathop{\vcenter{\hbox to \W@dth{\rightarrowfill}\kern-1.7ex
                 \hbox to \W@dth{\rightarrowfill}}%
       }\limits
\ifx\N@one\empty\else ^{\box\z@}\fi
\ifx\N@two\empty\else _{\box\@ne}\fi}}
\def\Tw@@^#1{\@ifnextchar_ {\Two@^{#1}}{\Two@^{#1}_{}}}

\def\to{\@ifnextchar^ {\t@@}{\t@@^{}}}
\def\from{\@ifnextchar^ {\t@@left}{\t@@left^{}}}
\def\tofro{\@ifnextchar^ {\t@fr@}{\t@fr@^{}}}
\def\To{\@ifnextchar^ {\T@@}{\T@@^{}}}
\def\From{\@ifnextchar^ {\T@@left}{\T@@left^{}}}
\def\Two{\@ifnextchar^ {\Tw@@}{\Tw@@^{}}}
\def\Tofro{\@ifnextchar^ {\T@fr@}{\T@fr@^{}}}

\makeatother

\title{Drazin Inverses in Categories}
\author{Robin Cockett, Jean-Simon Pacaud Lemay, Priyaa Varshinee Srinivasan}

\begin{document}
\allowdisplaybreaks

\maketitle
\begin{center} \em \vspace{-0.5cm}
    This paper is dedicated to Bill Lawvere, a great mathematician, a mentor, and a friend.
\end{center} 
 \begin{abstract} Drazin inverses are a fundamental algebraic structure which have been extensively deployed in semigroup theory, ring theory, and matrix theory. Drazin inverses can also be defined for endomorphisms in any category. However, beyond a paper by Puystjens and Robinson from 1987, there has been almost no further development of Drazin inverses in category theory. Here we provide a survey of the theory of Drazin inverses from a categorical perspective. We introduce Drazin categories, in which every endomorphism has a Drazin inverse, and provide various examples including the category of matrices over a field, the category of finite length modules over a ring, and finite set enriched categories. We also introduce the notion of expressive rank and prove that a category with expressive rank is Drazin. Moreover, we not only study Drazin inverses in mere categories, but also in additive categories and dagger categories. In an arbitrary category, we show how a Drazin inverse corresponds to an isomorphism in the idempotent splitting, as well as explain how Drazin inverses relate to Leinster's notion of eventual image duality. In additive categories, we consider core-nilpotent decompositions, image-kernel decompositions, and Fitting decompositions. We also develop the notion of Drazin inverses for pairs of opposing maps, generalizing the usual notion of Drazin inverse for endomorphisms. As an application of this new kind of Drazin inverse, for dagger categories, we provide a novel characterization of the Moore-Penrose inverse in terms of being a Drazin inverse of the pair of a map and its adjoint.  
 \end{abstract}

 \noindent \small \textbf{Acknowledgements.} The authors would like to thank Mike Prest for his basic ring theoretic reassurance. For this research, the first named author was partially funded by NSERC, and the third named author was partially funded by the Estonian Research Council (MOB3JD1227).
\setcounter{tocdepth}{1} 
 \tableofcontents

 \newpage


\section{Introduction}


In a semigroup $S$, a \textbf{Drazin inverse} of $x \in S$ is a $x^D \in S$ such that:
{\bf [D.1]} there is a $k \in \mathbb{N}$ (called the Drazin index) such that $x^{k+1} x^D = x^k$; {\bf [D.2]} $x^D x x ^D = x^D$; and {\bf [D.3]} $x^D x = x x^D$. While a Drazin inverse may not always exist, if a Drazin inverse exists then it is unique, so we may speak of \emph{the} Drazin inverse. The term Drazin inverse is named after Michael P. Drazin, who originally introduced the concept of Drazin inverses in rings and semigroups under the modest name ``pseudo-inverse'' \cite{drazin1958pseudo}. The ``inverse'' part in the term ``Drazin inverse'' is justified since it is a generalization of the usual notion of inverse, that is, if $S$ is a monoid and $u \in S$ is invertible, then its inverse is also its Drazin inverse, $u^D = u^{-1}$, and conversely elements of $S$ with Drazin index zero are precisely the invertible elements of $S$. Drazin inverses have an extensive literature and have been well studied in both ring theory and semigroup theory \cite{chen2019hirano, drazin1958pseudo, drazin2013commuting, lam2014jacobson, marovt2018orders, munn1961pseudo, wang2017class}, and are deeply connected to strong $\pi$-regularity \cite{ara1996strongly,azumaya1954strongly, Dischinger,nicholson1999strongly,wang2017class} and Fitting's results \cite{Fitting1933,jacobson1985basic,leinstercounting} (sometimes called Fitting's Lemma or Fitting's Decomposition result). Moreover, they have also been extensively studied in matrix theory, and have been found to have many important applications as it is a useful tool in various computations \cite{bu2005linear, campbell2009generalized, campbell1976applications, puystjens2004drazin, romo2020core, wei2011note, yang2011drazin, yu2014note}.  

At the time Drazin first introduced the concept of Drazin inverses in \cite{drazin1958pseudo}, category theory was still in its infancy. Thus, the theory of Drazin inverses was not developed from a categorical perspective nor, conversely, did the notion of a Drazin inverse become widely known in the categorical community. To the best of our knowledge, the only discussion of Drazin inverses in category theory appears in a section of a paper by Puystjens and Robinson \cite{robinson1987generalized} from 1987, where they develop the properties of Drazin inverses in additive categories. 

The purpose of this paper is, thus, to develop Drazin inverses from a categorical perspective and to make their theory more readily available to a categorical audience. In order to achieve this we have tried to make the exposition as self-contained as possible and, in particular, we have included complete proofs. We do, however, acknowledge that many of the proofs below are available in the extensive ring theory and semigroup literature and we have tried to document this. Having said that, inevitably, there is new material in our exposition including, of course, new definitions, new examples, new constructions, and new proofs. These arise largely because the categorical perspective allows one to consider a broader range of settings which are not directly related to semigroups or rings. For example, without the categorical perspective, one might not have known how Drazin inverses relate to idempotent splittings or the idempotent completion.  Nor would one have noticed that having a Drazin inverse implies having eventual image duality in the sense of Leinster \cite{leinster2022eventual}. Even more fundamentally, one might not see the point of considering Drazin inverses for arbitrary maps, not just endomorphisms.  

We suspect this just scratches the surface of novel applications for Drazin inverses that the categorical perspective uncovers. This paper was never meant to be the final word on the subject of Drazin inverses in category theory. Rather our hope is that it can be a useful entry point for others into what is a fundamental algebraic structure.  So we will be happy if we have left loose ends which others can pursue.

\textbf{Conventions:} For an arbitrary category $\mathbb{X}$, we denote objects by capital letters $A,B,X,Y$, etc. and maps by lowercase letters $f,g,x, y$, etc. Homsets are denoted by $\mathbb{X}(A,B)$ and maps as $f: A \to B$. Identity maps are denoted as $1_A: A \to A$.  Composition is written in \emph{diagrammatic order}, that is, the composition of a map $f: A \to B$ followed by $g: B \to C$ is denoted $fg: A \to C$. 



\section{Drazin Inverses, Drazin Objects, and Drazin Categories}\label{sec:Drazin}


In this section, we discuss Drazin inverses in a category and introduce the notions of Drazin categories and Drazin objects. We also provide examples of these concepts. 

\subsection{Drazin Inverses} The concept of Drazin inverses in an arbitrary category was first considered by Puystjens and Robinson in \cite[Sec 2]{robinson1987generalized}, though they were particularly interested in studying Drazin inverses in \emph{additive} categories -- which we discuss in Section \ref{sec:additive}. The notion of Drazin inverses in a category $\mathbb{X}$ corresponds to the notion of Drazin inverses in the semigroup of endomorphisms, $\mathbb{X}(A,A)$, of each object $A$. Explicitly: 

\begin{definition}\label{def:Drazin} In a category $\mathbb{X}$, a \textbf{Drazin inverse} \cite[Sec 2]{robinson1987generalized} of $x: A \to A$ is an endomorphism $x^D: A \to A$ such that: 
\begin{enumerate}[{\bf [D.1]}]
\item There is a $k \in \mathbb{N}$ such that\footnote{By convention, $x^0 = 1_A$.} $x^{k+1} x^D = x^k$; 
\item $x^D x x ^D = x^D$; 
\item $x^D x = x x^D$. 
\end{enumerate}
If $x: A \to A$ has a Drazin inverse $x^D: A \to A$, we say that $x$ is \textbf{Drazin}, and call the least $k$ such that $x^{k+1} x^D = x^k$ the {\bf Drazin index of $x$}, which we denote by $\mathsf{ind}(x)=k$.  
\end{definition}

For the remainder of this section, we work in an arbitrary category $\mathbb{X}$. An important fact is that Drazin inverses, if they exist, are unique. Before proving this, here are some useful basic identities. 

\begin{lemma}\label{lemma:Drazin-basic} Let $x: A \to A$ be Drazin with Drazin inverse $x^D: A \to A$: 
\begin{enumerate}[(i)]
\item\label{lemma:Drazin-basic.1} For all $n \geq \mathsf{ind}(x)$, $x^{n+1} x^D = x^n = x^D x^{n+1}$;
\item \label{lemma:Drazin-basic.3} For all $n \geq \mathsf{ind}(x)$ and $m \in \mathbb{N}$, $x^{n+m} (x^D)^m = x^n = (x^D)^m x^{n+m}$;
\item\label{lemma:Drazin-basic.2} For all $n \in \mathbb{N}$, $x^{n}(x^D)^{n+1} = x^D = (x^D)^{n+1} x^n$.
\end{enumerate}
\end{lemma}
\begin{proof} (i) Let $\mathsf{ind}(x) = k$. Then for all $n \geq k$, using \textbf{[D.1]} we compute that: 
\[ x^{n+1} x^D = x^{n-k + k +1} x^D = x^{n-k} x^{k+1} x^D = x^{n-k} x^k = x^{n-k+k} = x^n \]
So $x^{n+1} x^D = x^n$. By \textbf{[D.3]} it follows that $x^D x^{n+1} = x^n$ as well. \\ 

\noindent (ii) We prove this by induction on $m$. For the base case $m=0$, we clearly have that $x^n (x^D)^0 = x^n$. 
Now suppose that for all $0 \leq j \leq m$, $x^{n+j} (x^D)^j = x^n$. Then using the induction hypothesis and (\ref{lemma:Drazin-basic.1}), we compute that: 
\[ x^{n+m+1} (x^D)^{m+1} = x x^{n+m} (x^D)^m x^D = x x^n x^D = x^{n+1} x^D = x^n \]
So $x^{n+m+1} (x^D)^{m+1} = x^n$. Therefore all $m$, $x^{n+m} (x^D)^m = x^n$ as desired. By \textbf{[D.3]} it follows that $ x^n = (x^D)^m x^{n+m}$ as well. \\  

\noindent (iii) We prove this by induction. For the base case $n=0$, we clearly have that $x^0 x^D = x^D$. Now suppose that for all $0 \leq j \leq n$, $x^j (x^D)^{j+1} = x^D$. Then using the induction hypothesis, \textbf{[D.2]}, and \textbf{[D.3]}, we compute that: 
\[  x^{n+1}(x^D)^{n+2} = x x^n (x^D)^{n+1} x^D = x x^D x^D = x^D x x^D = x^D \]
So $x^{n+1}(x^D)^{n+2} = x^D$. Therefore all $n$, $x^{n}(x^D)^{n+1} = x^D$ as desired. By \textbf{[D.3]} it follows that $(x^D)^{n+1} x^n = x^D$ as well.
\end{proof}

\begin{proposition}\cite[Thm 1]{drazin1958pseudo} \label{prop:unique} If $x: A \to A$ has a Drazin inverse, it is unique. 
\end{proposition}
\begin{proof} Suppose that $x: A \to A$ has two possible Drazin inverses $y: A \to A$ and $z: A \to A$. So explicitly, there is a $k \in \mathbb{N}$ such that $x^k xy = x^k$, and also that $yxy =y$ and $xy = yx$, and there is a $k' \in \mathbb{N}$ such that $x^{k'} xz = x^{k'}$, and also that $xz = zx$ and $zxz = z$. Now set $j = \max(k,k')$. First observe that by Lemma \ref{lemma:Drazin-basic}.(\ref{lemma:Drazin-basic.1}) we have that $x^{j+1} y = x^j$ and $z x^{j+1} = x^j$, while by Lemma \ref{lemma:Drazin-basic}.(\ref{lemma:Drazin-basic.2}) we have that $y = x^j y^{j+1}$ and $z = z^{j+1} x^j$. Then we compute that: 
\begin{gather*}
y = x^j y^{j+1} = z x^{j+1} y^{j+1} = z x x^{j} y^{j+1} = z x y =    z^{j+1} x x^j y = z^{j+1} x^{j+1} y = z^{j+1} x^j = z
\end{gather*}
So $y=z$, and we conclude that the Drazin inverse is unique. 
\end{proof}

From now on we may speak of \emph{the} Drazin inverse of an endomorphism $x$ (if it exists of course) and denote it by $x^D$. 

\subsection{Drazin Categories} We call a category in which every endomorphism has a Drazin inverse a {\bf Drazin category}. Since Drazin inverses are unique, being Drazin is a property of a category rather than a structure. 

It is always possible to construct a Drazin category from any category by considering the full subcategory determined by the objects whose every endomorphism is Drazin: we call these Drazin objects. 

\begin{definition}\label{def:Drazin-obj} An object $A$ is a \textbf{Drazin object} if every endomorphism $x: A \to A$ is Drazin. Then define $\mathsf{D}\left(\mathbb{X}\right)$ to be the full subcategory of Drazin objects of $\mathbb{X}$. 
\end{definition}

\begin{lemma} $\mathsf{D}\left(\mathbb{X}\right)$ is a Drazin category.  Moreover, $\mathbb{X}$ is Drazin if and only if $\mathsf{D}\left(\mathbb{X}\right) = \mathbb{X}$. 
\end{lemma}
\begin{proof} Every endomorphism in $\mathsf{D}\left(\mathbb{X}\right)$ is Drazin, so $\mathsf{D}\left(\mathbb{X}\right)$ is Drazin. Also, clearly $\mathbb{X}$ is Drazin if and only if every object is Drazin, which means that $\mathsf{D}\left(\mathbb{X}\right) = \mathbb{X}$. 
\end{proof}

Thus, an equivalent way of describing a Drazin category is as a category in which every object is Drazin. Here are now some examples of Drazin inverses, Drazin categories, and Drazin objects. 

\subsection{Complex Matrices}\label{ex:complex} Arguably, the best-known (and most applied) examples of Drazin inverses are those of complex matrices. So let $\mathbb{C}$ be the field of complex numbers and let $\mathsf{MAT}(\mathbb{C})$ be the category of complex matrices, that is, the category whose objects are natural numbers $n \in \mathbb{N}$ and where a map ${A: n \to m}$ is an $n \times m$ complex matrix. Composition given by matrix multiplication and the identity on $n$ is the $n$-dimensional identity matrix. Endomorphisms in $\mathsf{MAT}(\mathbb{C})$ correspond precisely to square matrices: so an endomorphism ${A: n \to n}$ is an $n \times n$ square matrix $A$ 
It is well-known that every complex square matrix has a Drazin inverse \cite[Chap 7]{campbell2009generalized} making $\mathsf{MAT}(\mathbb{C})$ Drazin. To compute the Drazin inverse, first recall that every $n \times n$ complex matrix $A$ can be written in the form \cite[Thm 7.2.1]{campbell2009generalized}:
\[ A = P \begin{bmatrix} C & 0 \\ 
0 & N
\end{bmatrix} P^{-1} \]
for some invertible $n \times n$ matrix $P$, an invertible $m \times m$ matrix $C$ (where $m \leq n$), and a nilpotent $n-m \times n-m$ matrix $N$ (that is, $N^k=0$ for some $k \in \mathbb{N}$). See \cite[Algorithm 7.2.1]{campbell2009generalized} for how this form is computed. The Drazin inverse of $A$ is the $n \times n$ matrix $A^D$ defined as follows: 
\[ A^D = P \begin{bmatrix} C^{-1} & 0 \\
0 & 0
\end{bmatrix} P^{-1} \]
The Drazin index of $A$ corresponds precisely to the \textbf{index} of $A$ \cite[Def 7.2.1]{campbell2009generalized}, that is the least $k \in \mathbb{N}$ such that $\mathsf{rank}(A^{k+1}) = \mathsf{rank}(A^k)$. 

In fact, it is well-known \cite{bu2005linear} -- and, as we shall shortly show in Section \ref{rank} -- that a square matrix over any field has a Drazin inverse.  Thus, for any field $k$, its category of matrices $\mathsf{MAT}(k)$ is always a Drazin category. Moreover, as explained in Corollary \ref{cor:equivalence} below, a category which is equivalent to a Drazin category is certainly itself Drazin. So $k\text{-}\mathsf{FVEC}$, the category of finite-dimensional $k$-vector spaces and $k$-linear maps between them, is Drazin as well.

\subsection{Modules}\label{sec:modules} Let $R$ be a ring and let $R\text{-}\mathsf{MOD}$ be the category of (left) $R$-modules and $R$-linear morphisms between them. In general, $R\text{-}\mathsf{MOD}$ is not Drazin. Indeed, when $R = \mathbb{Z}$, consider the $\mathbb{Z}$-linear endomorphism $f: \mathbb{Z} \to \mathbb{Z}$ given by multiplying by two, $f(x) = 2x$. Now if $f$ had a Drazin inverse $f^D: \mathbb{Z} \to \mathbb{Z}$, it must be of the form $f^D(x) = nx$ for some $n \in \mathbb{Z}$. Then by \textbf{[D.1]} we would have that for some $k \in \mathbb{N}$, $2^{k+1}nx = 2^kx$ for all $x \in \mathbb{Z}$. This would imply that $2n=1$, which is a contradiction. So $\mathbb{Z}\text{-}\mathsf{MOD}$ (which recall is equivalent to the category of Abelian groups) is not Drazin. 

That said, while $R\text{-}\mathsf{MOD}$ may not always be Drazin, there are various characterizations of Drazin $R$-linear endomorphism -- see for example \cite[Lemma 2.1]{wang2017class}. In particular, an $R$-linear endomorphism $f: M \to M$ has a Drazin inverse if and only if $f$ is \emph{strongly $\pi$-regular} -- which will be discussed in Section \ref{sec:strong-pi} below. Alternatively, an $R$-linear endomorphism ${f: M \to M}$ is Drazin if and only if $M = \mathsf{im}(f^k) \oplus \mathsf{ker}(f^k)$ for some $k \geq 1$ \cite[Lemma 2.1.(4)]{wang2017class}, and this decomposition is sometimes called \textbf{Fitting's decomposition} \cite[Lemma 2.4]{leinstercounting}. Then an $R$-module $M$ is said to satisfy \textbf{Fitting's Lemma} \cite[Page 665]{armendariz1978injective} if every endomorphism has a Fitting's decomposition (or equivalently if $R\text{-}\mathsf{MOD}(M,M)$ is strongly $\pi$-regular \cite[Prop 2.3]{armendariz1978injective}). As such, the Drazin objects in $R\text{-}\mathsf{MOD}$ are precisely the $R$-modules which satisfy Fitting's Lemma, or in other words, $\mathsf{D}\left(R\text{-}\mathsf{MOD} \right)$ is the full subcategory of $R$-modules which satisfy Fitting's Lemma. For more on modules which satisfy Fitting's Lemma, see \cite{ara1996strongly, armendariz1978injective, nicholson1999strongly}, and also see \cite{armendariz1978injective} for classes of rings for which all finitely generated modules satisfy Fitting's Lemma and are therefore Drazin.

Famously, ``Fitting's Decomposition Theorem'' says that for every endomorphism of a \textbf{finite length} $R$-module has a Fitting's decomposition \cite[Page 113]{jacobson1985basic}. This implies that every $R$-linear endomorphism of a finite length $R$-module is Drazin, and thus finite length $R$-modules are Drazin. Therefore the full subcategory of finite length $R$-modules is Drazin. It is important to note that while every finite length $R$-module is Drazin, there are modules which do not have finite length which are Drazin. Indeed, consider the rationals $\mathbb{Q}$ seen as a $\mathbb{Z}$-module. Every $\mathbb{Z}$-linear endomorphism $f: \mathbb{Q} \to \mathbb{Q}$ is also $\mathbb{Q}$-linear, therefore $f(x) = \frac{p}{q} x$ for some fixed $\frac{p}{q} \in \mathbb{Q}$. Therefore, every $\mathbb{Z}$-linear endomorphism $f: \mathbb{Q} \to \mathbb{Q}$ is either zero or an isomorphism. If $f$ is zero, then as we will see in Lemma \ref{lemma:nilpotent-drazin}, $f$ is Drazin with $f^D=0$. On the other hand if $f$ is an isomorphism, so $f(x) = \frac{p}{q} x$ with $p, q \neq 0$, then as we we will see in Lemma \ref{lem:Drazin-0}, $f$ is Drazin with $f^D(x) = f^{-1}(x) =  \frac{q}{p} x$. So it follows that $\mathbb{Q}$ is Drazin. Modules of finite length are always finitely generated, while $\mathbb{Q}$ is famously not finitely generated as a $\mathbb{Z}$-module, and therefore cannot be of finite length as a $\mathbb{Z}$-module. So $\mathbb{Q}$ is a $\mathbb{Z}$-module which is Drazin but not of finite length. As such, the category of finite length $\mathbb{Z}$-modules is a proper subcategory of $\mathsf{D}\left(\mathbb{Z}\text{-}\mathsf{MOD} \right)$. 

\subsection{Finite Sets}\label{ex:finset} Let $\mathsf{SET}$ be the category of sets and functions between them, and let $\mathsf{FinSET}$ be the full subcategory of finite sets. Not all Drazin inverses exist in $\mathsf{SET}$. As an example, consider the successor function $s: \mathbb{N} \to \mathbb{N}$, $s(n) = n+1$. Suppose that $s$ had a Drazin inverse $s^D: \mathbb{N} \to \mathbb{N}$. By \textbf{[D.3]}, we would have that $s^D(n) = s^D(s^n(0)) =  s^n(s^D(0)) = s^D(0) + n$. Now if $\mathsf{ind}(s) =k$, by \textbf{[D.1]} we compute that $k = s^k(0) = s^D(s^{k+1}(0)) = s^D(k+1) = s^D(0) + k +1$. This implies that $0 = s^D(0) + 1$ -- which is a contradiction since $s^D(0) \in \mathbb{N}$. So the successor function does not have a Drazin inverse, and therefore $\mathsf{SET}$ is not Drazin. 

On the other hand, $\mathsf{FinSET}$ does have all Drazin inverses. In fact, every finite set enriched category is Drazin. Recall that a category is finite set enriched if and only if every homset is finite. 

\begin{lemma}\label{lem:finset-enriched} Every finite set enriched category is Drazin. 
\end{lemma}
\begin{proof} Let $\mathbb{X}$ be a finite set enriched category.  Let $x: A \to A$ be an endomorphism in $\mathbb{X}$. Since $\mathbb{X}(X,X)$ is a finite set, of cardinality $\vert \mathbb{X}(X,X) \vert$, we have that $x^0 = 1_A, x, x^2, x^3, \hdots,$ and $x^{\vert \mathbb{X}(X,X) \vert}$ are not all distinct. So as explained in the proof of \cite[Prop 6.3]{leinster2022eventual}, there is a $m\geq 0$ and a $k\geq 1$ with $m+k \leq \vert \mathbb{X}(X,X) \vert$, such that $x^m = x^{m+k}$. Moreover, $x^{n+rk} = x^n$ for all $n \geq m$ and $r \geq 0$, and so in particular $x^{mk} = x^{2mk}$, implying $x^{mk}$ is an idempotent. We now consider the following cases: 
\begin{enumerate}[(i)]
\item If $m=0$, then $x$ is an isomorphism with inverse $x^{k-1}$. Then as we will review in Lemma \ref{lem:Drazin-0}, $x$ is Drazin with $x^D= x^{k-1}$.
\item If $k=1$, then $x^{m} = x^{m+1}$ and $x^m$ is an idempotent. So setting $x^D = x^m$, we check that three Drazin identities hold: 
\begin{enumerate}[{\bf [D.1]}]
\item  $x^{m+1} x^D = x^{m+1}x^m = x^m x^m = x^m$
\item $x^D x x^D = x^m x x^m = x^{m+1}x^m = x^m x^m = x^m$
\item $x^D x = x^m x = x^{m+1} = x x^m = x x^D$
\end{enumerate}
So $x^D$ is the Drazin inverse of $x$. 
\item If $m \geq 1$ and $k \geq 2$, then $mk-1 \geq m \geq 0$. So setting $x^D = x^{mk-1}$, we show that the three Drazin axioms hold: 
\begin{enumerate}[{\bf [D.1]}]
\item  $x^{mk+1}x^D = x^{mk+1} x^{mk-1} = x^{mk+1+mk-1} = x^{2mk} = x^{mk}$
\item $x^D x x^D = x^{mk-1} x x^{mk-1} = x^{mk-1+1+mk-1} = x^{mk-1 + mk} = x^{mk-1} = x^D$
\item   $xx^D = xx^{mk-1} = x^{1+mk-1} = x^{mk-1+1}= x^{mk-1}x = x^Dx$
\end{enumerate}
So $x^D$ is the Drazin inverse of $x$.
\end{enumerate}
So, every $x: A \to A$ is Drazin, and therefore $\mathbb{X}$ is Drazin. 
\end{proof}

\begin{corollary} \label{cor:finiteset} {\sf FinSET} is a Drazin category. 
\end{corollary}

When $X$ is a finite set, another way of understanding the Drazin inverse of a function $f: X \to X$, is to consider the inclusion of subsets $X \supseteq \mathsf{im}(f) \supseteq \mathsf{im}(f^2) \supseteq \hdots$, which must eventually stabilize after at most $\vert X \vert$ steps. So there is a smallest $k$ such that $\mathsf{im}(f^k) =  \mathsf{im}(f^{k+1}) = \hdots$.  Then $f$ becomes an isomorphism on $\mathsf{im}(f^k)$, which we denote as $f\vert_{\mathsf{im}(f^k)}: \mathsf{im}(f^k) \to \mathsf{im}(f^k)$. The Drazin inverse of $f$ is defined as $f^D(x) = f\vert_{\mathsf{im}(f^k)}^{-1}(f^k(x))$, and its Drazin index is $\mathsf{ind}(f)=k$.

On the other hand, while $\mathsf{SET}$ may not be Drazin, we may still ask what are the Drazin objects in $\mathsf{SET}$. It turns out that they are precisely the finite sets. 

\begin{lemma}\label{Drazin-set} A set $X$ is Drazin in $\mathsf{SET}$ if and only if $X$ is a finite set. Therefore $\mathsf{D}\left( \mathsf{SET} \right) = \mathsf{FinSET}$. 
\end{lemma}
\begin{proof} By Corollary \ref{cor:finiteset}, we know that finite sets are always Drazin.  So, it remains to show if a set is infinite, that it cannot be Drazin. So suppose that $X$ is an infinite set, then there is an injection $\phi: \mathbb{N} \to X$. Writing $\phi_n \colon = \phi(n)$, define the function $f: X \to X$ as $f(x) = x$ if $x \notin \mathsf{im}(\phi)$, otherwise if $x=\phi_n$ then set $f(x) = \phi_{n+1}$ -- which is well-defined since $\phi$ is injective. So in particular, $f^n(\phi_m ) = \phi_{m+n}$. If $X$ is Drazin, $f$ would have a Drazin inverse $f^D: X \to X$ and, if $\mathsf{ind}(f) = k$, then by \textbf{[D.1]} we would have that:  
\[ \phi_k = f^k(\phi_0) = f^D(f^{k+1}(\phi_0)) = f^D(\phi_{k+1}) \]
So $f^D(\phi_{k+1}) = \phi_k$. Now consider $f^D(\phi_0)$. There are two possible cases:
\begin{enumerate}[(i)]
\item Case 1: $f^D(\phi_0) \notin \mathsf{im}(\phi)$. By definition of $f$, we have that $f^{k+1}(f^D(\phi_0)) = f^D(\phi_0)$. However by \textbf{[D.1]} and \textbf{[D.3]}: 
\[f^D(\phi_0) =f^{k+1}(f^D(\phi_0)) = f^D( f^{k+1} (\phi_0 ) ) = f^k(\phi_0) = \phi_k\]
So $f^D(\phi_0) =\phi_k$ which means $f^D(\phi_0) \in \mathsf{im}(\phi)$, which is a contradiction. 
\item Case 2:  $f^D(\phi_0) \in \mathsf{im}(\phi)$. So there is some $n \in \mathbb{N}$ such that $f^D(\phi_0) = \phi_n$. But by \textbf{[D.2]}, we get that: 
\[ \phi_k = f^D(\phi_{k+1}) = f^D(f^{k+1}(\phi_0)) = f^{k+1}(f^D(\phi_0)) = f^{k+1}(\phi_n) = \phi_{n+k+1}   \]
Since $\phi$ is injective, this would imply that $k = n+k+1$. In turn, this implies that $0=n+1$, which is a contradiction since $n \in \mathbb{N}$.
\end{enumerate}
Since both cases lead to a contradiction, we conclude that $X$ cannot be an infinite set. Therefore, if $X$ is Drazin, $X$ must be a finite set. So $\mathsf{D}\left( \mathsf{SET} \right) = \mathsf{FinSet}$ as desired. 
\end{proof}



\section{Properties of Drazin Inverses}\label{sec:properties}


In this section, we review some well-known properties of Drazin inverses from ring theory literature, as well as providing some new properties of Drazin inverses with a more categorical flavour. In particular, we show that many basic categorical constructions behave nicely with respect to Drazin inverses. Thus, when applying these constructions to Drazin categories, we obtain new Drazin categories. Again, throughout this section we work in an arbitrary category $\mathbb{X}$. 


\subsection{Drazin maps whose composite is not Drazin.} \label{drazin-composition} We begin with the important observation that, unfortunately, Drazin inverses do not behave well with respect to composition. So here we take the opportunity to exhibit two Drazin endomorphisms whose composite is not Drazin. We first observe that a function $f: \mathbb{N} \to \mathbb{N}$ which is strictly increasing (i.e. $f(n) > n$) and 
order-preserving (i.e. if $m \leq n$ then $f(m) \leq f(n)$) cannot be Drazin. Indeed, suppose that there is a Drazin inverse $f^D$ with index $k$, then by \textbf{[D.1]} and \textbf{[D.3]} we get that:
\[ f^k(0) = f^D(f(f^k(0))) = f(f^k(f^D(0))) > f^k(f^D(0)) \geq f^k(0). \]
So $f^k(0) > f^k(0)$, which is a contradiction, and so $f$ does not have a Drazin inverse. Now consider the two idempotents:
\[ e(n) = \left\{ \begin{array}[c]{ll} n+1 & \mbox{if $n$ is even} \\  n & \mbox{otherwise} \end{array} \right. ~~~~~  e'(n) = \left\{ \begin{array}[c]{ll} n & \mbox{if $n$ is even} \\  n+1 & \mbox{otherwise} \end{array} \right.\]
As these are idempotents they certainly have a Drazin inverse: namely, themselves (as we will explain in Lemma \ref{lemma:e-drazin}). Their composite 
is the function:
\[ e'(e(n)) = \left\{ \begin{array}[c]{ll} n+2 & \mbox{if $n$ is even} \\  n+1 & \mbox{otherwise} \end{array} \right. \]
which is strictly increasing and order-preserving, and therefore not Drazin.

\subsection{Strong $\pi$-Regularity}\label{sec:strong-pi}
In ring theory, an important equivalent way of describing being Drazin is in terms of a being \emph{strongly $\pi$-regular} \cite[Sec 2]{azumaya1954strongly} and there is a extensive literature on strongly $\pi$-regular rings. This characterization of being Drazin is also valid for maps and objects in a category as we now explain. 

\begin{definition}\label{def:strong-pi} $x: A \to A$ is \textbf{strongly $\pi$-regular} if there exists endomorphisms ${y: A \to A}$ and $z: A \to A$, and $p,q \in \mathbb{N}$ such that $yx^{p+1}=x^p$ and $x^{q+1}z=x^q$. An object $A$ is said to be \textbf{strongly $\pi$-regular} if every endomorphism of $A$ is strongly $\pi$-regular. Similarly, a category $\mathbb{X}$ is said to be \textbf{strongly $\pi$-regular} if every endomorphism of $\mathbb{X}$ is strongly $\pi$-regular (or equivalently if every object is strongly $\pi$-regular). 
\end{definition}

\begin{lemma}\label{lem:pi}\cite[Thm 4]{drazin1958pseudo} $x: A \to A$ is Drazin if and only if it $x$ is strongly $\pi$-regular. Moreover, an object (resp. category) is Drazin if and only if it is strongly $\pi$-regular. 
\end{lemma}
\begin{proof} For ($\Rightarrow$), suppose that $x$ is Drazin with $\mathsf{ind}(x)=k$. Then set $y=z=x^D$ and $p=q=k$. By Lemma \ref{lemma:Drazin-basic}.(\ref{lemma:Drazin-basic.1}), we get $x^{q+1}z=x^q$ and $yx^{p+1}=x^p$ as desired. 

For ($\Leftarrow$), let $k = \mathsf{max}(p,q)$ and define $x^D \colon = x^{k}z^{k+1}$. We show that $x^D$ satisfies the three Drazin axioms: 
\begin{description}
\item[\textbf{[D.1]}] First observe that for all $n\in \mathbb{N}$, by assumption it follows that $x^{k+n+1} z^{n+1} = x^k$. Therefore, we compute that: 
\[ x^{k+1} x^D = x^{k+1} x^{k} z^{k+1} = x^{k+k+1} z^{k+1} = x^k  \]
\item[\textbf{[D.3]}] First observe that $yx^{k+1}=x^k = x^{k+1}z$, and so 
\[ y^n x^{k} = y^{n+1}x^{k+1} = x^{k+1}z^{n+1} = x^k z^n \] 
for all $n \in \mathbb{N}$. As such, we also have that $x^D = y^{k+1}x^k$. Using this, we compute that: 
\[ xx^D = xx^{k}z^{k+1}= x^{k+1} z z^{k} = x^k z^k = y^k x^k = y^k y x^{k+1} = y^{k+1} x^k x = x^D x  \]
\item[{\bf [D.2]}] Using \textbf{[D.3]} and \textbf{[D.1]}, we have: 
\[ x^D x x^D = x^D x x^{k}z^{k+1}  = x^D x^{k+1} z^{k+1} = x^{k+1} x^D z^{k+1} = x^{k} z^{k+1}  = x^D  \]
\end{description} 
So we conclude that $x^D$ is the Drazin inverse of $x$. From this it is immediete that either an object or a category, being strongly $\pi$-regular is equivalent to being Drazin as well. 
\end{proof}

We will revisit strong $\pi$-regularity in an \emph{additive} category in Section \ref{sec:additive-pi}. 


\subsection{Isomorphisms} What are the endomorphisms whose Drazin index is zero? It turns out that they precisely correspond to isomorphisms. So every isomorphism is Drazin and its Drazin inverse is precisely its inverse. Thus the notion of a Drazin inverse is a true generalization of an inverse. This implies that identity morphisms are Drazin and their own Drazin inverse. Furthermore, it provides another basic source of Drazin categories as any groupoid is immediately Drazin.

\begin{lemma}\label{lem:Drazin-0} $x: A \to A$ is Drazin with $\mathsf{ind}(x)=0$ if and only if $x$ is an isomorphism. In particular, the identity $1_A: A \to A$ is Drazin and its own Drazin inverse, $1^D_A = 1_A$. 
\end{lemma}
\begin{proof} For ($\Rightarrow$), suppose that $x$ is Drazin with $\mathsf{ind}(x)=0$. In particular, \textbf{[D.1]} can be rewritten as $xx^D = 1_A$. It follows from \textbf{[D.3]}, that we also have that $x^D x = 1_A$. Therefore $x$ is an isomorphism with inverse $x^D$. For ($\Leftarrow$), suppose that $x$ is an isomorphism. Then its inverse $x^{-1}$ satisfies the three Drazin axioms: \textbf{[D.1]}, $xx^{-1} = 1_A = x^0$; \textbf{[D.2]} $x^{-1} x x^{-1} = x^{-1}$; and \textbf{[D.3]} $xx^{-1} = 1_A = x^{-1} x$. So $x^{-1}$ is the Drazin inverse of $x$. Moreover, \textbf{[D.1]} holds for $k=0$, thus $\mathsf{ind}(x)=0$. 
\end{proof}

\subsection{Group Inverses}\label{sec:groupinv} Another special case to consider is when the Drazin index is less than or equal to $1$. In the literature, this is better known as having a \emph{group inverse} \cite[Def 7.2.4]{campbell2009generalized}. In other words, a group inverse is the Drazin inverse for endomorphisms with Drazin index less than or equal to $1$. Puystjens and Robinson described group inverses in an arbitrary category in \cite[Sec 2]{robinson1987generalized}. 

\begin{definition}\label{def:groupinv} A \textbf{group inverse} \cite[Sec 2]{robinson1987generalized} of $x: A \to A$ is an endomorphism $x^D: A \to A$ such that the following equalities hold: 
 \[ \textbf{[G.1]}~ x x^D x = x; ~~~~~\textbf{[G.2]}~ x^D x x ^D = x^D; ~~~~~\textbf{[G.3]}~ x^D x = x x^D. \]   
\end{definition}

\begin{lemma}\label{lemma:ind1} $x: A \to A$ is Drazin with $\mathsf{ind}(x) \leq 1$ if and only if $x$ has a group inverse. 
\end{lemma}
\begin{proof} For ($\Rightarrow$), suppose that $x$ is Drazin and $\mathsf{ind}(x) \leq 1$. We show that its Drazin inverse $x^D$ is a group inverse. Note that \textbf{[G.2]} and \textbf{[G.3]} are the same as \textbf{[D.2]} and \textbf{[D.3]}. Since $\mathsf{ind}(x) \leq 1$, by Lemma \ref{lemma:Drazin-basic}.(\ref{lemma:Drazin-basic.1}), we have that $x^2 x^D = x$, which using \textbf{[D.3]} can be rewritten as $x x^D x = x$, so \textbf{[G.1]} also holds. Therefore $x^D$ is a group inverse of $x$. For ($\Leftarrow$), suppose that $x$ has a group inverse $x^D$. Again, note that \textbf{[D.2]} and \textbf{[D.3]} are the same as \textbf{[G.2]} and \textbf{[G.3]}. Now using \textbf{[G.3]}, we can rewrite \textbf{[G.1]} as $x^{1+1} x^D = x$, so \textbf{[D.1]} holds for at least $k=1$. Therefore, $x$ is Drazin with Drazin inverse $x^D$ and $\mathsf{ind}(x) \leq 1$. 
\end{proof}

Since group inverses are special cases of Drazin inverses, they are unique. A special case of an endomorphism with a group inverse is an idempotent.  Later, in Lemma \ref{lemma:ind1=cbi}, we will provide an alternative description of an endomorphism with a group inverse as a commuting \emph{binary idempotent}.  

\subsection{The Drazin Inverse of a Drazin Inverse} A Drazin inverse is always Drazin, and in fact its Drazin inverse is a group inverse, meaning that the Drazin index of the Drazin inverse is less than or equal to one. The Drazin inverse of a Drazin inverse is called its \textbf{core} and will play an important role in Section \ref{sec:additive}. 

\begin{lemma} \label{lem:Drazin-inverse1} Let $x: A \to A$ be Drazin. Then: 
\begin{enumerate}[(i)]
\item \label{Drazin-inverse-inverse} \cite[Thm 3]{drazin1958pseudo} \label{repeated-Drazin-inverse.i} $x^D$ is Drazin where $x^{DD} \colon = x x^D x$ and $\mathsf{ind}(x^D) \leq 1$;
\item \label{Drazin-inverse-inverse-inverse}  \cite[Cor 4]{drazin1958pseudo} $x^{DD}$ is Drazin where $x^{DDD} = x^D$;
\item \label{cor:ginverse-ginverse} If $\mathsf{ind}(x) \leq 1$, then $x^{DD} = x$. 
\end{enumerate} 
\end{lemma}
\begin{proof} We begin by checking that $x^{DD} \colon = x x^D x$ satisfies the three group inverse axioms. 
\begin{enumerate}[{\bf [G.1]}]
\item By {\bf [D.2]} twice, we compute that: 
\[ x^D x^{DD} x^D = x^D x x^D x x^D = x^D x x^D = x^D \]
\item By {\bf [D.2]} twice, we compute that: 
\[ x^{DD} x^D x^{DD} = x x^D x x^D x x^D x = x x^D x x^D x =  x x^D x = x^{DD} \]
\item By {\bf [D.3]} we get that: 
\[ x^D  x^{DD} = x^D x x^D x  = x x^D x x^D = x^{DD} x^D \]
\end{enumerate}
So $x^{DD}$ is a group inverse of $x^D$. By Lemma \ref{lemma:ind1}, $x^D$ is Drazin with Drazin inverse $x^{DD}$ and $\mathsf{ind}(x^D) \leq 1$. Now by applying (\ref{Drazin-inverse-inverse}) to $x^{DD}$, we have that $x^{DD}$ also has a Drazin inverse given by $x^{DDD} = x^D x^{DD} x^D$. Expanding this out and applying {\bf [D.2]} twice again, we compute that: 
\[ x^{DDD} =  x^D x^{DD} x^D = x^D x x^D x x^D = x^D x x^D = x^D \]
So $x^D$ is the Drazin inverse of $x^{DD}$, or in other words, $x^{DDD} = x^D$. Lastly, if $\mathsf{ind}(x) \leq 1$, then by Lemma \ref{lemma:ind1}, $x^D$ is a group inverse of $x$. So by \textbf{[G.1]}, we have that $x^{DD} = x x^D x = x$.
\end{proof}

\subsection{Drazin inverses are absolute} We now turn our attention to other properties of Drazin inverse that are more categorically flavoured. We first observe that Drazin inverses are \emph{absolute}, that is, every functor preserves Drazin inverses on the nose. 

\begin{proposition}\label{Drazin-absolute} Let $\mathsf{F}: \mathbb{X} \to \mathbb{Y}$ be a functor and let $x: A \to A$ be Drazin in $\mathbb{X}$. Then $\mathsf{F}(x): \mathsf{F}(A) \to \mathsf{F}(A)$ is Drazin in $\mathbb{Y}$ where $\mathsf{F}(x)^D = \mathsf{F}(x^D)$ and $\mathsf{ind}\left( \mathsf{F}(x) \right) \leq \mathsf{ind}(x)$. 
\end{proposition}
\begin{proof} We show that $\mathsf{F}(x)^D \colon = \mathsf{F}(x^D)$ satisfies the three Drazin inverse axioms.  
\begin{enumerate}[{\bf [D.1]}]
\item Let $\mathsf{ind}(x)=k$. By {\bf [D.1]} for $x$, we compute that: 
\[ \mathsf{F}(x)^{k+1} \mathsf{F}(x)^D = \mathsf{F}(x^{k+1}) \mathsf{F}(x^D) = \mathsf{F}(x^{k+1} x^D) = \mathsf{F}(x^k) = \mathsf{F}(x)^k  \]
\item By {\bf [D.2]} for $x$, we compute that: 
\[ \mathsf{F}(x)^D \mathsf{F}(x) \mathsf{F}(x)^D = \mathsf{F}(x^D) \mathsf{F}(x) \mathsf{F}(x^D) = \mathsf{F}(x^D x x^D) = \mathsf{F}(x^D) = \mathsf{F}(x)^D \]
\item By {\bf [D.3]} for $x$, we have: 
\[  \mathsf{F}(x)^D \mathsf{F}(x) =  \mathsf{F}(x^D) \mathsf{F}(x) =  \mathsf{F}(x^D x) =  \mathsf{F}(x x^D) = \mathsf{F}(x) \mathsf{F}(x^D) = \mathsf{F}(x) \mathsf{F}(x)^D  \]
\end{enumerate}
So we conclude that $\mathsf{F}(x^D)$ is indeed the Drazin inverse of $\mathsf{F}(x)$. Moreover, by the calculation for \textbf{[D.1]}, we have that $\mathsf{ind}\left( \mathsf{F}(x) \right) \leq \mathsf{ind}(x)$.
\end{proof}

Being Drazin is, furthermore, a property that is preserved by categorical equivalence. To see this, we first show that conjugation by an isomorphism preserves being Drazin. 

\begin{lemma}\label{lemma:conjugate} If $x: A \to A$ is Drazin, then for any isomorphism $p: B \to A$, the composite $p x p^{-1}: B \to B$ is also Drazin where $(p x p^{-1})^D \colon = p x^D p^{-1}$. 
\end{lemma}
\begin{proof} We show that $(p x p^{-1})^D \colon = p x^D p^{-1}$ satisfies the three Drazin inverse axioms. 
\begin{enumerate}[{\bf [D.1]}]
\item Suppose that $\mathsf{ind}(x) = k$. Now note that $(p x p^{-1})^{n} = px^n p^{-1}$ for all $n\in \mathbb{N}$. So by {\bf [D.1]} for $x$, we have: 
\[ (p x p^{-1})^{k+1}(p x p^{-1})^D = px^{k+1} p^{-1} p x^D p^{-1} = px^{k+1} x^D p^{-1} = p x^k p^{-1} = (p x p^{-1})^{k} \]
\item By {\bf [D.2]} for $x$, we compute: 
\[ (p x p^{-1})^D p x p^{-1} (p x p^{-1})^D = p x^D p^{-1} p x p^{-1} p x^D p^{-1} = p x^D x x^D p^{-1} = p x^D p^{-1} = (p x p^{-1})^D  \]
\item By {\bf [D.3]} for $x$, we have: 
\begin{gather*}
     (p x p^{-1})^D p x p^{-1} \!= \! p x^D p^{-1}p x p^{-1} \! = \! p x^D x p^{-1} \!= \! p x^D x p^{-1} \! = \! p x p^{-1}  p x^D p^{-1} = p x p^{-1}(p x p^{-1})^D 
\end{gather*}
\end{enumerate}
    So we conclude that $(p x p^{-1})^D$ is indeed the Drazin inverse of $p x p^{-1}$ as desired. 
\end{proof}

\begin{corollary}\label{cor:equivalence} If $\mathbb{Y}$ is Drazin and $\mathbb{X}$ is equivalent to $\mathbb{Y}$, then $\mathbb{X}$ is Drazin.
\end{corollary}
\begin{proof} Suppose that $\mathsf{F}: \mathbb{X} \to \mathbb{Y}$ and $\mathsf{G}: \mathbb{Y} \to \mathbb{X}$ form an equivalence with natural isomorphisms $\eta_A: A \to \mathsf{G}\mathsf{F}(A)$ and $\epsilon_B: B \to \mathsf{F}\mathsf{G}(B)$. Now for every endomorphism $x: A \to A$ in $\mathbb{X}$, since $\mathbb{Y}$ is Drazin, we have that the endomorphism $\mathsf{F}(x): \mathsf{F}(A) \to \mathsf{F}(A)$ is Drazin. Then by Proposition \ref{Drazin-absolute}, we have that $\mathsf{G}\mathsf{F}(x): \mathsf{G}\mathsf{F}(A) \to \mathsf{G}\mathsf{F}(A)$ is Drazin. Then by Lemma \ref{lemma:conjugate}, we have that $\eta_A \mathsf{G}\mathsf{F}(x) \eta^{-1}_A: A \to A$ is Drazin. Of course, $x = \eta_A \mathsf{G}\mathsf{F}(x) \eta^{-1}_A$, so $x$ is Drazin, and therefore $\mathbb{X}$ is also a Drazin category. 
\end{proof}

\subsection{Commutative Squares} Drazin observed that the Drazin inverse of an endomorphism commutes with everything with which that endomorphism commutes \cite{drazin2013commuting}. Here we generalize this observation and show that the idea applies to commutative squares: 

\begin{proposition}\label{commutative-squares}
\label{Drazin-commuting}  Let $x: A \to A$ and $y: B \to B$ be Drazin. If the square on the left commutes, then the square on the right commutes: 
\[ \begin{array}[c]{c} \xymatrixcolsep{5pc}\xymatrix{A \ar[d]_-{f} \ar[r]^-{x} & A \ar[d]^-{f} \\ B \ar[r]_-{y} & B} \end{array} \Rightarrow \begin{array}[c]{c} \xymatrixcolsep{5pc}\xymatrix{A \ar[d]_-{f} \ar[r]^-{x^D} & A \ar[d]^-{f} \\ B \ar[r]_-{y^D} & B} \end{array} \]
\end{proposition}

\begin{proof} Let $k = \max\left( \mathsf{ind}(x), \mathsf{ind}(y) \right)$. By Lemma \ref{lemma:Drazin-basic}.(\ref{lemma:Drazin-basic.1}), we have that $y^{k+1} y^D = y^k$ and $x^D x^{k+1} = x^k$. Moreover, by the left square, we have that $x^k f = f y^k$. So we first compute: 
\[ x^D x^k f = x^D f y^k =  x^D f y^{k+1} y^D =  x^D x^{k+1} f y^D = x^k f y^D = f y^k y^D  \]
So $x^D x^k f = f y^k y^D$. From this, we also get that $(x^D)^{k+1} x^k f =  f y^k (y^D)^{k+1}$. Now by Lemma \ref{lemma:Drazin-basic}.(\ref{lemma:Drazin-basic.2}), recall that we also have that $(x^D)^{k+1} x^k = x^D$ and $y^k (y^D)^{k+1} = y^D$. Then we compute that: 
\[ x^D f = (x^D)^{k+1} x^k f =  f y^k (y^D)^{k+1} = f y^D \]
So the square on the right commutes as desired. 
\end{proof}

Proposition \ref{Drazin-commuting} will be quite useful for constructing new Drazin categories. 

\subsection{Iteration} A basic operation on an endomorphism is to iterate it. The iteration of an endomorphism which is Drazin is again Drazin, and the Drazin inverse of that iteration is just the iteration of the Drazin inverse.  

\begin{lemma} \cite[Thm 2]{drazin1958pseudo} \label{lem:Drazin-iteration} If $x: A \to A$ is Drazin, with $\mathsf{ind}(x)=k$, then for each $n \in \mathbb{N}$, $x^n$ is Drazin where $(x^n)^D = (x^D)^n$ and $\mathsf{ind}(x^n) \leq k$. 
\end{lemma}
\begin{proof} When $n=0$, by Lemma \ref{lem:Drazin-0}, we know that $x^0=1_A$ is Drazin with $(x^0)^D = (x^D)^0 = 1_A$ and $\mathsf{ind}(x^0) = 0$ which is less than or equal to $k$. So now assume that $n \geq 1$. We show $x^n$ has Drazin index $k$ with $(x^n)^D \colon = (x^D)^n$ by checking the three Drazin axioms:
\begin{enumerate}[{\bf [D.1]}]
\item First observe that, from \textbf{[D.1]} for $x$, we have that $x^{k+n} (x^D)^{n} = x^{k}$ and thus:
\[ (x^n)^{k + 1} (x^{n})^D = x^{nk+n} (x^D)^{n} =  x^{(n-1)k} x^{k+n} (x^D)^{n} = x^{(n-1)k} x^k = x^{nk} = (x^n)^{k}  \]
\item By {\bf [D.2]} and {\bf [D.3]} for $x$, we compute that: 
\[ (x^n)^D x^n (x^n)^D = (x^D)^n x^n (x^D)^n = \left( x^D x x^D \right)^n = (x^D)^n = (x^n)^D \]
\item By {\bf [D.3]} for $x$ we have: 
\[ (x^n)^D x^n = (x^D)^n x^n = x^n (x^D)^n  = x^n (x^n)^D \]
\end{enumerate}
So we conclude that $x^n$ is Drazin. 
\end{proof}

More surprisingly, we use Proposition \ref{commutative-squares} to show that if a non-trivial iteration of an endomorphism is Drazin then the original endomorphism must have been Drazin.

\begin{lemma} $x: A \to A$ is Drazin if and only if there is a $k \in \mathbb{N}$ such that $x^{k+1}: A \to A$ is Drazin. 
\end{lemma} 
\begin{proof} For ($\Rightarrow$), this is immediate by Lemma \ref{lem:Drazin-iteration}. For ($\Leftarrow$), suppose that for some $k \in \mathbb{N}$, $x^{k+1}$ is Drazin with $\mathsf{ind}(x^{k+1}) =p$. Define $x^D \colon = (x^{k+1})^D x^k$. We show that $x^D$ satisfies the three Drazin axioms:
\begin{enumerate}[{\bf [D.1]}]
\item Using \textbf{[D.1]} for $x^{k+1}$, we compute that: 
\[ x^{kp+k+p+1} x^D = (x^{k+1})^{p+1} (x^{k+1})^D x^k = (x^{k+1})^p x^k = x^{kp+k+p} \]
\item Using {\bf [D.2]} for $x^{k+1}$, we compute that:
\[ x^D x x^D = (x^{k+1})^D x^k x (x^{k+1})^D x^k = (x^{k+1})^D x^{k+1} (x^{k+1})^D x^k = (x^{k+1})^D x^k = x^D\]
\item Since $x x^{k+1} = x^{k+1} x$, by Proposition \ref{commutative-squares} we get that $x (x^{k+1})^D = (x^{k+1})^D x$. So we get that $x^k (x^{k+1})^D = (x^{k+1})^D x^k = x^D$. Then using {\bf [D.3]} for $x^{k+1}$, we get that:  
\[ x^D x = (x^{k+1})^D x^k x =  (x^{k+1})^D x^{k+1} = x^{k+1} (x^{k+1})^D = x x^k (x^{k+1})^D = x x^D \]
\end{enumerate}
So we conclude that $x$ is Drazin. \end{proof}


\subsection{Slice Categories}\label{sec:slice} We now turn our attention to constructing Drazin categories from other Drazin categories. As a first example, we apply Proposition \ref{Drazin-commuting} to show that slice and coslice categories of a Drazin category are again Drazin. Indeed,  for a category $\mathbb{X}$ and an object $X$, consider the slice category $\X/X$ and the underlying functor $\mathsf{U}: \X/X \to \X$. Now by Proposition \ref{Drazin-absolute}, we know that if $x: (f: A \to X) \to (f: A \to X)$ has a Drazin inverse in $\mathbb{X}/X$ then $U(x)=x: A \to A$ will have a Drazin inverse in $\mathbb{X}$ where  $\mathsf{U}(x)^D = \mathsf{U}(x^D)$.  Now $\mathsf{U}$ not only preserves Drazin inverses, $\mathsf{U}$ also reflects them. This follows immediately from the following concrete observation:

\begin{lemma} Let $x: A \to A$ be Drazin. Then if the diagram on the left commutes, then the diagram on the right commutes: 
\[ \begin{array}[c]{c} \xymatrixcolsep{5pc}\xymatrix{A \ar[dr]_-{f} \ar[rr]^-{x} && A \ar[dl]^-{f} \\ & X } \end{array} \Rightarrow \begin{array}[c]{c} \xymatrixcolsep{5pc}\xymatrix{A \ar[dr]_-{f} \ar[rr]^-{x^D} && A \ar[dl]^-{f} \\ & X} \end{array} \]
\end{lemma}
\begin{proof} Apply Proposition \ref{Drazin-commuting} by setting $y=1_X$, which recall is Drazin by Lemma \ref{lem:Drazin-0}.  
\end{proof}

\begin{corollary}
If $\X$ is a Drazin category then for every object $X \in \mathbb{X}$, the slice category $\X/X$ (respectively, the coslice $X/\mathbb{X}$) is Drazin. 
\end{corollary} 

More generally it is not hard to see that the same is true for comma categories. 

\begin{corollary} For functors $\mathsf{F}: \mathbb{Y} \to \mathbb{X}$ and $\mathsf{G}: \mathbb{Z} \to \mathbb{X}$, if $\mathbb{Y}$ and $\mathbb{Z}$ are Drazin, then the comma category $\mathsf{F}/\mathsf{G}$ is Drazin. 
\end{corollary}

\begin{proof}
An endomorphism on $a$ in $\mathsf{F}/\mathsf{G}$ is a pair of maps $f:X \to X$ and $g: Y \to Y$ such that 
the left diagram below commutes.  As $F(f^D) = F(f)^D$ and $G(g^D) = G(g)^D$ then by Proposition \ref{Drazin-commuting} we have:
\[ \xymatrix{X \ar[d]_f & F(X) \ar[d]_{F(f)} \ar[r]^a & G(Y) \ar[d]^{G(y)} & Y \ar[d]^{g~~~~~~\Rightarrow} 
       & & X \ar[d]_{f^D} & F(X) \ar[d]_{F(f)^D} \ar[r]^a & G(Y) \ar[d]^{G(y)^D} & Y \ar[d]^{g^D}\\ 
             X & F(X) \ar[r]_a & G(Y) & Y & & X & F(X) \ar[r]_a & G(Y) & Y} \]
\end{proof}

\subsection{Finitely Specified Structures} We can also use Proposition \ref{Drazin-commuting} to show that for any Drazin category, finitely specified structures and their morphisms will also form a Drazin category. 

\begin{lemma} When $\mathbb{X}$ is a Drazin category, if $(\mathbb{S},{\cal F},{\cal L})$ is a sketch which involves only finitely many objects, then the category of models ${\sf Mod}((\mathbb{S},{\cal F},{\cal L}),\X)$ is Drazin.
\end{lemma}
\begin{proof}
The Drazin inverse of a natural transformation (morphism of models) must be taken pointwise. Thus Proposition \ref{Drazin-commuting} assures us that this will result in a morphism of the models. In order to have an index which works globally for all the objects of ${\cal F}$, we must ensure that there is an index which works at each component of the natural transformation. To obtain this we have to take the maximum of the indices at each component, however, this is only guaranteed to be defined if there are only finitely many objects -- hence the caveat.
\end{proof}

As such, the category of {\em any\/} finitely specified structure over a Drazin category is itself Drazin. For example, the arrow category of a Drazin category is Drazin, and the category of directed graphs and graph morphisms of a Drazin category is Drazin.  Many algebraic structures have a finite specification so this provides many examples: thus, monoids and rings (with respect to products) in a Drazin category form a Drazin category. 

\subsection{Algebras of Endofunctors}\label{sec:algebras} Another interesting application of Proposition \ref{Drazin-commuting} is that an endomorphism of an algebra of an endofunctor has a Drazin inverse whenever the underlying endomorphism has a Drazin inverse (and similarly for coalgebras). Recall that for an endofunctor ${\mathsf{T}: \mathbb{X} \to \mathbb{X}}$, a $\mathsf{T}$-algebra is a pair $(A, \nu)$ consisting of an object $A \in \mathbb{X}$ and a map $\nu: \mathsf{T}(A) \to A$, and that a $\mathsf{T}$-algebra morphism $f: (A, \nu) \to (B, \nu^\prime)$ is a map $f: A \to A$ such that $\mathsf{T}(f) \nu^\prime = \nu f$. 

\begin{lemma} Let $x: (A, \nu) \to (A, \nu)$ be a $\mathsf{T}$-algebra endomorphism. If $x: A \to A$ is Drazin, then $x^D: (A, \nu) \to (A, \nu)$ is a $\mathsf{T}$-algebra endomorphism. 
\end{lemma}
\begin{proof} Suppose that $x: A \to A$ is Drazin. By Proposition \ref{Drazin-absolute}, $\mathsf{T}(x): \mathsf{T}(A) \to \mathsf{T}(A)$ is also Drazin where $\mathsf{T}(x)^D = \mathsf{T}(x^D)$. By definition of being a $\mathsf{T}$-algebra endomorphism, the diagram below on the left commutes. Then by Proposition \ref{Drazin-commuting}, the diagram on the right commutes: 
\[ \xymatrixcolsep{5pc}\xymatrix{\mathsf{T}(A) \ar[d]_-{\nu} \ar[r]^-{\mathsf{T}(x)} & \mathsf{T}(A) \ar[d]^-{\nu} \\ A \ar[r]_-x & A} ~~\xymatrixcolsep{5pc}\xymatrix{~ \ar@{}[d]|{\Rightarrow} \\ ~} ~~  \xymatrixcolsep{5pc}\xymatrix{\mathsf{T}(A) \ar[d]_-{\nu} \ar[r]^-{\mathsf{T}(x)^D = \mathsf{T}(x^D)} & \mathsf{T}(A) \ar[d]^-{\nu} \\ A \ar[r]_-{x^D} & A} \]
However, the diagram on the right is precisely the statement that $x^D: (A, \nu) \to (A, \nu)$ is a $\mathsf{T}$-algebra endomorphism. 
\end{proof}

Therefore, it follows that the category of (co)algebras of any endofunctor on a Drazin category is again Drazin. For an endofunctor $\mathsf{T}: \mathbb{X} \to \mathbb{X}$, we denote $\mathsf{ALG}(\mathsf{T})$ (resp. $\mathsf{COALG}(\mathsf{T})$) to be the category of $\mathsf{T}$-(co)algebras and $\mathsf{T}$-(co)algebra morphisms between them. 

\begin{corollary}If $\mathbb{X}$ is Drazin, then $\mathsf{ALG}(\mathsf{T})$ and $\mathsf{COALG}(\mathsf{T})$ are Drazin categories. 
\end{corollary}

It follows that the (co)Eilenberg-Moore category of a (co)monad on a Drazin category is also a Drazin category. 

\subsection{Chu Construction} \label{chu-construction} Surprisingly, another important construction that behaves well with Drazin inverses is the \emph{Chu construction}, which is a construction for building star-autonomous categories from monoidal closed categories \cite{Barr1996}. We will show that if a monoidal category is Drazin, then its Chu construction is Drazin as well. In particular, we can apply the Chu construction to the category of finite sets or the category of finite-dimensional vector spaces to yield examples of Drazin star-autonomous categories. 

For a monoidal category $\mathbb{X}$, we denote the monoidal product by $\otimes$ and the monoidal unit by $I$. For the rest of this section we work in a monoidal category $\mathbb{X}$. We first observe that the monoidal product preserves Drazin inverses: 

\begin{lemma} If $x: A \to A$ and $y: B \to B$ are Drazin, then $x \otimes y: A \otimes B \to A \otimes B$ is Drazin where $(x \otimes y)^D = x^D \otimes y^D$. 
\end{lemma}
\begin{proof} It is straightforward to see that if $x$ and $y$ are Drazin in $\mathbb{X}$, then $(x,y)$ is also Drazin in $\mathbb{X} \times \mathbb{X}$ with $(x,y)^D = (x^D, y^D)$. Then since the monoidal product is a functor $\otimes: \mathbb{X} \times \mathbb{X} \to \mathbb{X}$, by Proposition \ref{Drazin-absolute} we get that $\otimes(x,y) = x \otimes y$ is Drazin with Drazin inverse $\otimes(x,y)^D = \otimes(x^D, y^D)$, which is precisely that $(x \otimes y)^D = x^D \otimes y^D$.
\end{proof}

Now for each object $Z \in \mathbb{X}$, define ${\sf Chu}(\X,Z)$, to be the category whose objects are triples $(X, Y, a)$ consisting of objects $X$ and $Y$ and a map $a: X \otimes Y \to Z$, and whose maps are pairs $(f,g): (X, Y, a) \to (X^\prime, Y^\prime, a^\prime)$ consisting of maps $f: X \to X^\prime$ and $g: Y^\prime \to Y$ such that $(f \otimes 1_Y) a = (1_A \otimes g) a^\prime$. Identities are $(1_X, 1_Y): (X, Y, a) \to (X, Y, a)$, and composition is $(f,g)(h,k) = (fh, kg)$. It is well known that if $\X$ is a monoidal closed category with finite limits that ${\sf Chu}(\X,Z)$ is a star-autonomous category. 

\begin{proposition} $(x,y): (X, Y, a) \to (X, Y, a)$ in ${\sf Chu}(\X,Z)$ is Drazin  whenever $x$ and $y$ are Drazin in $\mathbb{X}$.
\end{proposition}
\begin{proof}
Let $(x,y): (X, Y, a) \to (X, Y, a)$ be an endomorphism in ${\sf Chu}(\X,Z)$, so we have $(x \otimes 1_Y) a= (1_X \otimes y)a$. Also, suppose that $x$ and $y$ are Drazin in $\mathbb{X}$. We first show that $(x^D,y^D): (X, Y, a) \to (X, Y, a)$ is also an endomorphism in the Chu construction. So let $j = \mathsf{max}\left( \mathsf{ind}(x), \mathsf{ind}(y) \right)$. Using Lemma \ref{lemma:Drazin-basic} we compute that:  
\begin{gather*}
   (x^D \otimes 1_Y) a  = ((x^D)^{j+1} \otimes 1_Y)(x^j \otimes 1_Y) a = ((x^D)^{j+1} \otimes 1_Y) (1_X \otimes y^j) a \\ 
   = \! (\!(x^D)^{j+1} \otimes\! 1_Y) (1_X \otimes \! (y^D)^{j+1}) (1_X \otimes y^{j+j+1}) a  \! = \!(\!(x^D)^{j+1} \otimes 1_Y) (x^{j+j+1} \otimes 1_Y)  (1_X \otimes (y^D)^{j+1}) a \\
   = (x^j \otimes 1_Y) (1_X \otimes (y^D)^{j+1}) a = (1_X \otimes y^j) (1_X \otimes (y^D)^{j+1}) a = (1_X \otimes y^D) a
\end{gather*}
So $(x^D \otimes 1_Y) a = (1_X \otimes y^D) a$, and therefore $(x^D,y^D): (X, Y, a) \to (X, Y, a)$ is indeed a map in ${\sf Chu}(\X,Z)$. From here, by the definition of composition in ${\sf Chu}(\X,Z)$, it is straightforward to see that $(x^D,y^D)$ is a Drazin inverse of $(x,y)$. 
\end{proof}

\begin{corollary}
If $\X$ is a monoidal category which is Drazin then ${\sf Chu}(\X,Z)$ is Drazin.
\end{corollary}


\section{Rank} \label{rank}


For a complex square matrix $A$, an intuitive way of finding its Drazin inverse is to iterate $A$ until the rank does not change (which is always guaranteed to happen), that is, $\mathsf{rank}(A^{k}) = \mathsf{rank}(A^{k+1}) = \mathsf{rank}(A^{k+2}) = \hdots$. When this happens, one can reverse any later iterations and thus build a Drazin inverse. The same principle holds true for linear endomorphisms on a finite-dimensional vector space or endomorphisms on a finite set. The purpose of this section is to make this procedure rigorous.  This involves, in particular, making precise what is meant by ``rank''. As such, we introduce the concept of \textbf{expressive rank} for an arbitrary category and show that any category with expressive rank is Drazin. 

\subsection{Expressive Rank}\label{sec:expressive} In linear algebra, the rank of a matrix or a linear transformation is the dimension of its image space. We wish to generalize this in a category by associating every map to a natural number which represents its rank. We express this in terms of a \emph{colax} functor into a 2-category which we call $\mathsf{Rank}$\footnote{Steve Lack pointed out to us that $\mathsf{Rank}$ is precisely the span category for the natural numbers.}. To help with notation, for $n,m \in \mathbb{N}$ we denote $n \wedge m = \mathsf{min}(n,m)$.  $\mathsf{Rank}$ is the 2-category defined as follows:
\begin{description}
\item{{\bf [0-cells]}:}  $n \in \mathbb{N}$;
\item{{\bf [1-cells]}:}  $m : n_1 \to n_2$ where $m \in \mathbb{N}$ with $m \leq n_1\wedge n_2$,  the identity on $n$ is $n: n \to n$, and composition of $m_1: n_1 \to n_2$ and $m_2: n_2 \to n_3$ is $m_1 \wedge m_2: n_1 \to n_3$; 
\item{{\bf [2-cells]}:}  $m_1 \Rightarrow m_2$ if and only if $m_1 \leq m_2$.
\end{description}
For a category $\mathbb{X}$, by a colax functor $\mathsf{rank}: \mathbb{X} \to \mathsf{Rank}$, we mean a mapping which associates objects $A$ of $\mathbb{X}$ to $0$-cells $\mathsf{rank}(A) \in \mathbb{N}$, maps $f: A \to B$ of $\mathbb{X}$ to $1$-cells $\mathsf{rank}(f): \mathsf{rank}(A)\to \mathsf{rank}(B)$, so in particular $\mathsf{rank}(f) \leq \mathsf{rank}(A)$ and $\mathsf{rank}(f) \leq \mathsf{rank}(B)$. We also ask that a colax functor preserves identities, $\mathsf{rank}(1_A) = \mathsf{rank}(A)$, while for composition we only require that $\mathsf{rank}(fg) \leq \mathsf{rank}(f) \wedge \mathsf{rank}(g)$. For objects, we think of $\mathsf{rank}(A)$ as the dimension of $A$, while for maps we think of $\mathsf{rank}(f)$ as the rank of $f$. 

The notion of expressive rank is defined on a category with a \emph{factorization system}\footnote{In this paper, by a factorization system we mean an \emph{orthogonal} factorization system.}, and we ask that the colax functor into $\mathsf{Rank}$ be compatible with the factorization system as well as reflects isomorphisms. To this end, we note that the only 1-cell isomorphisms in $\mathsf{Rank}$ are the identity 1-cells $n: n \to n$. Moreover, given any 1-cell $r: n \to m$ we may factorize the map as $n \to^r r \to^r m$. Thus setting $\mathcal{E} = \lbrace r: n \to r \vert~ \forall n,r \in \mathbb{N}, r \leq n \rbrace$ and $\mathcal{M} = \lbrace r: r \to n \vert~ \forall n,r\in \mathbb{N}, r \leq n \rbrace$, we get a factorization system on the $1$-cells. 

\begin{definition}\label{def:expressive} A category $\mathbb{X}$ is said to have \textbf{expressive rank} if: 
\begin{enumerate}[{\bf [ER.1]}]
\item $\mathbb{X}$ comes equipped with a colax functor $\mathsf{rank}: \X \to \mathsf{Rank}$;
\item $\mathbb{X}$ has a factorization system $(\mathcal{E}, \mathcal{M})$ which {\bf expresses rank}, that is, for every map $f: A \to B$, if the diagram on the left is a factorization of $f$ with $\varepsilon_f \in \mathcal{E}$ and $m_f \in \mathcal{M}$, 
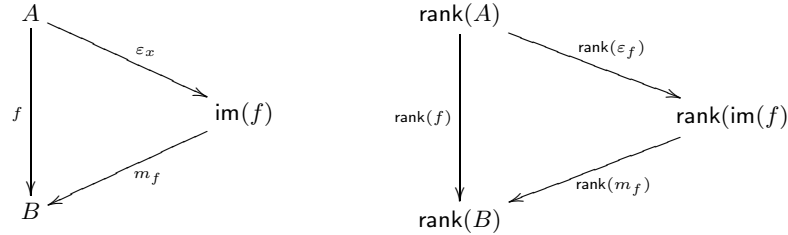
\[ \xymatrixcolsep{5pc}\xymatrix{ A \ar[dd]_f \ar[dr]^{\varepsilon_x} \\ & {\sf im}(f) \ar[dl]^{m_f} \\ B} ~~~~~~~~~~~~
   \xymatrixcolsep{5pc}\xymatrix{ \mathsf{rank}(A) \ar[dd]_{\mathsf{rank}(f)} \ar[dr]^{\mathsf{rank}(\varepsilon_f)} \\ & \mathsf{rank}({\sf im}(f) \ar[dl]^{\mathsf{rank}(m_f)} \\ \mathsf{rank}(B) } \]
then $\mathsf{rank}(f) = \mathsf{rank}(\varepsilon_f)= \mathsf{rank}(m_f) =  \mathsf{rank}({\sf im}(f))$.
\item $\mathsf{rank}$ reflects isomorphisms, that is, if $f: A \to B$ and $\mathsf{rank}(f) = \mathsf{rank}(A)= \mathsf{rank}(B)$, then $f$ is an isomorphism.
\end{enumerate}
In a category $\mathbb{X}$ with expressive rank, we call $\mathsf{rank}(A)$ the \textbf{dimension} of an object $A$, and we call $\mathsf{rank}(f)$ the \textbf{rank} of a map $f$. 
\end{definition} 

Here are now some examples of expressive rank. 


\subsection{Matrices}\label{ex:matrix-rank} For a field $k$, $\mathsf{MAT}(k)$ has expressive rank, where the factorization system is the usual surjection-injection factorization system, $\mathsf{rank}(n) =n$, and for a matrix $A$, $\mathsf{rank}(A)$ is the usual rank of the matrix. 

\subsection{Finite-Dimensional Vector Spaces}\label{ex:vec-rank} Slightly generalizing the above example, for a field $k$, $k\text{-}\mathsf{FVEC}$ also has expressive rank, where the factorization system is the usual surjection-injection factorization system, for a finite-dimensional vector space $V$, $\mathsf{rank}(V) = \mathsf{dim}(V)$, and for a linear transformation $f$, $\mathsf{rank}(f) = \mathsf{dim}(\mathsf{im}(f))$, which recall is called the rank of $f$.


\subsection{Finite Sets} \label{finite-sets-drazin} $\mathsf{FinSet}$ has expressive rank, where the factorization system is the usual surjection-injection factorization system, for a finite set $X$, $\mathsf{rank}(X) = \vert X \vert$, and for a function $f$, $\mathsf{rank}(f) = \vert \mathsf{im}(f) \vert$. 


\subsection{Finite Length Modules}\label{ex:module-rank} Let $R$ be a ring. Recall that for an $R$-module $M$, a \emph{composition series} of $M$ is a finite chain of submodules $0 = M_0 \subseteq M_1 \subseteq M_2 \subseteq \hdots \subseteq M_n = M$ such that the quotients $M_{i+1}/M_i$ are simple. Then $M$ is said to be of \textbf{finite length} $n$ if it has a composition series of length $n$, and we call $\mathsf{length}(M) = n$ the length of $M$. The \emph{Jordan-H\"older theorem} \cite[Page 108]{jacobson1985basic} says that this length of decompositions and the factors are unique. Equivalently, $M$ has finite length if and only if $M$ is Artinian and Noetherian \cite[Theorem 3.5]{jacobson1985basic}. So let $R\text{-}\mathsf{FLMOD}$ be the category of $R$-modules of finite length and $R$-linear morphisms between them. Then $R\text{-}\mathsf{FLMOD}$ has expressive rank, where the factorization system is again the usual surjection-injection factorization system, for a finite length $R$-module $M$, $\mathsf{rank}(M)=\mathsf{length}(M)$, and for an $R$-linear morphism $f$ between finite length $R$-modules, $\mathsf{rank}(f)=\mathsf{length}(\mathsf{im}(f))$. It is clear that this gives a colax functor $\mathsf{rank}: R\text{-}\mathsf{FLMOD} \to \mathsf{Rank}$ and that the factorization system expresses rank, but it is not immediately clear why it reflects isomorphisms. So here is an explanation for why  $\mathsf{rank}$  indeed reflects isomorphisms. It is well-known that if $0 \to M_1 \to M_2 \to M_3 \to 0$ is a short exact sequence of $R$-modules, then $M_2$ has finite length if and only if $M_1$ and $M_3$ have finite length, and in this case $\mathsf{length}(M_2) = \mathsf{length}(M_1) + \mathsf{length}(M_3)$. So in particular if $f: M \to N$ is a $R$-linear morphism between finite length $R$-modules, then since $0 \to M \to N \to N/{\sf im}(f) \to 0$ is a short-exact sequence and $N$ has finite length, as do $\mathsf{im}(f)$ and $N/{\sf im}(f)$, and moreover $\mathsf{length}(N) = \mathsf{length}(\mathsf{im}(f)) + \mathsf{length}(N/\mathsf{im}(f))$. Now suppose that $\mathsf{rank}(f)$ is an isomorphism. This means that $\mathsf{rank}(f)=\mathsf{rank}(M)=\mathsf{rank}(N)$, so $M$, $N$, and $\mathsf{im}(f)$ all have the same length. However if $\mathsf{length}(N) = \mathsf{length}(\mathsf{im}(f))$, then we must have that $\mathsf{length}(N/{\sf im}(f)) =0$. The only $R$-module of length zero is $0$, so $N/{\sf im}(f) = 0$. This gives us that $N = {\sf im}(f)$, and so $f$ is surjective. Finally, when $f$ is surjective we have, similarly, $\mathsf{length}(\mathsf{ker}(f)) + \mathsf{length}(M) =\mathsf{length}(N)$ so the kernel is $0$ and the map is also injective and so is an isomorphism. Thus, $\mathsf{rank}$ reflects isomorphisms.


\subsection{From Expressive Rank to Drazin} We now prove the main result of this section. 

\begin{theorem} \label{rank-theorem} A category which has an expressive rank is Drazin.
\end{theorem} 
\begin{proof} Let $\mathbb{X}$ be a category with expressive rank, so with colax functor $\mathsf{rank}: \mathbb{X} \to {\sf Rank}$ and factorization system $(\mathcal{E},\mathcal{M})$. We will first show that for every endomorphism $x: A \to A$, there is a $k \in \mathbb{N}$ such that the rank of $x^k$ is equal to $x^{k+1}$. First observe that by colaxity, for any suitable maps we get the inequalities: $\mathsf{rank}(g) \geq \mathsf{rank}(f) \wedge \mathsf{rank}(g) \wedge \mathsf{rank}(h) \geq \mathsf{rank(fgh)}$. As such, for any endomorphism $x: A \to A$ we get a descending chain of inequalities: $ \mathsf{rank}(A) = \mathsf{rank}(x^0) \geq \mathsf{rank}(x^1) \geq  \mathsf{rank}(x^2) \geq \hdots$. As the natural numbers are well-ordered there is an $r \in \mathbb{N}$ which is the minimum of all these ranks, so $r = \mathsf{min}_{n\in \mathbb{N}}( \mathsf{rank}(x^n) )$. So set $k$ to be the least natural number such that $\mathsf{rank}(x^k) = r$. Once the sequence hits this rank, all subsequent ranks are equal, so $r = \mathsf{rank}(x^k)=\mathsf{rank}(x^{k+1}) = \hdots$. 

Now consider a factorization of $x^k: A \to A$ via $\varepsilon_{x^k}: A \to {\sf im}(x^k)$ and $m_{x^k}: A \to {\sf im}(x^k)$, so $x^k = \varepsilon_{x^k} m_{x^k}$ with $\varepsilon_{x^k} \in \mathcal{E}$ and $m_{x^k} \in \mathcal{M}$. In order to define the Drazin inverse, we will first show that the composite $m_{x^k} \varepsilon_{x^k}$ is an isomorphism, and so is the unique map $\gamma_{x^k}: {\sf im}(x^k) \to {\sf im}(x^k)$ induced by the factorization system which makes the following diagram commute: 
\[ \xymatrixcolsep{5pc}\xymatrix{X \ar@/_2pc/[dd]_{x^k}  \ar[r]^-x \ar[d]^-{\varepsilon_{x^k}} & X \ar[d]_-{\varepsilon_{x^k}} \ar@/^2pc/[dd]^-{x^k}  \\
                {\sf im}(x^k) \ar[d]^-{m_{x^k}} \ar@{-->}[r]^{\gamma_{x^k}} &  {\sf im}(x^k) \ar[d]_-{m_{x^k}} \\
                X \ar[r]_-x & X} \]
 Now since $x^{k+1} = \varepsilon_{x^k} \gamma_{x^k} m_{x^k}$, by colaxity and \textbf{[ER.2]}, we get:  
             \begin{gather*}
               \mathsf{rank}(x^k) \! = \! \mathsf{rank}({\sf im}(x^k)) \!\geq\!  \mathsf{rank}(\gamma_{x^k}) \!\geq\! \mathsf{rank}(\varepsilon_{x^k} \gamma_{x^k} m_{x^k}) \! =\! \mathsf{rank}(x^{k+1})\! = \! \mathsf{rank}(x^k) \! = \!\mathsf{rank}({\sf im}(x^k)\!) 
             \end{gather*}   
So we get that $\mathsf{rank}(\gamma_{x^k}) = \mathsf{rank}({\sf im}(x^k))$. Then since $\mathsf{rank}$ reflects isomorphism, by \textbf{[ER.3]} we get that $\gamma_{x^k}$ is an isomorphism. Now note that by transposing the above diagram we also get that the following diagram commutes as well: 
 \[ \xymatrixcolsep{5pc}\xymatrix{ {\sf im}(x^k)  \ar[d]_-{m_{x^k}} \ar[r]^-{\gamma_{x^k}} &  {\sf im}(x^k) \ar[d]^-{m_{x^k}}   \\
                X \ar[d]_-{\varepsilon_{x^k}} \ar[r]_-x & X \ar[d]^-{\varepsilon_{x^k}} \\
                {\sf im}(x^k) \ar[r]_-{\gamma_{x^k}} &  {\sf im}(x^k)} \]
Moreover, by definition of $\gamma_{x^k}$ we also get that: 
\[ \varepsilon_{x^k} \gamma_{x^k}^k m_{x^k} = x^k \varepsilon_{x^k} m_{x^k} = \varepsilon_{x^k} m_{x^k} \varepsilon_{x^k} m_{x^k} = x^k x^k = x^{2k} \]
So $\varepsilon_{x^k} \gamma_{x^k}^k m_{x^k} = x^{2k}$. Also by \textbf{[ER.2]}, since $\mathsf{rank}({\sf im}(x^k)) = \mathsf{rank}(m_{x^k}) = \mathsf{rank}(\varepsilon_{x^k})$, we of course have that $\mathsf{rank}({\sf im}(x^k)) = \mathsf{rank}(m_{x^k}) \wedge \mathsf{rank}(\varepsilon_{x^k})$. Then using colaxity and  $\mathsf{rank}(x^k)=\mathsf{rank}(x^{2k})$, we have that: 
\begin{gather*}\mathsf{rank}({\sf im}(x^k)) = \mathsf{rank}(m_{x^k}) \wedge \mathsf{rank}(\varepsilon_{x^k}) \geq \mathsf{rank}(m_{x^k} \varepsilon_{x^k}) \geq \mathsf{rank}(\varepsilon_{x^k} m_{x^k} \varepsilon_{x^k} m_{x^k}) \\
= \mathsf{rank}(x^{2k}) =  \mathsf{rank}(x^k) = \mathsf{rank}({\sf im}(x^k)) 
\end{gather*} 
Thus $\mathsf{rank}({\sf im}(x^k)) = \mathsf{rank}(m_{x^k} \varepsilon_{x^k})$, and by \textbf{[ER.3]}, we obtain that $m_{x^k} \varepsilon_{x^k}$ is also an isomorphism. Moreover, we also compute that:
\begin{gather*}
    \gamma^k_{x^k} = \gamma^k_{x^k} m_{x^k} \varepsilon_{x^k} (m_{x^k} \varepsilon_{x^k})^{-1} =  m_{x^k} x^k \varepsilon_{x^k} (m_{x^k} \varepsilon_{x^k})^{-1} = m_{x^k} \varepsilon_{x^k} m_{x^k} \varepsilon_{x^k} (m_{x^k} \varepsilon_{x^k})^{-1} =  m_{x^k} \varepsilon_{x^k}
\end{gather*}
So $\gamma_{x^k}^k = m_{x^k} \varepsilon_{x^k}$. So in particular, $\gamma_{x^k} m_{x^k} \varepsilon_{x^k} = \gamma_{x^k}^{k+1}$.  

Finally, setting $x^D \colon = \varepsilon_{x^k} (\gamma_{x^k}^{-1})^{k+1} m_{x^k}$, we show that $x^D$ satisfies the Drazin axioms:
\begin{enumerate}[{\bf [D.1]}]
\item Using $\gamma_{x^k} m_{x^k} \varepsilon_{x^k} = \gamma_{x^k}^{k+1}$, we compute that: 
\begin{gather*}
    x^{k+1} x^D = \varepsilon_{x^k} \gamma_{x^k} m_{x^k} \varepsilon_{x^k} (\gamma_{x^k}^{-1})^{k+1} m_{x^k} = \varepsilon_{x^k} \gamma_{x^k}^{k+1} (\gamma_{x^k}^{-1})^{k+1} m_{x^k} = \varepsilon_{x^k} m_{x^k} = x^k 
\end{gather*}
\item Again using $\gamma_{x^k} m_{x^k} \varepsilon_{x^k} = \gamma_{x^k}^{k+1}$, we compute that: 
\begin{gather*}
    x^D x x^D = \varepsilon_{x^k} (\gamma_{x^k}^{-1})^{k+1} m_{x^k} x \varepsilon_{x^k} (\gamma_{x^k}^{-1})^{k+1} m_{x^k} = \varepsilon_{x^k} (\gamma_{x^k}^{-1})^{k+1} \gamma_{x^k} m_{x^k} \varepsilon_{x^k} (\gamma_{x^k}^{-1})^{k+1} m_{x^k} \\
    = \varepsilon_{x^k} (\gamma_{x^k}^{-1})^{k+1} \gamma_{x^k}^{k+1} (\gamma_{x^k}^{-1})^{k+1} m_{x^k} = \varepsilon_{x^k} (\gamma_{x^k}^{-1})^{k+1} m_{x^k} = x^D 
\end{gather*}
\item Here we use that $\gamma_{x^k} m_{x^k} \varepsilon_{x^k} = \gamma_{x^k}^{k+1} = m_{x^k} \varepsilon_{x^k} \gamma_{x^k}$, so we compute that: 
\begin{gather*}
    x^D x = \varepsilon_{x^k} (\gamma_{x^k}^{-1})^{k+1} m_{x^k} x = \varepsilon_{x^k} (\gamma_{x^k}^{-1})^{k+1} \gamma_{x^k} m_{x^k}  = \varepsilon_{x^k} (\gamma_{x^k}^{-1})^{k} m_{x^k} \\
    = \varepsilon_{x^k} \gamma_{x^k} (\gamma_{x^k}^{-1})^{k+1} m_{x^k} = x \varepsilon_{x^k} (\gamma_{x^k}^{-1})^{k+1} m_{x^k} = x x^D 
\end{gather*}
\end{enumerate} 
Thus $x$ is Drazin with Drazin inverse $x^D$, and so we conclude that $\mathbb{X}$ is Drazin. 
\end{proof}


\subsection{Drazin Inverses in Proper Factorization Systems} The proof of Theorem \ref{rank-theorem} suggests that having a factorization system allows for an equivalent description of an endomorphism $x$ being Drazin in terms of $\gamma_{x^k}$ and $m_{x^k} \varepsilon_{x^k}$ being isomorphisms. We show that this is true for a \emph{proper} factorization system, that is a factorization system in which all the ${\cal E}$-maps are epic and all the ${\cal M}$-maps are monic.

\begin{proposition}\label{factorization-drazin}
In a category $\mathbb{X}$ with a proper factorization system $(\mathcal{M}, \mathcal{E})$, an endomorphism ${x: A \to A}$ is Drazin if and only if there is a $k \in \mathbb{N}$ for which $x^{k+1}: A \to A$ has a factorization via $\varepsilon_{x^{k+1}}: A \to {\sf im}(x^{k+1}) \in \mathcal{E}$ and $m_{x^{k+1}}: {\sf im}(x^{k+1}) \to A \in \mathcal{M}$ for which $m_{x^{k+1}} \varepsilon_{x^{k+1}}: {\sf im}(x^{k+1}) \to {\sf im}(x^{k+1})$ is an isomorphism.
\end{proposition}
\begin{proof}~For ($\Rightarrow$), Suppose $x: A \to A$ is Drazin and let $\mathsf{ind}(x) = k$. Then for $x^{k+1}: A \to A$, let $x^{k+1}$ factor via $\varepsilon_{x^{k+1}}: A \to {\sf im}(x^{k+1})$ and $m_{x^{k+1}}: A \to {\sf im}(x^{k+1})$, so $x^{k+1} = \varepsilon_{x^{k+1}} m_{x^{k+1}}$ with $\varepsilon_{x^{k+1}} \in \mathcal{E}$ and $m_{x^{k+1}} \in \mathcal{M}$. Denote $\beta_{x^{k+1}} \colon = m_{x^{k+1}} \varepsilon_{x^{k+1}}$. Now observe that by the factorization system, $\beta_{x^{k+1}}$ is the unique map which makes the below diagram on the left commute. On the other hand by \textbf{[D.3]}, we have that $(x^D)^k x^k = x^k (x^D)^k$. Therefore, define $\beta_{x^{k+1}}^{-1}$ as the unique map induced by the factorization system which makes the below diagram on the right commute:
\[  \begin{array}[c]{c} \xymatrixcolsep{3pc}\xymatrix{A \ar@/_2.5pc/[dd]_{x^{k+1}}  \ar[r]^-{x^{k+1}} \ar[d]^-{\varepsilon_{x^{k+1}}} & A \ar[d]_-{\varepsilon_{x^{k+1}}} \ar@/^2.5pc/[dd]^-{x^{k+1}}  \\
                {\sf im}(x^{k+1}) \ar[d]^-{m_{x^{k+1}}} \ar@{-->}[r]^{\beta_{x^{k+1}}} &  {\sf im}(x^{k+1}) \ar[d]_-{m_{x^{k+1}}} \\
                A \ar[r]_-{x^{k+1}} & A} \end{array} \quad \begin{array}[c]{c} \xymatrixcolsep{3pc}\xymatrix{A \ar@/_2.5pc/[dd]_{x^{k+1}}  \ar[r]^-{(x^D)^{k+1}} \ar[d]^-{\varepsilon_{x^{k+1}}} & A \ar[d]_-{\varepsilon_{x^{k+1}}} \ar@/^2.5pc/[dd]^-{x^{k+1}}  \\
                {\sf im}(x^{k+1}) \ar[d]^-{m_{x^{k+1}}} \ar@{-->}[r]^{\beta^{-1}_{x^{k+1}}} &  {\sf im}(x^{k+1}) \ar[d]_-{m_{x^{k+1}}} \\
                A \ar[r]_-{(x^D)^{k+1}} & A} \end{array} \]
Now by Lemma \ref{lemma:Drazin-basic}.(\ref{lemma:Drazin-basic.2}), we compute that: 
\begin{gather*}
 \varepsilon_{x^{k+1}} \beta_{x^{k+1}}\beta^{-1}_{x^{k+1}} m_{x^{k+1}} = x^{k+1} \varepsilon_{x^{k+1}} m_{x^{k+1}} (x^D)^{k+1} \\
 = x^{k+1} x^{k+1} (x^D)^{k+1} = x^{k+1+k+1} (x^D)^{k+1} = x^{k+1} = \varepsilon_{x^{k+1}} m_{x^{k+1}} 
\end{gather*}
Since $\varepsilon_{x^{k+1}}$ is epic and $m_{x^{k+1}}$ is monic, we get that $\beta_{x^{k+1}}\beta^{-1}_{x^{k+1}} = 1_{{\sf im}(x^{k+1})}$. Similarly, we can also show that $\beta^{-1}_{x^{k+1}}\beta_{x^{k+1}} = 1_{{\sf im}(x^{k+1})}$ as well. So we conclude that $\beta_{x^{k+1}} \colon = m_{x^{k+1}} \varepsilon_{x^{k+1}}$ is an isomorphism as desired. 

For ($\Leftarrow$), suppose that there is a $k \in \mathbb{N}$ such that for $x^{k+1}: A \to A$, and a factorization $\varepsilon_{x^{k+1}}: A \to {\sf im}(x^{k+1}) \in \mathcal{E}$ and $m_{x^{k+1}}: A \to {\sf im}(x^{k+1}) \in \mathcal{M}$, that $\beta_{x^{k+1}} \colon = m_{x^{k+1}} \varepsilon_{x^{k+1}}$ is an isomorphism. Now consider the unique map induced by the factorization system $\gamma_{x^{k+1}}: {\sf im}(x^{k+1}) \to {\sf im}(x^{k+1})$ defined as in the proof of Theorem \ref{rank-theorem}. Recall that in that proof we needed $\gamma_{x^{k+1}}$ to be an isomorphism to build the Drazin inverse, and we achieved this with an argument involving expressive rank. However, we will now show that from $\beta_{x^{k+1}}$ being an isomorphism, it follows that $\gamma_{x^{k+1}}$ is also an isomorphism. So by definition of $\gamma_{x^{k+1}}$, we first compute that: 
\begin{gather*}
    \varepsilon_{x^{k+1}} \gamma^{k+1}_{x^{k+1}} m_{x^{k+1}} = x^{k+1} \varepsilon_{x^{k+1}}  m_{x^{k+1}} = \varepsilon_{x^{k+1}}  m_{x^{k+1}} \varepsilon_{x^{k+1}}  m_{x^{k+1}} =  \varepsilon_{x^{k+1}}  \beta_{x^{k+1}} m_{x^{k+1}}  
\end{gather*}
Since $\varepsilon_{x^{k+1}}$ is epic and $m_{x^{k+1}}$ is monic, we get that $\gamma^{k+1}_{x^{k+1}} = \beta_{x^{k+1}}$. However, this implies that $\gamma_{x^{k+1}}$ is an isomorphism with inverse $\gamma^{-1}_{x^{k+1}} = \gamma^k_{x^{k+1}} \beta^{-1}_{x^{k+1}} = \beta^{-1}_{x^{k+1}} \gamma^k_{x^{k+1}}$. Then setting $x^D \colon = \varepsilon_{x^k} (\gamma_{x^k}^{-1})^{k+1} m_{x^k}$, using the same calculations as in the proof of Theorem \ref{rank-theorem}, we get that $x^D$ is a Drazin inverse of $x$. \end{proof}



\section{Drazin Inverses and Idempotents} \label{drazin-idempotents}


An important class of endomorphisms in any category are the idempotents, that is, endomorphisms $e: A \to A$ such that $e^2 = e$. In this section, we explore the connection between Drazin inverses and idempotents. In particular, we will show how to characterize Drazin maps in terms of isomorphisms in the idempotent splitting of a category -- this displays an aspect of Drazin inverses which is categorically inspired.   Moreover, this allows us to relate the notion of a Drazin inverse to Leinster's concept of an eventual image \cite{leinster2022eventual}. Throughout this section, we work in an arbitrary category $\mathbb{X}$.

\subsection{Idempotents} We begin by observing that every Drazin endomorphism induces an idempotent by composing it with its Drazin inverse. 

\begin{lemma}\label{lemma:e_x} Let $x: A \to A$ be Drazin. Define the map $e_x\colon = x^Dx: A \to A$ (or equivalently by \textbf{[D.3]} as $e_x = xx^D$). Then $e_x$ is an idempotent.
\end{lemma}
\begin{proof} This is immediate by \textbf{[D.2]} since: $e_x e_x \colon = x^D x x^D x = x^D x = e_x$.
\end{proof}

On the other hand, every idempotent is Drazin and, indeed, is its own Drazin inverse. Moreover, every idempotent has Drazin index at most one, meaning that it is its own group inverse as well. The only idempotents with Drazin index zero are the identity maps. Thus, non-trivial idempotents have Drazin index one. 

\begin{lemma}\label{lemma:e-drazin} An idempotent $e: A \to A$ is Drazin, its own Drazin inverse, $e^D = e$, and $\mathsf{ind}(e) \leq 1$. Moreover, $\mathsf{ind}(e) = 0$ if and only if $e= 1_A$. 
\end{lemma}
\begin{proof} By Lemma \ref{lemma:ind1}, to show that $e$ is its own Drazin inverse and $\mathsf{ind}(e) \leq 1$, it suffices to show that it is its own group inverse. Then setting $e^D = e$, we clearly have that: \textbf{[G.1]} $ee^De= e$; \textbf{[G.2]} $e^D e e^D = e^D$; and  \textbf{[G.3]} $e^D e = e e^D$. So $e$ is its own group inverse, and therefore also its own Drazin inverse with $\mathsf{ind}(e) \leq 1$. Now if $\mathsf{ind}(e) = 0$, then by Lemma \ref{lem:Drazin-0}, $e$ is an isomorphism and its Drazin inverse is its inverse. This would mean that $e^{-1} = e^D = e$, so $e$ is its own inverse. However, by idempotency, this would imply that $1_A = e e^{-1} = ee = e$. So the only idempotent with Drazin index zero is the identity. 
\end{proof}

The converse is not true: that is not every Drazin endomorphism with Drazin index at most one is an idempotent. Instead, having a Drazin index less than or equal to one can be equivalently described in terms of \emph{binary idempotents}. 


\subsection{Binary Idempotents}\label{sec:binary-idempotent} A \textbf{binary idempotent} \cite[Def 9.2]{priyaathesis} is a pair $(f,g)$ consisting of maps $f: A \to B$ and $g: B \to A$ such that $fgf=f$ and $gfg=g$. If $(f,g)$ is a binary idempotent, then $fg: A \to A$ and $gf: B \to B$ are both idempotents. Now when $A=B$, a \textbf{commuting binary idempotent} is a pair $(x,y)$ consisting of endomorphisms $x: A \to A$ and $y: A \to A$ such that $(x,y)$ is a binary idempotent and $xy=yx$. For a commuting binary idempotent $(x,y)$ we denote its induced idempotent as $e_{(x,y)}: A \to A$, that is, $e_{(x,y)}= xy = yx$. 

\begin{lemma}\label{lemma:ind1=cbi} $x: A \to A$ is Drazin with $\mathsf{ind}(x) \leq 1$ if and only if there is an endomorphism $x^D: A \to A$ such that $(x,x^D)$ is a commuting binary idempotent. 
\end{lemma}
\begin{proof} Recall that by Lemma \ref{lemma:ind1}, $x$ is Drazin with $\mathsf{ind}(x) \leq 1$ if and only if $x$ has a group inverse $x^D$. However, the requirements for $(x,x^D)$ being a commuting binary idempotent are precisely the same as saying that $x^D$ is a group inverse of $x$. 
\end{proof}

While not every Drazin endomorphism and its Drazin inverse form a commuting binary idempotent, we can always obtain one involving the Drazin inverse. 

\begin{corollary} If $x: A \to A$ is Drazin, then $(x^D, x^{DD})$ is a commuting binary idempotent. Moreover, the resulting idempotent is precisely $e_{(x^D, x^{DD})} = e_x$.
\end{corollary}
\begin{proof} By Lemma \ref{lem:Drazin-inverse1}.(\ref{Drazin-inverse-inverse}), if $x$ is Drazin, then $x^D$ is also Drazin and $\mathsf{ind}(x^D) \leq 1$. Then by Lemma \ref{lemma:ind1=cbi}, it follows that $(x^D, x^{DD})$ is a commuting binary idempotent. Now recall that $x^{DD} = x x^D x$. Then using \textbf{[D.2]}, we have that:  
\[ e_{(x^D, x^{DD})} = x^D x^{DD} = x^D x x^D x = x^D x = e_x \]
So $e_{(x^D, x^{DD})} = e_x$ as desired. 
\end{proof}

\subsection{Drazin split} We can also consider the case when the induced idempotent of a Drazin endomorphism is \emph{split}. Recall that an idempotent $e: A \to A$ \textbf{splits} (or is a \textbf{split idempotent}) if there are maps $r: A \to B$ and $s: B \to A$ such that $rs=e$ and $sr=1_A$.   

\begin{definition} $x: A \to A$ is \textbf{Drazin split} if it is Drazin and the induced idempotent $e_x\colon = xx^D: A \to A$ splits. 
\end{definition}

We can alternatively characterize Drazin split endomorphisms as follows: 

\begin{lemma}\label{lemma:alpha1} $x: A \to A$ is Drazin split if and only if there is an idempotent $e: A \to A$, with splitting $r: A \to B$ and $s: B \to A$, a $k \in \mathbb{N}$, and an isomorphism $\alpha: B \to B$ such that the following diagram commutes: 
\[ \xymatrixcolsep{5pc}\xymatrix{   & A  \ar[d]^-r  \ar[r]^-x & A  \ar[d]^-r \\  A \ar[ru]^{x^k} \ar[dr]_{x^k}& B  \ar[d]^-s \ar[r]_-\alpha & B  \ar[d]^-s \\ & A  \ar[r]_-x & A } \]
\end{lemma}
\begin{proof} For $(\Rightarrow)$, suppose that $x$ is Drazin split where $e_x$ splits via $r: A \to B$ and $s: B \to A$, and with $\mathsf{ind}(x) = k$. For the required split idempotent we take $e_x$, and for the natural number we take $\mathsf{ind}(x) = k$. For the isomorphism, define the map $\alpha: B \to B$ as the composite $\alpha \colon = s x r$, and define the map $\alpha^{-1}: B \to B$ as the composite $\alpha^{-1} \colon = s x^D r$. We first check that these maps are inverses of each other. To do so, we use \textbf{[D.2]}, \textbf{[D.3]}, and the splitting:
\[ \alpha \alpha^{-1} = s x rs x^D r= s x e_x x^D  r =  s x x^D x x^D r = s x x^D r = s e_x r = s r sr = 1_B  \]
\[ \alpha^{-1} \alpha = s x^D r s x r = s x^D e_x x r = s x^D x x^D x r = s x^D x r = s e_x r = s r sr = 1_B \]
So $\alpha$ is indeed an isomorphism. Now we show that the diagram commutes. By \textbf{[D.1]}, the triangle on the left commutes: 
\[ x^k r s = x^k e_x = x^k x x^D = x^{k+1} x^D = x^k \]
By \textbf{[D.2]}, \textbf{[D.3]}, and the splitting, we can show that both the top and bottom squares commute: 
\[ x r = x r \alpha^{-1} \alpha = x r s x^D r \alpha = x e_x x^D r \alpha = x x^D x x^D r \alpha = x x^D r \alpha = e_x r \alpha = r s r \alpha = r \alpha \]
\[ sx = \alpha  \alpha^{-1} s x = \alpha s x^D r s x =  \alpha s x^D e_x x = \alpha s x^D x x^D x = \alpha s x^D x = \alpha s e_x = \alpha s r s = \alpha s  \]
So we conclude that the desired identities hold. 

For ($\Leftarrow$), define $x^D: A \to A$ as the composite $x^D \colon = r \alpha^{-1} s$. We show that $x^D$ satisfies the three Drazin inverse axioms. 
\begin{enumerate}[{\bf [D.1]}]
\item Using the $k\in \mathbb{N}$ in the assumptions, and that $x^ke = e$ and $r = x r \alpha^{-1}$ we have: 
\[x^{k+1} x^D = x^k x r \alpha^{-1} s = x^k r s = x^k e = x^k \]
\item Using $s x =  \alpha s$, we compute that: 
\[x^D x x^D =  r \alpha^{-1} s x  r \alpha^{-1} s = r \alpha^{-1} \alpha s r \alpha^{-1} s = rsr\alpha^{1} s = r \alpha^{-1} s = x^D\]
\item Using both $s x =  \alpha s$ and $x r = r \alpha$, we have: 
\[x^D x = r \alpha^{-1} s x = r \alpha^{-1} \alpha s = rs = r \alpha \alpha^{-1} s = x r \alpha^{-1}s = x x^D\]
\end{enumerate}
So $x$ is Drazin with Drazin inverse $x^D$. Next, observe that the induced idempotent $e_x$ is $e$ since: 
\[ e_x = x^D x = r \alpha^{-1} sx = r \alpha^{-1} \alpha s = rs = e \]
So the induced idempotent is precisely the idempotent with which we started. Since $e$ was assumed to be split, we conclude that $x$ is Drazin split. 
\end{proof}

Recall that a category is {\bf idempotent complete} if all its idempotents split. Similarly, we can consider categories in which every Drazin endomorphism Drazin splits:

\begin{definition} A \textbf{Drazin complete category} is a category in which every endomorphism Drazin splits. 
\end{definition}

\begin{lemma}\label{lemma:Drazin-complete} A category is Drazin complete if and only if it is Drazin and idempotent complete. 
\end{lemma}
\begin{proof} For ($\Rightarrow$), by definition of being Drazin complete, every endomorphism is Drazin split, so in particular Drazin, and therefore the category is Drazin. Moreover, again by definition of being Drazin complete, every idempotent is Drazin split and therefore since by Lemma \ref{lemma:e-drazin} it is its own induced idempotent, every idempotent also splits. For ($\Leftarrow$), by definition of being Drazin and idempotent complete, every endomorphism is Drazin and the induced idempotent splits. Therefore, every endomorphism is Drazin split and so the category is Drazin complete. \end{proof}

Every Drazin category embeds into a Drazin complete category via the idempotent splitting construction. Indeed, recall that every category $\mathbb{X}$ embeds into an idempotent complete category $\mathsf{Split}(\mathbb{X})$ called its idempotent splitting. The objects of $\mathsf{Split}(\mathbb{X})$ are pairs $(A,e)$ consisting of an object $A$ and an idempotent $e: A \to A$, and where a map $f: (A, e) \to (B, e^\prime)$ is a map $f: A \to B$ such that $efe^\prime = f$ (or equivalently $ef = f= fe^\prime$). Composition in $\mathsf{Split}(\mathbb{X})$ is defined as in $\mathbb{X}$, while the identity for $(A,e)$ is the idempotent $e$, that is, $1_{(A,e)}\colon = e: (A,e) \to (A,e)$. To show that the idempotent splitting of a Drazin category is Drazin complete, we first show that for any endomorphism in the idempotent splitting, whose underlying endomorphism in the base category is Drazin, has its Drazin inverse also an endomorphism in the idempotent splitting. 

\begin{lemma}\label{lemma:Drazininv-splitmap} If $x: (A,e) \to (A,e)$ is an endomorphism in $\mathsf{Split}(\mathbb{X})$ and $x: A \to A$ is Drazin in $\mathbb{X}$, then $x^D: (A,e) \to (A,e)$ is an endomorphism in $\mathsf{Split}(\mathbb{X})$, and moreover, $x: (A,e) \to (A,e)$ is Drazin with Drazin inverse $x^D: (A,e) \to (A,e)$.
\end{lemma}
\begin{proof} Since  $x: (A,e) \to (A,e)$ is a map in the idempotent splitting, we have that $exe = x$. Then using \textbf{[D.2]} and \textbf{[D.3]}, we compute that: 
\[ e x^D e = e x^D x x^D e = e x^D x^D x e = e x^D x^D x = e x x^D x^D x = x x^D x^D = x^D x x^D = x^D \]
    So $e x^D e = x^D$, and therefore we have that $x^D: (A,e) \to (A,e)$ is indeed an endomorphism in $\mathsf{Split}(\mathbb{X})$. Since the composition in $\mathsf{Split}(\mathbb{X})$ is the same as in $\mathbb{X}$, we get that $x^D: (A,e) \to (A,e)$ is also the Drazin inverse of $x: (A,e) \to (A,e)$ in $\mathsf{Split}(\mathbb{X})$. 
\end{proof}

\begin{proposition} If $\mathbb{X}$ is a Drazin category, then $\mathsf{Split}(\mathbb{X})$ is Drazin complete.
\end{proposition} 
\begin{proof} Since we know that $\mathsf{Split}(\mathbb{X})$ is idempotent complete, to show that it is also Drazin complete, by Lemma \ref{lemma:Drazin-complete} we need only show that it is also Drazin. So consider $x: (A,e) \to (A,e)$ in $\mathsf{Split}(\mathbb{X})$. Since $\mathbb{X}$ is Drazin, $x: A\to A$ is Drazin. So by Lemma \ref{lemma:Drazininv-splitmap}, we have that $x: (A,e) \to (A,e)$ is also Drazin. So $\mathsf{Split}(\mathbb{X})$ is Drazin, and so we conclude that $\mathsf{Split}(\mathbb{X})$ is Drazin complete as well.  
\end{proof}

\subsection{Intertwining idempotents} We now begin to build up towards showing that Drazin inverses give rise to isomorphisms in the idempotent splitting. Toward this goal, we first introduce a useful intermediate notion. Let $(e,e^\prime)$ be a pair of idempotents $e: A \to A$ and $e^\prime: B \to B$. We shall say that a map $f: A \to B$ is {\bf $(e,e^\prime)$-intertwined} if $ef = fe^\prime$, so the below diagram on the left commutes. We say that an $(e,e^\prime)$-intertwined map $f$ has an {\bf $(e,e^\prime)$-intertwined inverse} in case there is a map $g: B \to A$ such that $e^\prime g e = g$ (or equivalently if $e^\prime g = g = ge$), and $fg= e$ and $gf = e^\prime$, that is, the diagram on the right below commutes. 
\begin{align*}
    \begin{matrix}
 \xymatrixcolsep{5pc}\xymatrix{A  \ar[d]_-e \ar[r]^-f & B \ar[d]^-{e^\prime} \\  A \ar[r]_-f & B} & ~~~~~~~~~ &
\xymatrixcolsep{5pc}\xymatrix{A  \ar[d]_-e \ar[r]^-f & B \ar[d]^-{e^\prime} \ar@{-->}[ldd]|g \ar@{-->}[ld]|g\\  
                A \ar[d]_-e & B \ar[d]^-{e^\prime} \ar@{-->}[ld]|g\\
                 A \ar[r]_-f & B} \\
                 (e,e^\prime)\mbox{-intertwined}  & ~ &(e,e^\prime)\mbox{-intertwined inverse}
 \end{matrix}
\end{align*}

\begin{lemma}\label{lemma:intertwined-iso} Let $f: A \to B$ be a $(e,e^\prime)$-intertwined map in $\X$. Then: 
\begin{enumerate}[(i)]
\item \label{lemma:intertwined-iso1} $fe^\prime : (A,e) \to (B, e^\prime)$ is a map in $\mathsf{Split}(\mathbb{X})$;
\item \label{lemma:intertwined-iso2} If $f$ has an $(e,e^\prime)$-intertwined inverse $g: B \to A$, then $fe^\prime : (A,e) \to (B, e^\prime)$ is an isomorphism in $\mathsf{Split}(\mathbb{X})$ with inverse $g: (B, e^\prime) \to (A,e)$. 
\end{enumerate} 
\end{lemma}
\begin{proof} For (i), since $f$ is  $(e,e')$-intertwined, we have that:
\[ e fe^\prime  e^\prime = e e f e^\prime = e f e^\prime = e e f = e f= fe^\prime  \]
So $fe^\prime : (A,e) \to (B, e^\prime)$ is indeed a map in $\mathsf{Split}(\mathbb{X})$. For (ii), by definition of an $(e,e^\prime)$-intertwined inverse, we have that $e^\prime g e =g$. So we get that $g: (B, e^\prime) \to (A,e)$ is indeed a map in $\mathsf{Split}(\mathbb{X})$. Using that $fg = e$ and $gf = e^\prime$, we can show that $fe^\prime$ and $g$ are inverses in the idempotent splitting:
\begin{align*}
  fe^\prime  g = e fg = e e = e = 1_{(A,e)} &&   g fe^\prime = e^\prime e^\prime = e^\prime = 1_{(B,e^\prime)}
\end{align*}
So we conclude that $fe^\prime : (A,e) \to (B, e^\prime)$ is an isomorphism in $\mathsf{Split}(\mathbb{X})$. \end{proof}

We can now characterize being Drazin in terms of being idempotent intertwined. 

\begin{lemma} 
\label{drazin-intertwined} $x: A \to A$ is Drazin if and only if there is an idempotent $e: A \to A$ and a $k \in \mathbb{N}$ such that $x^k e = x^k$ and 
$x$ is an $(e,e)$-intertwined map with an $(e,e)$-intertwined inverse.
\end{lemma}
\begin{proof} For $(\Rightarrow)$, suppose that $x$ is Drazin with $\mathsf{ind}(x) = k$. Take the idempotent $e$ to be the induced idempotent $e_x$. Then by \textbf{[D.1]}, we have that: 
\[ x^k e_x = x^k x x^D = x^{k+1} x^D = x^k\]
So $x^k e_x = x^k$. Now by \textbf{[D.3]}, we have that:
\[ e_x x = x x^D x = x e_x  \]
So $x$ is $(e_x,e_x)$-intertwined. Next, by \textbf{[D.2]}, we have that: 
\[ e_x x^D  = x^D x x^D  = x^D   \]
Moreover, by definition, we also have that $xx^D = e_x$ and $x^D x=e_x$. So we conclude that $x$ has a $(e_x,e_x)$-intertwined inverse and it is $x^D$. 

For $(\Leftarrow)$, let $x^D$ be the $(e,e)$-intertwined inverse of $x$. We show that $x^D$ satisfies the three Drazin inverse axioms. 
\begin{enumerate}[{\bf [D.1]}]
\item Using that $x x^D = e$ and $x^k e = x^k$, we compute that: 
\[ x^{k+1} x^D = x^k x x^D = x^k e = x^k \]
\item Using that $x x^D = e$ and $x^D e = x^D$, we compute that: 
\[ x^D x x^D = x^D e = x^D \]
\item By definition of being an $(e,e)$-intertwined inverse, we have that $xx^D =e$ and $x^D x=e$. So $xx^D = x^D x$. 
\end{enumerate}
So we conclude that $x$ is Drazin with Drazin inverse $x^D$. \end{proof}

We can therefore completely capture the requirements of a Drazin inverse with the following commuting diagram:
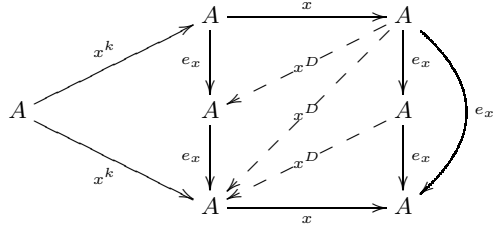
\[ \xymatrixcolsep{5pc}\xymatrix{& A \ar[d]_-{e_x} \ar[r]^-x & A \ar[d]^-{e_x} \ar@{-->}[dl]|{x^D} \ar@{-->}[ddl]|{x^D} \ar@/^2pc/[dd]^{e_x} \\
                  A \ar[ur]^{x^k} \ar[dr]_{x^k} & A \ar[d]_-{e_x} & A \ar[d]^-{e_x} \ar@{-->}[dl]|{x^D} \\
                  & A  \ar[r]_-x & A } \]
Moreover, by Lemma \ref{lemma:intertwined-iso}, we see how being Drazin induces an isomorphism on its induced idempotent: 

\begin{corollary}\label{cor:iso-split} Let $x: A \to A$ be Drazin in $\X$, then $x e_x: (A, e_x) \to (A, e_x)$ is an isomorphism in $\mathsf{Split}(\mathbb{X})$ with inverse $x^D: (A, e_x) \to (A, e_x)$. 
\end{corollary}
\begin{proof} In the proof of Lemma \ref{drazin-intertwined}, we showed that $x$ is $(e_x,e_x)$-intertwined and $x^D$ is its $(e_x,e_x)$-intertwined inverse. Applying Lemma \ref{lemma:intertwined-iso}.(\ref{lemma:intertwined-iso2}), we get that ${x e_x\! :\! (A, e_x) \to (A, e_x)}$ is an isomorphism in $\mathsf{Split}(\mathbb{X})$ with inverse $x^D: \! (A, e_x) \to (A, e_x)$. \end{proof}

\subsection{Munn's Power Theorem} We now generalize Drazin's \cite[Thm 7]{drazin1958pseudo} to the categorical context using idempotent splitting. In the context of semigroups, this result was expressed particularly succinctly by Munn as ``An element $x$ of a semigroup $S$ is pseudo-invertible if and only if some power of $x$ lies in a subgroup of $S$'' \cite[Thm 1]{munn1961pseudo}. We can understand this statement categorically using the idempotent splitting, that is, a map $x$ is Drazin if and only if some positive power of $x$ underlies an endo-isomorphism in the idempotent splitting.

\begin{theorem}\label{characterization-by-iteration} $x: A \to A$ is Drazin in $\X$ if and only if there is an idempotent $e: A \to A$ such that for some $k \in \mathbb{N}$, $x^{k+1}: (A,e) \to (A,e)$ is an isomorphism in $\mathsf{Split}(\mathbb{X})$. 
\end{theorem}
\begin{proof} For ($\Rightarrow$), let $x: A \to A$ be Drazin with $\mathsf{ind}(x) = k$. By Lemma \ref{lem:Drazin-iteration}, we know that $x^{k+1}$ is also Drazin. Therefore by Corollary \ref{cor:iso-split}, $x^{k+1} e_{x^{k+1}}: (A, e_{x^{k+1}}) \to (A, e_{x^{k+1}})$ is an isomorphism in $\mathsf{Split}(\mathbb{X})$. Now recall that $(x^{k+1})^D = (x^D)^{k+1}$. Then by using \textbf{[D.3]}, we compute that: 
\[ e_{x^{k+1}} = (x^{k+1}) (x^{k+1})^D = x^{k+1} (x^D)^{k+1} = (x x^D)^{k+1} = e_x^{k+1} = e_x \]
So $e_{x^{k+1}} = e_x$. Moreover, using \textbf{[D.1]} we compute that: 
\[ x^{k+1} e_x = x^{k+1} x^D x = x^{k} x= x^{k+1} \]
So $x^{k+1} e_x = x^{k+1}$. So we conclude that $x^{k+1}: (A,e_x) \to (A,e_x)$ is an isomorphism in $\mathsf{Split}(\mathbb{X})$ as desired. 

For ($\Leftarrow$), this direction requires more work. So suppose that there is some $k$ such that $x^{k+1}: (A,e) \to (A,e)$ is an isomorphism in $\mathsf{Split}(\mathbb{X})$ with inverse $v: (A,e) \to (A,e)$. Explicitly this means that the following equalities hold:
\begin{align*}
e x^{k+1} = x^{k+1} = x^{k+1} e && ev = v =ve && x^{k+1} v = 1_{(A,e)} = e && v x^{k+1} = 1_{(A,e)} = e   
\end{align*}
We will first show that $x^k$ and $v$ commute with each other. To do so, we use the above identities: 
\begin{gather*}
    x^k v = x^k e v = x^k x^{k+1} v v = x^{k+1} x^k v v = e x^{k+1} x^k v v = e x^k x^{k+1} v v = e x^k e v = e x^k v \\
    = v x^{k+1} x^k v = v x^k x^{k+1} v = v x^k e = v e x^k e = v v x^{k+1} x^k e = v v x^k x^{k+1} e \\
    = v v x^{k} x^{k+1} =  v v x^{k+1} x^k = v e x^k = v x^k
\end{gather*}
So $v x^k = x^k v$. Then define $x^D \colon = v x^k = x^k v$. We show that $x^D$ satisfies the three Drazin inverse axioms. 
\begin{enumerate}[{\bf [D.1]}]
\item Here we use that $x^{k+1} v=e$ and $x^{k+1} e = x^{k+1}$: 
\[x^{(k+1)+1} x^D = x^{k+1} x x^k v = x^{k+1} x^{k+1} v = x^{k+1} e = x^{k+1}\]
\item Here we use that $x^{k+1} v =e $ and $ve = v$: 
\[x^D x x^D = x^{k} v x x^{k} v= x^{k} v x^{k+1} v = x^{k} v e = x^{k} v = x^D\]  
\item Here we use that $x^{k+1} v =  e  = v x^{k+1}$: 
\[x x^D = x x^{k} v = x^{k+1} v =  e  = v x^{k+1} = v x^{k} x = x^D x\]
\end{enumerate}
So we conclude that $x$ is Drazin.\end{proof}


\subsection{Eventual Image Duality} \label{ssec:eventual-image} In \cite{leinster2022eventual}, Leinster introduced the concept of the \emph{eventual image} of an endomorphism in a category. Naively, the idea is that after iterating an endomorphism enough times, it will eventually stabilize, and this is the eventual image. Since this concept of iterating an endomorphism until it stabilizes is also closely linked to Drazin inverses, it is natural to ask what the relationship is between Drazin inverses and Leinster's eventual image. Here we answer this question by using the results relating Drazin inverses to idempotents. In particular, we show that every Drazin split endomorphism has an eventual image which is absolute.  

Let $\mathbb{Z}$ be the set of integers. Then $x: A \to A$ in $\X$ has \textbf{eventual image duality} \cite[Definition 2.1]{leinster2022eventual} in the case the diagram $\mathsf{EI}(x): \mathbb{Z} \to \X$ given by the doubly infinite chain: 
\[ ... \to^x A\to^x A\to^x A\to^x ...\]
has both a limit cone $(\pi_i: L \to A)_{i \in \mathbb{Z}}$ and a colimit cocone $(\sigma_i: A \to M)_{i \in \mathbb{Z}}$ such that the canonical map $\tilde{x} \colon= \pi_i \sigma_i: L \to M$ (which is independent of the choice of $i \in \mathbb{Z}$) is an isomorphism. Of course, we can arrange for this isomorphism to be the identity, so we can obtain: 
\[ \xymatrix{ & & & {\sf im}^\infty(x) \ar[dll]|{\pi_{-2}} \ar[dl]|{\pi_{-1}} \ar[d]|{\pi_0} \ar[dr]|{\pi_1}  \ar[drr]|{\pi_2} \\
                   ~ \ar@{}[r]|{....}  & A \ar[r]|-{x} \ar[drr]|{\sigma_{-2}} & A \ar[dr]|{\sigma_{-1}} \ar[r]|-{x} & A \ar[d]|{\sigma_{0}} \ar[r]|-{x} & A \ar[dl]|{\sigma_{1}} \ar[r]|f & A  \ar[dll]|{\sigma_{2}} \ar@{}[r]|{....} & ~\\
                   & & & {\sf im}^\infty(x) } \]
where $\pi_i\sigma_i =1_{{\sf im}^\infty(x)}$ making $\pi_i$ a section of the retraction $\sigma_i$. This induces the idempotent $x^\infty \colon = \sigma_i\pi_i: A \to A$, which is independent of the choice of $i \in \mathbb{Z}$. 

In an arbitrary category, not all endomorphisms have eventual image duality, see for example \cite[Example 2.2]{leinster2022eventual} which explains why most endomorphisms in $\mathsf{SET}$ do not have eventual image duality. Furthermore, eventual image duality need not be absolute, that is, eventual images are not necessarily preserved by all functors. However, there is a class of endomorphisms which do have {\em absolute\/} eventual image duality, and this class includes the ones induced by being Drazin.  

\begin{definition} For $x: A \to A$, we say that $(s_i: E \to A,r_i: A \to E)_{i \in \mathbb{Z}}$ {\bf eventuates} $x$ with \textbf{index} $k \in \mathbb{N}$ in case:
    \begin{enumerate}[{\bf [ev.1]}]
        \item For all $i \in \mathbb{Z}$,  $s_ir_i = 1_E$; 
        \item For all $i,j \in \mathbb{Z}$, $r_i s_i = r_j s_j$; 
        \item For all $i \in \mathbb{Z}$, $s_i x = s_{i+1}$ and $x r_{i+1} = r_i$;
        \item For all $i \in \mathbb{Z}$, $r_i s_i x^{k+1} = x^{k+1} = x^{k+1} r_i s_i$. 
    \end{enumerate}
\end{definition}

We first show that if an endomorphism has a family of maps which eventuates it, then it has an absolute eventual image duality. 

\begin{lemma}\label{lemma:eventuates1}
    If $(s_i: E \to A,r_i: A \to E)_{i \in \mathbb{Z}}$ eventuates $x: A \to A$ with index $k$, then $x$ has an absolute eventual image duality.
\end{lemma}
\begin{proof} Note that by \textbf{[ev.3]}, $(r_i: A \to E)_{i \in \mathbb{Z}}$ is a cone and $(s_i: E \to A)_{i \in \mathbb{Z}}$ a cocone over the diagram $\mathsf{EI}(x)$. So we need to show that they are respectively a limit and a colimit. So let's show that $(s_i)_{i \in \mathbb{Z}}$ is a universal cone over $\mathsf{EI}(x)$. So suppose that $(z_i: Z \to A)_{i \in \mathbb{Z}}$ is a cone over $\mathsf{EI}(x)$: 
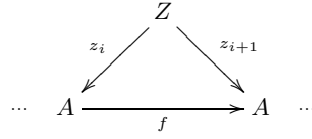
    \[ \xymatrix{ & & Z \ar[dl]_{z_i} \ar[dr]^{z_{i+1}} \\
       ~ \ar@{}[r]|{...} & A \ar[rr]_x & & A \ar@{}[r]|{...} & ~} \]
Then we first observe that:  
\[ z_{i+1} r_{i+1} = z_i x r_{i+1} = z_i r_i \]
So for any $i, j \in \mathbb{Z}$, we have that $z_i r_i = z_j r_j$. So define the map $h: Z \to E$ as the composite $h\colon = z_ir_i$ (for any $i \in \mathbb{Z}$). Next note that by the cone, we have that $z_i = z_{i-n}x^n$ for all $n\in \mathbb{N}$. In particular by \textbf{[ev.4]}, we get that:   
    \[  h s_i = z_i r_i s_i = z_{i-(k+1)}x^{k+1} r_i s_i = z_{i-(k+1)}x^{k+1} = z_i  \]
So we get that the following diagram commutes:  
    
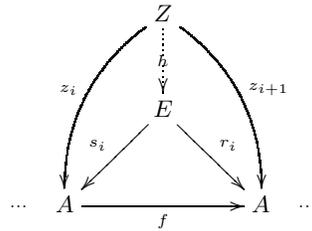
\[ \xymatrix{ & & Z \ar@{.>}[d]|-{h} \ar@/_1pc/[ddl]_{z_i} \ar@/^1pc/[ddr]^{z_{i+1}} \\ 
                  & &E \ar[dl]_{s_i} \ar[dr]^{r_i} \\
                  ~ \ar@{}[r]|{...} & A \ar[rr]_x & & A \ar@{}[r]|{...} & ~ } \]
Now by \textbf{[ev.1]}, each $s_i$ is monic. Therefore it follows that $h$ is necessarily unique, and therefore we get that $(s_i)_{i \in \mathbb{Z}}$ is a limit cone. Similarly, by dual arguments, we also get that $(r_i)_{i \in \mathbb{Z}}$ is a colimit cocone. By \textbf{[ev.1]}, the canonical map is $\tilde{x} = s_i r_i = 1_E$. So we have that $x$ has eventual image duality. Finally, this (co)limit is clearly absolute as the reasoning only uses constructs which are preserved by all functors.
\end{proof}

We now show that if an endomorphism is Drazin split, then we can build a family which eventuates it: 

\begin{lemma}\label{lemma:eventuates2} If $x: A \to A$ is Drazin split where $e_x$ splits via $r: A \to E$ and $s: E \to A$, then $(s_i: E \to A,r_i: A \to E)_{i \in \mathbb{Z}}$ eventuates $x$ with index $\mathsf{ind}(x) =k$ where: 
    \begin{description} 
    \item[ {[$i = 0$]}:] $s_0 \colon = s$,  $r_0 \colon = r$;
    \item[ {[$i > 0$]}:]  $s_i \colon = s x^i$ and $r_i \colon = (x^D)^i r$;
    \item[ {[$i < 0$]}:]\footnote{Notice in the last case since $i<0$ that $-i$ is positive so that the definitions of $s_i$ and $r_i$ make sense.
} $s_i \colon = s (x^D)^{-i}$ and $r_i = x^{-i} r$.
    \end{description}
\end{lemma}
\begin{proof}
    \begin{enumerate}[{\bf [ev.1]}]
        \item Recall that $sr=1_E$. When $i=0$, $s_0r_0 = sr = 1_E$.  If $i > 0$ then by \textbf{[D.3]} and that $e_x$ is idempotent: 
        \[ s_ir_i = s x^i (x^D)^i r = s (xx^D)^i r = s e_x^i r= se_x r = srsr = 1_E \]
        while if $i < 0$ then similarly: 
        \[ s_ir_i = s (x^D)^{-i} x_i^{-i} = s(x^Dx)^{-i} r = s e_x^{-i} r = s e_x r = sr s r = 1_E\]
        \item When $i=0$, $r_0s_0 = rs = e_x$. When $i > 0$, by \textbf{[D.3]} and that $e_x$ is idempotent, we compute that: 
        \[ r_i s_i = (x^D)^i r s x^i = (x^D)^i e_x x^i = (x^D)^i x^D x x^i = (x^D)^{i+1} x^{i+1} = (x^Dx)^{i+1} = e_x^{i+1} = e_x \]
        When $i<0$, we compute that: 
        \begin{gather*} r_i s_i = x^{-i} r s (x^D)^{-i} = x^{-i} e_x (x^D)^{-i} =  x^{-i} x x^D (x^D)^{-i} \\
        =x^{-i+1}  (x^D)^{-i+1}  = (xx^D)^{-i+1} = e_x^{-i+1} = e_x 
        \end{gather*}
        So for $i \in \mathbb{Z}$, $r_i s_i = e_x$.  
        \item When $i=0$: 
        \begin{gather*}
            s_0 x = s x = s_1 \\
            x r_1 = x x^D r = e_x r = rsr = r = r_0
        \end{gather*}
When $i >0$:
  \begin{gather*}
           s_i x = s x^i x = s x^{i+1} = s_{i+1} \\
           x r_{i+1} = x (x^D)^{i+1} r = (x^D)^{i} x x^D r =(x^D)^{i} e_x r = (x^D)^{i} rsr = (x^D)^{i} r = r_i
        \end{gather*}
        When $i < 0$: 
     \begin{gather*}
           s_i x = s (x^D)^{-i} x \!=\! s x x^D (x^D)^{-(i+1)} = s e_x (x^D)^{-(i+1)} = s rs (x^D)^{-(i+1)}  = s (x^D)^{-(i+1)}  = s_{i+1} \\
           x r_{i+1} = x x^{-(i+1)} r = x^{-i} r = r_i 
        \end{gather*}    
        \item As shown above, for all $i \in \mathbb{Z}$ we have that $r_i s_i = e_x$. Now in the proof of Theorem \ref{characterization-by-iteration}, we showed that $x^{k+1}: (A,e_x) \to (A,e_x)$ was a map in the idempotent splitting. Recall this means that $e_x x^{k+1} = x^{k+1} = x^{k+1} e_x$. So we get that $r_i s_i x^{k+1} = x^{k+1} = x^{k+1} r_i s_i$. 
    \end{enumerate} 
So we conclude that $(s_i,r_i)_{i \in \mathbb{Z}}$ eventuates $x$. 
\end{proof}

Then by applying Lemma \ref{lemma:eventuates1} to a Drazin split endomorphism, we get that: 

\begin{corollary}\label{cor:Drazin-eventual} An endomorphism that is Drazin split has eventual image duality.
\end{corollary}

This raises the question when if $(s_i,r_i)_{i \in \mathbb{Z}}$ eventuates $x$, whether $x$ is Drazin. In general, this does not appear to be the case. We now add what is missing:

\begin{lemma}\label{lemma:eventual3} $x: A \to A$ is Drazin split if and only if there exists a family of maps $(s_i: E \to A,r_i: A \to E)_{i \in \mathbb{Z}}$ which eventuates $x$ with index $k$ and $s_{i} x r_i: E \to E$ is an isomorphism for some $i \in \mathbb{Z}$.
\end{lemma}
\begin{proof} For ($\Rightarrow$), suppose that $x: A \to A$ is Drazin split where $e_x$ splits via $r: A \to E$ and $s: E \to A$. Then in Lemma \ref{lemma:eventuates2}, we showed that the family $(s_i: E \to A,r_i: A \to A)_{i \in \mathbb{Z}}$ eventuates $x$. Now in the proof of Lemma \ref{lemma:alpha1}, we showed that $\alpha = s x r$ was an isomorphism. However recall that $s_0 = s$ and $r_0 = r$. So $s_0 x r_0$ is an isomorphism. 

For ($\Leftarrow$), suppose that $(s_i,r_i)_{i \in \mathbb{Z}}$ eventuates $x$ with index $k$ and $s_{i} x^k r_i$ is an isomorphism for some $i \in \mathbb{Z}$. To show that $x$ is Drazin split, we will apply Lemma \ref{lemma:alpha1}. By \textbf{[ev.1]} and \textbf{[ev.2]}, we get a split idempotent $e = r_{i+1} s_{i+1}$ with splitting $r=r_{i+1}$ and $s= s_{i+1}$. Moreover, we also have our isomorphism by setting $\alpha = s_{i} x r_i$. Lastly, take the natural number to be $k+1$. Now note that by \textbf{[ev.3]}, we get that $x^n r_{j+n} = r_{j}$ and $s_{j} x^n = s_{j+n}$ for all $j \in \mathbb{Z}$ and $n \in \mathbb{N}$. From these and \textbf{[ev.4]}, we compute that: 
\[ r \alpha = r_i s_i x r_i = r_i s_i x x^k r_{i+k} = r_i s_i x^{k+1} r_{i+k} = x^{k+1} r_{i+k} = x x^k r_{i+k} = x r_i = x r \]
\[ \alpha s = s_i x r_i s_i = s_{i-k} x^k x r_i s_i =  s_{i-k} x^{k+1} r_i s_i = s_{i-k} x^{k+1} = s_{i-k} x^k x = s_i x = sx \]
Moreover, by \textbf{[ev.4]} we get that $x^{k+1} rs = x^{k+1}$. So the diagram in Lemma \ref{lemma:alpha1} commutes, and therefore we conclude that $x$ is Drazin split. 
\end{proof}



\section{Drazin Inverses in Additive Categories} \label{sec:additive}


In this section we consider Drazin inverses in \emph{additive} categories, following Robinson and Puystjens' work in \cite{robinson1987generalized}. In this paper, by an additive\footnote{Additive categories are sometimes called ``ringoids'' (a many object ring), or $\mathsf{Ab}$-enriched categories.} category, we mean a category $\mathbb{X}$ which is enriched over Abelian groups, that is, every homset $\mathbb{X}(A,B)$ is an Abelian group  and composition, $\mathbb{X}(A,B) \ox \mathbb{X}(B,C) \to \mathbb{X}(A,C)$, and composition units, $\mathbb{Z} \to \mathbb{X}(A,A)$, are Abelian group homomorphisms. 

\subsection{Right/Left $\pi$-Regularity in Additive Categories}\label{sec:additive-pi} In Section \ref{sec:strong-pi}, we reviewed how being Drazin was equivalent to being strongly $\pi$-regular. The definition of being strongly $\pi$-regular can be split into two parts: being left $\pi$-regular and being right $\pi$-regular. Dischinger showed that for a ring, being left $\pi$-regular was equivalent to being right $\pi$-regular\footnote{Dischinger only states and proves that for a ring, being right $\pi$-regular implies being left $\pi$-regular. However, by dualizing the proof, it is clear that for a ring, being left $\pi$-regular also implies being right $\pi$-regular.} \cite[Thm 1]{Dischinger}. Unfortunately the same is not true for an arbitrary category (in fact Dischinger even remarks that his proof does not extend to semigroups). Here we show that Dischinger's result generalizes to the setting of an additive category. Therefore, for an additive category, being Drazin is equivalent to being left $\pi$-regular or right $\pi$-regular.  

We first define the notions of left/right $\pi$-regularity for an arbitrary category. 

\begin{definition}\label{def:left-right-pi} In a category $\mathbb{X}$, $x: A \to A$ is said to be: 
\begin{enumerate}[(i)]
\item \textbf{right $\pi$-regular} if there is an endomorphism $x^R: A \to A$ and a $k \in \mathbb{N}$ such that $x^{k+1}x^R = x^k$. The $k \in \mathbb{N}$ is a called a \textbf{right $\pi$-index}.  
\item \textbf{left $\pi$-regular} if there is an endomorphism $x^L: A \to A$ and a $k \in \mathbb{N}$ such that $x^L x^{k+1} = x^k$. The $k \in \mathbb{N}$ is a called a \textbf{left $\pi$-index}. 
\end{enumerate}
An object $A \in \mathbb{X}$ is said to be \textbf{right (resp. left) $\pi$-regular} if every endomorphism of type $A \to A$ is right (resp. left) $\pi$-regular. Similarly, a category $\mathbb{X}$ is said to be \textbf{right (resp. left) $\pi$-regular} if every endomorphism in $\mathbb{X}$ is right (resp. left) $\pi$-regular. 
\end{definition}

Of course, by definition, we have that an endomorphism which is strongly $\pi$-regular is both right $\pi$-regular and left $\pi$-regular. Now recall from Section \ref{sec:strong-pi} that being strongly $\pi$-regular was equivalent to being Drazin. So we may state:

\begin{lemma}In a category $\mathbb{X}$, if $x: A \to A$ is Drazin, then $x$ is both right $\pi$-regular and left $\pi$-regular. Moreover, if $\mathbb{X}$ is a Drazin category (resp. if $A$ is a Drazin object), then $\mathbb{X}$ (resp. $A$) is both right $\pi$-regular and left $\pi$-regular.
\end{lemma}

We work towards showing that in an \emph{additive category}, the converse is also true. To do so, we first make the following useful observation:

\begin{lemma}\label{lemma:right-power} In a category $\mathbb{X}$, $x: A \to A$ is right (resp. left) $\pi$-regular if and only if there is a $k \in \mathbb{N}$ such that $x^{k+1}$ is right (resp. left) $\pi$-regular with a right (resp. left) $\pi$-index $1$.
\end{lemma}
\begin{proof} We prove the statement for right $\pi$-regularity -- the left $\pi$-regular statement is proved in dual fashion. For ($\Rightarrow$), suppose $x$ is right $\pi$-regular with right $\pi$-index $k \in \mathbb{N}$, that is, there is an endomorphism $x^R: A \to A$ such that $x^{k+1}x^R = x^k$. We first prove by induction that for all $m \in \N$, $x^{k+m} (x^R)^m = x^k$. The base case $m=0$ is trivial. So suppose that for all $0 \leq j \leq m$, $x^{k+j} (x^R)^j = x^k$. Then for $m+1$ we compute that: 
\[ x^{k+m+1} (x^R)^{m+1} = x x^{k+m} (x^R)^m x^R = x x^k x^R = x^{k+1} x^R = x^k \]
So $x^{k+m} (x^R)^m = x^k$ holds for all $m \in \mathbb{N}$. Then setting $(x^{k+1})^R = (x^R)^{k+1}$, we compute: 
\[ (x^{k+1})^{2} (x^{k+1})^R = x x^{k+k+1} (x^R)^{k+1} = x x^k = x^{k+1}  \]
Therefore, $x^{k+1}$ is right $\pi$-regular with a right $\pi$-index of $1$. 

For ($\Leftarrow$), Suppose that for some $k \in \mathbb{N}$, $x^{k+1}$ is right $\pi$-regular with a right $\pi$-index of $1$. So there is an endomorphism $y: A \to A$ such that $(x^{k+1})^{2} y = x^{k+1}$. Now define $x^R = y x^{k}$. So we compute: 
\[ x^{2k+1+1} x^R = (x^{k+1})^{2} y x^k = x^{k+1} x^k = x^{2k+1} \]
So $x$ is right $\pi$-regular. 
\end{proof}

We are now in a position to prove our desired result for additive categories. The proof follows the same steps as in \cite{Dischinger}, however, we fill in some gaps and provide some of the details which were omitted. 

\begin{theorem}\label{right-implies-left} In an additive category $\mathbb{X}$, an object $A$ is right $\pi$-regular if and only if $A$ is left $\pi$-regular. Therefore, an additive category $\mathbb{X}$ is right $\pi$-regular if and only if $\mathbb{X}$ is left $\pi$-regular.
\end{theorem}
\begin{proof} For ($\Rightarrow$), suppose that $A$ is right $\pi$-regular. Then every endomorphism $x: A \to A$ is right $\pi$-regular. As such by Lemma \ref{lemma:right-power}, there is a $k \in \mathbb{N}$ such that $x^{k+1}$ is right $\pi$-regular with a right $\pi$-index of $1$. So let $y: A \to A$ be an endomorphism such that $(x^{k+1})^{2} y = x^{k+1}$. Now $y$ is also right $\pi$-regular, so there is a $j \in \mathbb{N}$ such that $y^{k+1}$ is right $\pi$-regular with a right $\pi$-index of $1$. Setting $b=y^{j+1}$, that means we have an endomorphism $c: A \to A$ such that $b^2 c= b$. Now setting $a= (x^{k +1})^{j+1}$, we compute that: 
\begin{gather*}
    a^2 b = ((x^{k +1})^{j+1})^2 y^{j+1} \!=\! (x^{k +1})^{2j + 2} y^{j+1} = (x^{k+1})^{j} (x^{k +1})^{1+j+1} y^{j+1}  =(x^{k +1})^{j} x^{k +1}  = (x^{k+1})^{j+1} = a 
\end{gather*}
So $a$ is right $\pi$-regular with a right $\pi$-index of $1$, and in particular $a^2 b =a$. Now from $a^2 b =a$ and $b^2 c = c$, we compute the following equalities: 
\begin{eqnarray*}
        abc & = & a^2b^2c = a^2b = a\\
        ac & = & a^2bc = a(abc) = a^2 \\
        (c-a)^2 & = & c^2 -ac -ca + a^2 = c^2 -ca = c(c-a) \\
        ab(c-a)^2 & = & abc(c-a) = a(c-a) = ac-a^2 = 0 \\
        b(c-a) & = & b^2c(c-a) = b^2(c-a)^2 
    \end{eqnarray*}
The last equation can be iterated to give $b(c-a) = b^{n+1}(c-a)^{n+1}$ for all $n \in \mathbb{N}$. Now, as $(c-a)$ is also right $\pi$-regular, we have a $d$ and a $p \in \mathbb{N}$ such that $(c-a)^{p+1} d = (c-a)^p$. From this, we get that: 
\begin{gather*} ab(c-a) = a b^{p+1}(c-a)^{p+1} =ab^{p+1}(c-a)^{p+2}d^2 = \\
ab^{p+1}(c-a)^{p+1} (c-a) d^2 = a b (c-a) (c-a) d^2 =  ab(c-a)^2 d^2= 0 \end{gather*}
So from $ab(c-a) = 0$ we get that $a = abc = aba$. This gives us the following: 
\begin{gather*} ab - ab^2a = ab^2c - ab^2a = ab^2(c-a) = abb (c-a) = \\
ab b^{p+1}(c-a)^{p+1} = ab^{p+2}(c-a)^{p+2}d^2  = ab(c-a)d^2 = 0 \end{gather*}
As such we get that $ab = ab^2a$.  Finally, we have that:
\[ a b^2 a^2 = a b^2 a a = ab a = a  \]
Therefore $a = (x^{k +1})^{j+1}$ is left $\pi$-regular with index 1. Then by Lemma \ref{lemma:right-power}, $x$ is left $\pi$-regular. For ($\Leftarrow$), the proof of this direction is dual to that of the previous direction. 

Now a category is a right (resp. left) $\pi$-regular if and only if every object is right (resp. left) $\pi$-regular. Then since we have shown that for an additive category, an object is right $\pi$-regular if and only if it is left $\pi$-regular, it follows that an additive category is right $\pi$-regular if and only if it is left $\pi$-regular, as desired. 
\end{proof}

The above theorem immediately gives us an equivalent characterization of being Drazin for an additive category: 

\begin{corollary} In an additive category $\mathbb{X}$, an object $A$ is Drazin if and only if $A$ is right (or left) $\pi$-regular. So $\mathbb{X}$ is Drazin if and only if $\mathbb{X}$ is right (or left) $\pi$-regular.
\end{corollary}

\subsection{Negatives and Scalar Multiplication} In an additive category, we can ask how the Drazin inverse interacts with the sum and the negation. As for usual inverses, the Drazin inverse does not necessarily behave well with sums. Indeed, the sum of Drazin endomorphisms $x$ and $y$ is not necessarily Drazin, and even if $x+y$ was Drazin, then the Drazin inverse $(x+y)^D$ is not necessarily the sum of the Drazin inverses $x^D + y^D$. On the other hand, the negation of a Drazin endomorphism is Drazin with Drazin inverse the negation of the original Drazin inverse. 

\begin{lemma} In an additive catgory $\mathbb{X}$, if $x: A \to A$ is Drazin, then $-x: A \to A$ is Drazin where $(-x)^D = -x^D$ and $\mathsf{ind}(-x) = \mathsf{ind}(x)$. 
\end{lemma}
\begin{proof} We show that $(-x)^D = -x^D$ satisfies the three axioms of a Drazin inverse. 
    \begin{enumerate}[{\bf [D.1]}]
\item Suppose that $\mathsf{ind}(x) =k$, then applying \textbf{[D.1]} for $x$ we compute that: 
\begin{gather*}
  (-x)^{k+1} (-x)^D = (-x)^{k+1} (-x^D) = (-1)^{k+2} x^{k+1} x^D = (-1)^k x^k = (-x)^k
\end{gather*}
So ${\sf ind}(-x) \leq k$.
\item Applying \textbf{[D.2]} for $x$ we compute that: 
\[(-x)^D (-x) (-x)^D = (-x^D) (-x) (-x^D) = (-1)^3 x^D x x^D = - x^D = (-x)^D \]  
\item Applying \textbf{[D.3]} for $x$ we compute that: 
\[ (-x)^D (-x) = (-x^D) (-x) = (-1)^2 x^D x = (-1)^2 x x^D = (-x) (-x)^D \]
\end{enumerate}
So $-x$ is Drazin with Drazin inverse $(-x)^D = -x^D$. Now if $\mathsf{ind}(-x) = j < k= \mathsf{ind}(x)$, then the above calculation would tells us that \textbf{[D.1]} for $-(-x) = x$ holds for $j$, which contradicts that $k$ is the Drazin index. So $\mathsf{ind}(-x) = \mathsf{ind}(x)$. 
\end{proof}

Similarly, if we happen to be in a setting where we can scalar multiply maps by rationals $\frac{p}{q} \in \mathbb{Q}$, then if $x$ is Drazin so is $\frac{p}{q} x$ where $(\frac{p}{q} x)^D = \frac{q}{p} x^D$ when $\frac{p}{q} \neq 0$. In particular, in such a setting, for $m \geq 1$, we would have that if $x$ is Drazin, then $mx = x + x + \hdots + x$ is also Drazin with $(mx)^D = \frac{1}{m} x^D$. More generally, if $R$ is a commutative ring, then in a category $\mathbb{X}$ which is enriched over $R$-modules, for any unit $u \in R$, if $x$ is Drazin then $ux$ is Drazin where $(ux)^D = u^{-1} x^D$. The case of scalar multiplying by $0$ is discussed in the next section.  

\subsection{Nilpotents} Recall that in an additive category, an endomorphism $n: A \to A$ is said to be \textbf{nilpotent} if there is a $k \in \mathbb{N}$ such that $n^{k} =0$, and the smallest such $k$ is called the \textbf{nilpotent index} of $n$. It turns out that nilpotent endomorphisms are precisely the Drazin endomorphisms whose Drazin inverse is zero. In this case, the Drazin index and the nilpotent index coincide. 

\begin{lemma}\label{lemma:nilpotent-drazin} In any additive category $\mathbb{X}$, $n: A \to A$ is nilpotent if and only if $n$ is Drazin with $n^D = 0$. In particular, the zero morphism $0: A \to A$ is Drazin and its own Drazin inverse, $0^D = 0$.
\end{lemma} 
\begin{proof} For ($\Rightarrow$), suppose that $n$ is nilpotent with nilpotent index $k$, so $n^k =0$. We must show that $n^D =0$ satisfies the three axioms of a Drazin inverse. However, note that \textbf{[D.2]} and \textbf{[D.3]} trivially hold since $0n0 =0$ and $n0=0=0n$. On the other hand, \textbf{[D.1]} also holds since $n^{k+1} 0 = 0 = n^k$. So $n$ is Drazin with Drazin inverse $0$. Now we know that $\mathsf{ind}(n) \leq k$. But if $\mathsf{ind}(n) < k$, then there would be a $j < k$ such that $n^{j} = n^{j+1} 0 = 0$, which contradicts $k$ being the nilpotent index. So we have that $\mathsf{ind}(n)=k$ as desired. 

For ($\Leftarrow$), suppose that $n$ is Drazin with Drazin inverse $n^D = 0$ and $\mathsf{ind}(n)=k$. Then \textbf{[D.1]} tells us that $n^k = n^{k+1} 0 = 0$. So $n$ is nilpotent. If there was a $j < k$ such that $n^j =0$, then we would also have that $n^{j} = n^{j+1} 0 = 0$, which contradicts $k$ being the Drazin index. So we have that the nilpotent index of $n$ is its Drazin index. 
\end{proof}

\subsection{Core-Nilpotent Decomposition} \label{ssec:core-nilpotent} For matrices, an important concept in relation to the Drazin inverse is the notion of the \emph{core-nilpotent decomposition} \cite[Thm 2.2.21]{mitra2010matrix}. The \emph{core} \cite[Def 7.3.1]{campbell2009generalized} of a matrix is defined as the Drazin inverse of its Drazin inverse, while its \emph{nilpotent part} \cite[Def 7.3.2]{campbell2009generalized} is the matrix minus its core. The \emph{core-nilpotent decomposition} is the statement that a matrix is equal to the sum of its core and nilpotent parts. Drazin later generalized this concept for elements in a ring \cite[Sec 1]{drazin2001generalizations}, and showed that having a core-nilpotent decomposition was equivalent to being Drazin \cite[Prop 3.1 \& 3.2]{drazin2001generalizations}. Here, we discuss core-nilpotent decompositions of endomorphisms in an additive category, in other words, core-nilpotent decomposition in the ring $\mathbb{X}(A,A)$. 

\begin{definition} In an additive category $\mathbb{X}$, a \textbf{core-nilpotent decomposition} of ${x: A \to A}$ is a pair $(c,n)$ of endomorphisms $c: A \to A$ and $n: A \to A$ such that: 
    \begin{enumerate}[{\bf [CND.1]}]
\item $c$ is Drazin with $\mathsf{ind}(c) \leq 1$ (so $c$ has a group inverse); 
\item $n$ is nilpotent with nilpotent index $k \in \mathbb{N}$ (so that $n^{k} = 0$);
\item $cn=0=nc$;
\item $x = c + n$. 
\end{enumerate}
\end{definition}

Before we revisit the proof that having a core-nilpotent decomposition is equivalent to being Drazin, it will be useful to explicitly define the core and nilpotent parts of a Drazin endomorphism. 

\begin{definition} In any additive category $\mathbb{X}$, for a Drazin endomorphism $x: A \to A$, 
\begin{enumerate}[(i)]
\item The \textbf{core}\footnote{Note that by Lemma \ref{lem:Drazin-inverse1}.(\ref{Drazin-inverse-inverse}), the core can in fact be defined in an arbitrary category.} of $x$ is the endomorphism $c_x: A \to A$ defined by $c_x \colon = x^{DD} = x x^D x$. 
\item The \textbf{nilpotent part} of $x$ is the endomorphism $n_x \colon = x - c_x: A \to A$.
\end{enumerate}
\end{definition}

\begin{theorem}\label{thm:cnd} \cite[Prop 3.1, 3.2, \& 3.3]{drazin2001generalizations} In an additive category $\mathbb{X}$, $x: A \to A$ is Drazin if and only if $x$ has a core-nilpotent decomposition. Moreover, when $x$ is Drazin, $(c_x,n_x)$ is its unique core-nilpotent decomposition. 
\end{theorem}
\begin{proof} For ($\Rightarrow$), Suppose that $x$ is Drazin with index $k$. We show that $(c_x,n_x)$ satisfies the four axioms of a core-nilpotent decomposition. 
\begin{enumerate}[{\bf [CND.1]}]
\item This is the statement of Lemma \ref{lem:Drazin-inverse1}.(\ref{Drazin-inverse-inverse}). 
\item Using the binomial theorem (which recall holds in any additive category), \textbf{[D.3]}, and Lemma \ref{lemma:Drazin-basic}.(\ref{lemma:Drazin-basic.3}), we compute that: 
\begin{gather*}
    n_x^{k} = (x - c_x)^{k} = \sum\limits^{k}_{i=0} (-1)^i \binom{k}{i} x^{k-i} c_x^i =  \sum\limits^{k}_{i=0} (-1)^i \binom{k}{i} x^{k-i} (x x^D x)^i \\
    = \sum\limits^{k}_{i=0} (-1)^i \binom{k}{i} x^{k-i} x^{2i} {x^D}^i =  \sum\limits^{k+1}_{i=0} (-1)^i \binom{k}{i} x^{k-i+2i} {x^D}^i \\
    = \sum\limits^{k}_{i=0} (-1)^i\binom{k}{i} x^{k+i} {x^D}^i = \sum\limits^{k}_{i=0} (-1)^i\binom{k}{i} x^{k} = (1-1)^{k} x^{k} = 0
\end{gather*}
So $n_x$ has nilpotent index less than or equal to $k$.
\item Using \textbf{[D.2]} and \textbf{[D.3]}, we have:
\begin{gather*}
    c_x n_x = c_x(x - c_x) = c_x x - c_x c_x = x x^D x x - x x^D x x x^D x\\
    = x^3 x^D - x^3 x^D x x^D =  x^3 x^D - x^3 x^D = 0
\end{gather*}
So $c_x n_x = 0$, and similarly  $n_x c_x = 0$. 
\item From the definition of $n_x$ it follows that $x = c_x + (x - c_x) = c_x + n_x$. 
\end{enumerate}

For ($\Leftarrow$),  Suppose that $x$ has a core-nilpotent decomposition $(c,n)$. Let $k \geq 1$ be the nilpotent index of $n$ so that $n^{k} = 0$. Using {\bf [CND.1]}, we set $x^D \colon = c^D$ and show that $x^D$ satisfies the three Drazin inverse axioms. 
\begin{enumerate}[{\bf [D.1]}]
\item First note that, for $m > 0$, $x^m = (c + n)^m = c^m + n^m$ as the cross-terms vanish using {\bf [CND.3]}. So, in particular, for $j \geq k$, we have that $x^{j} = c^{j}$ using {\bf [CND.2]}. Now since $\mathsf{ind}(c) \leq 1 \leq k$, by Lemma \ref{lemma:Drazin-basic}.(\ref{lemma:Drazin-basic.1}) we compute that: 
\begin{gather*}
    x^{k+1} x^D = (c+n)^{k+1} c^D = (c^{k+1} + n^{k+1}) c^D = c^{k+1} c^D \\ = c^{k} = c^{k} + n^{k} = (c+n)^{k} = x^{k} 
\end{gather*}
So ${\sf ind}(x) \leq k$.
\item By {\bf [CND.3]}, we have that $cn = 0 = 0c$. So by Lemma \ref{Drazin-commuting}, we have that $c^Dn = 0c^D$, so $c^D n =0$. Using this and \textbf{[D.2]}, we compute that: 
\[x^D x x^D = c^D (c + n) c^D = c^D c c^D + c^D n c^D = c^D + 0 = c^D = x^D \]  
\item By {\bf [CND.3]}, we have that $cn= nc$. So by Lemma \ref{Drazin-commuting}, we have that $c^Dn = nc^D$. Using this and \textbf{[D.3]}, we have: 
\[x x^D = (c+n)c^D = cc^D + nc^D = c^Dc +c^Dn = c^D(c+n) = x^D x\]
\end{enumerate}
So $x$ is Drazin with Drazin inverse $x^D \colon = c^D$ and ${\sf ind}(x) \leq k$.  

It remains to show that a core-nilpotent decomposition is unique. So let $x$ be Drazin. Suppose that $x$ has another core-nilpotent decomposition $(c,n)$. Then we have that $c_x^D = x^D = c^D$. Since $\mathsf{ind}(c) \leq 1$ and $\mathsf{ind}(c_x) \leq 1$, by Lemma \ref{lem:Drazin-inverse1}.(\ref{cor:ginverse-ginverse}) we have that $c = (c^D)^D = (c^D_x)^D = c_x$, so $c = c_x$. By \textbf{[CND.4]}, $c + n = x = c_x + n_x$, so subtracting $c=c_x$ from both sides gives us $n=n_x$. So we conclude that $(c_x, n_x)$ is the unique core-nilpotent decomposition of $x$.
\end{proof} 

Therefore, we obtain another characterization of when an additive category is Drazin: 

\begin{corollary} An additive category $\mathbb{X}$ is Drazin if and only if every endomorphism has a core-nilpotent decomposition.
\end{corollary}

\subsection{Kernel-cokernel coincidence} \label{ssec:kernel-cokernel} In \cite[Thm 2]{robinson1987generalized}, Robinson and Puystjens provide conditions for being Drazin in an additive category in terms of kernels and cokernels. Here we revisit their result and show how their setup is in fact equivalent to being Drazin such that the \emph{complement} of the induced idempotent is split.  

In an additive category, the \textbf{complement} of an idempotent $e: A \to A$ is the endomorphism $e^c \colon=1_A - e: A \to A$. Note that since $e$ is an idempotent, we have that $ee^c =0 = e^c e$. Moreover, the complement of an idempotent is again an idempotent, so we may consider when it splits. In an additive category, we say that an idempotent $e$ is \textbf{complement-split} if $e^c$ is split. In particular, we may consider a Drazin endomorphism whose induced idempotent (from Lemma \ref{lemma:e_x}) is complement-split. 

\begin{definition} In an additive category $\mathbb{X}$, $x: A\to A$ is \textbf{Drazin complement-split} if $x$ is Drazin and its induced idempotent $e_x: A \to A$ is complement-split. 
\end{definition}

We now show an endomorphism $x$ is Drazin complement-split precisely when some power of $x$ has a kernel and a cokernel which are isomorphic. In particular, this means that we can describe the Drazin inverse of $x$ in terms of the splitting of the complement of the induced idempotent. 

\begin{theorem}\label{thm:kernels} In an additive category $\mathbb{X}$, $x: A \to A$ is Drazin complement-split if and only if there is a $k \in \mathbb{N}$ such that $x^{k+1}$ has a kernel and cokernel:
\[ \xymatrixcolsep{5pc}\xymatrix{ \mathsf{ker}(x^{k+1}) ~\ar@{>->}[r]^-{\kappa} & A \ar@<1ex>[r]^-{x^{k+1}} \ar@<-1ex>[r]_0 & A \ar@{->>}[r]^-{\lambda} & \mathsf{coker}(x^{k+1}) } \]
such that $\kappa \lambda$ is an isomorphism and $x^{k+1}: (A, e^c_{\lambda,\kappa}) \to (A, e^c_{\lambda,\kappa})$ is an isomorphism in $\mathsf{Split}(\mathbb{X})$, where $e^c_{\lambda,\kappa}$ is the complement of the idempotent $e_{\lambda,\kappa} = \lambda (\kappa\lambda)^{-1} \kappa$. 
\end{theorem}
\begin{proof} For ($\Rightarrow$), suppose that $x$ is Drazin complement-split where $e^c_x$ splits via $r^c: A \to K$ and $s^c: K \to A$. Let $\mathsf{ind}(x) = k$. We first show that $s^c: K \to A$ is the kernel of $x^{k+1}$ and that  $r^c: A \to K$ is the cokernel of $x^{k+2}$. Starting with the kernel, we first compute using Lemma \ref{lemma:Drazin-basic}.(\ref{lemma:Drazin-basic.1}) that: 
\[ x^{k+1} e^c_x = x^{k+1}(1_A- e_x) = x^{k+2} - x^{k+1}e_x =   x^{k+1} - x^{k+1}x x^D = x^{k+1} - x^{k+1+1} x^D = x^{k+1} - x^{k+1} = 0 \]
So $x^{k+1} e^c_x = 0$. Since $e^c_x r^c = s^c$, post-composing both since by $r^c$ gives us that $x^{k+1} s^c =0$. Now suppose that we had a map $f: B \to A$ such that $f x^{k+1} =0$ as well. Now observe that using idempotency and \textbf{[D.3]}, we have that: 
\[ f e_x = f (e_x)^{k+1} = f (x x^D)^{k+1} = f x^{k+2} (x^D)^{k+1} = 0 (x^D)^{k+1} = 0   \]
Since $f e_x =0$, it follows that $fe^c_x = f$. Then define $f^\sharp: B \to K$ as $f^\sharp \colon= f r^c$. Since $r^c s^c = e^c_x$, it follows that $f^\sharp s^c =f$. Since $s^c$ is monic (since $s^c r^c =1_K$), it follows that $f^\sharp$ must be unique. So we conclude that $s^c: K \to A$ is the kernel of $x^{k+1}$ as desired. By similar arguments, one can also show that $r^c: A \to K$ is the cokernel of $x^{k+1}$. Now setting $\kappa = s^c$ and $\lambda= r^c$, trivially $\kappa \lambda = 1_K$ is an isomorphism. Moreover:
\[e_{\lambda,\kappa} =  \lambda (\kappa\lambda)^{-1} \kappa = \lambda \kappa = r^c s^c = e^c_x\] 
So $e_{\lambda,\kappa} = e^c_x$ and $e^c_{\lambda,\kappa} = e_x$. Now as was shown in the proof of Theorem \ref{characterization-by-iteration}, we have that $x^{k+1}: (A, e_x) \to (A, e_x)$ is an isomorphism in the idempotent splitting, or in other words, $x^{k+1}: (A, e^c_{\lambda,\kappa}) \to (A, e^c_{\lambda,\kappa})$ is an isomorphism in $\mathsf{Split}(\mathbb{X})$ as desired. 

For ($\Leftarrow$), By Theorem \ref{characterization-by-iteration}, since $x^{k+1}: (A, e^c_{\lambda,\kappa}) \to (A, e^c_{\lambda,\kappa})$ is an isomorphism in $\mathsf{Split}(\mathbb{X})$, we get that $x$ is Drazin. So it remains to explain why its induced idempotent is complement-split. Recall that in the proof of Theorem \ref{characterization-by-iteration}, we had showed that the Drazin inverse of $x$ is $x^D \colon = v x^k = x^k v$, where $v: (A, e^c_{\lambda,\kappa}) \to (A, e^c_{\lambda,\kappa})$ is the inverse of $x^{k+1}: (A, e^c_{\lambda,\kappa}) \to (A, e^c_{\lambda,\kappa})$ in $\mathsf{Split}(\mathbb{X})$. So its induced idempotent is $e_x = x^{k+1} v$. However $x^{k+1} v= 1_{(A, e^c_{\lambda,\kappa})} = e^c_{\lambda,\kappa}$. So $e_x = e^c_{\lambda,\kappa}$ and thus $e_x^c = e_{\lambda,\kappa}$. However, clearly $e^c_x = e_{\lambda,\kappa}$ splits via $\lambda: A \to \mathsf{coker}(x^{k+1})$ and $(\kappa\lambda)^{-1} \kappa: \mathsf{coker}(x^{k+1}) \to A$. So we conclude that $x$ is Drazin complement-split. 
\end{proof}

In \cite[Thm 2]{robinson1987generalized}, rather than asking that $x^{k+1}$ is an isomorphism in the idempotent splitting, Robinson and Puystjens ask that $x^{k+1}+e_{\lambda,\kappa}$ be an actual isomorphism. However, the following lemma shows that these statements are equivalent. 

\begin{lemma} \label{lemma:iso-comp} In an additive category $\mathbb{X}$, $f: (A,e) \to (A,e)$ is an isomorphism in $\mathsf{Split}(\mathbb{X})$ if and only if $f + e^c$ is an isomorphism in $\mathbb{X}$. 
\end{lemma}
\begin{proof} Assume that $f: (A,e) \to (A,e)$ is a map in $\mathsf{Split}(\mathbb{X})$, which recall means that $ef = f =fe$. So for ($\Rightarrow$), suppose that $f: (A,e) \to (A,e)$ is an isomorphism in $\mathsf{Split}(\mathbb{X})$ with inverse $g: (A,e) \to (A,e)$. Recall that this means that: $eg=g=ge$ and also that $fg=e=gf$. Now note that this implies that $fe^c=0=fe^c$ and $ge^c =0 = e^cg$. Then define $(f + e^c)^{-1} \colon = g + e^c$. So we compute that:
\[ (f +e^c)(f + e^c)^{-1} = (f + e^c) (g + e^c) = fg + fe^c + e^c g + e^c e^c = e + 0 + 0 + e^c = e + e^c = 1_A \]
\[ (f + e^c)^{-1}(f +e^c) = (g + e^c) (f + e^c) = gf + ge^c + e^c f + e^c e^c = e + 0 + 0 + e^c = e + e^c = 1_A  \]
So $f + e^c$ is an isomorphism.

For  ($\Leftarrow$), suppose that $f + e^c$ is an isomorphism. Our goal is to find the inverse of $f: (A,e) \to (A,e)$ in $\mathsf{Split}(\mathbb{X})$. So consider the map $e(f + e^c)^{-1})e$. We clearly have that $e(f + e^c)^{-1}e: (A,e) \to (A,e)$ is a map in $\mathsf{Split}(\mathbb{X})$, so it remains to show that it is the inverse of $f: (A,e) \to (A,e)$. So we compute: 
\begin{gather*}
    f e(f + e^c)^{-1})e = e f (f + e^c)^{-1})e = (ef + 0) (f + e^c)^{-1})e \\
    =  (ef + ee^c) (f + e^c)^{-1})e = e (f + e^c)  (f + e^c)^{-1})e = e e = e 
\end{gather*}
So $f e(f + e^c)^{-1})e = e$, and similarly we can show that $e(f + e^c)^{-1})e f = e$ as well. 
\end{proof}

The reader may notice that in the statement of the lemma, we started with a map in the idempotent splitting, while in \cite[Thm 2]{robinson1987generalized} this is not stated. However, this is not an issue. Indeed, if $x^{k+1}$ has a kernel and cokernel as above, such that $\kappa \lambda$ is an isomorphism, then we can compute that: 
\[ x^{k+1}e^c_{\lambda,\kappa} = x^{k+1}(1_A - \lambda (\kappa\lambda)^{-1} \kappa) = x^{k+1} - x^{k+1}  \lambda (\kappa\lambda)^{-1} \kappa = x^{k+1} - 0 = x^{k+1}  \]
So $x^{k+1}e^c_{\lambda,\kappa} = x^{k+1}$ and similarly $x^{k+1}e^c_{\lambda,\kappa} = x^{k+1}$. So $x^{k+1}: (A,e^c_{\lambda,\kappa}) \to (A,e^c_{\lambda,\kappa})$ is a map in the idempotent splitting. Then under the other conditions of the above theorem or \cite[Thm 2]{robinson1987generalized}, by applying the above lemma, the statement that $x^{k+1}: (A, e^c_{\lambda,\kappa}) \to (A, e^c_{\lambda,\kappa})$ is an isomorphism is equivalent to saying that $x^{k+1} + e_{\lambda,\kappa}$ is an isomorphism. Finally, we also recapture the formula for the Drazin inverse in \cite[Thm 2]{robinson1987generalized} in terms of the complement of the induced idempotent (recalling that we showed that $e_{\lambda,\kappa}=e^c_x$ and $e^c_{\lambda,\kappa} = e_x$): 

\begin{corollary} In an additive category $\mathbb{X}$, if $x: A \to A$ is Drazin complement-split and $\mathsf{ind}(x) = k$, then $x^D = x^k (x^{k+1} + e^c_x)^{-1} = (x^{k+1} + e^c_x)^{-1} x^k$. 
\end{corollary}
\begin{proof} As shown in the proof of Theorem \ref{characterization-by-iteration}, $x^{k+1}: (A, e_x) \to (A, e_x)$ is an isomorphism in $\mathsf{Split}(\mathbb{X})$. In the same proof, we also showed that $x^D = v x^k = x^k v$, where $v: (A, e_x) \to (A, e_x)$ is the inverse of $x^{k+1}: (A, e_x) \to (A, e_x)$ in $\mathsf{Split}(\mathbb{X})$. However, by Lemma \ref{lemma:iso-comp}, we have that $x^{k+1} + e^c_x$ is an isomorphism and that $v = e_x (x^{k+1} + e^c_x ) e_x$. Therefore, we get that $x^D = x^k (x^{k+1} + e^c_x)^{-1} = (x^{k+1} + e^c_x)^{-1} x^k$ as desired. 
\end{proof}

\subsection{Fitting Decomposition} \label{ssec:fitting-decomposition} In Section \ref{ex:complex}, we reviewed how to compute the Drazin inverse of a matrix, we first decomposed that matrix as a block diagonal matrix, involving an invertible matrix and a nilpotent matrix, conjugated by another invertible matrix. This approach generalizes to finite-dimensional vector spaces thanks to \textbf{Fitting's Decomposition Theorem} \cite[Lemma 2.4]{leinstercounting}. Here, we generalize this decomposition to additive categories with \emph{finite biproducts}, and call it a \emph{Fitting decomposition}. Moreover, we show that having a Fitting decomposition is equivalent to being both Drazin split and Drazin complement-split, which we call being \emph{Drazin decomposable}. 

\begin{definition} In an additive category $\mathbb{X}$, $x: A\to A$ is \textbf{Drazin decomposable} if $x$ is both Drazin split and Drazin complement-split. 
\end{definition}

We now work in an additive category with finite biproducts: we denote the biproduct by $\oplus$. To help understand Fitting decompositions, it will be useful to recall the matrix representation for maps between biproducts. A map of type $F: A_1 \oplus \hdots \oplus A_n \to B_1 \oplus \hdots \oplus B_m$ is uniquely determined by a family of maps $f_{i,j}: A_i \to B_j$. As such, $F$ can be represented as an $n\times m$ matrix:
\[ F \colon = \begin{bmatrix} f_{1,1} & f_{1,2} & \hdots & f_{1,m} \\
f_{2,1} & f_{2,2} & \hdots & f_{2,m} \\
\vdots & \vdots & \ddots & \vdots \\
f_{n,1} & f_{n,2} & \hdots & f_{n,m}  
\end{bmatrix}
\]
Moreover, composition corresponds to matrix multiplication, and identities correspond to the identity matrix. See \cite[Sec 2.2.4]{HeunenVicary} for details of the matrix representation of biproducts. 

We now define a Fitting decomposition of an endomorphism by turning \cite[Thm 7.2.1]{campbell2009generalized} into a definition:

\begin{definition} \label{Fitting-decomposition} In an additive category $\mathbb{X}$ with finite biproducts, a \textbf{Fitting decomposition} of $x: A \to A$ is a triple $(p, \alpha, \eta)$ consisting of an isomorphism $p: A \to I \oplus K$, an isomorphism $\alpha: I \to I$, and a nilpotent endomorphism $\eta: K \to K$ such that the following equality holds: 
\[ x = p \begin{bmatrix} \alpha & 0 \\ 0 & \eta \end{bmatrix} p^{-1} \]
\end{definition}

\begin{theorem}\label{thm:iso-nil} In an additive category $\mathbb{X}$ with finite biproducts, $x: A \to A$ is Drazin decomposable if and only if $x$ has a Fitting decomposition. 
\end{theorem}
\begin{proof} For ($\Rightarrow$), suppose that $x$ is Drazin decomposable, where $e_x$ splits via $r: A \to I$ and $s: I \to A$, and $e^c_x$ splits via $r^c: A \to K$ and $s^c: K \to A$. We first show that $p = \begin{bmatrix}
        r & r^c
    \end{bmatrix}:  A \to I \oplus K$ is an isomorphism. So define $p^{-1}: I \oplus K \to A$ as follows: 
\begin{align*}
    p^{-1} = \begin{bmatrix}
        s \\ s^c
    \end{bmatrix} 
\end{align*}
Now using that $e_x = rs$ and $e_x^c = r^c s^c$, we first compute that: 
\[  pp^{-1} = \begin{bmatrix}
        r & r^c
    \end{bmatrix} \begin{bmatrix}
        s \\ s^c
    \end{bmatrix} = rs + r^c s^c = e_x + e_x^c = e_x + 1_A - e_x = 1_A \]
    On the other hand, since $ee^c = 0 = e^c e$, it follows that $s r^c=0$ and $s^c r=0$. So by using that $sr=1_I$ and $s^c r^c = 1_K$, we compute that: 
    \[  p^{-1} p =\begin{bmatrix}
        s \\ s^c
    \end{bmatrix} \begin{bmatrix}
        r & r^c
    \end{bmatrix} =  \begin{bmatrix}
        sr & s r^c \\ s^c r & s^c r^c
    \end{bmatrix} =  \begin{bmatrix}
        1_I & 0 \\ 0 & 1_K
    \end{bmatrix} = 1_{I \oplus K} \]
    So $p$ is indeed an isomorphism. Now define $\alpha: I \to I$ and $\eta: K \to K$ as $\alpha = s x r$ and $\eta = s^c n_x r^c$ respectively, where recall that $n_x$ is the nilpotent of the core-nilpotent decomposition of $x$. In the proof of Lemma \ref{lemma:alpha1} we have already shown that $\alpha$ is an isomorphism as well. On the other hand if $\mathsf{ind}(x) =k$, then in the proof of Theorem \ref{thm:cnd} we showed that $n_x$ was indeed nilpotent and that $n^{k+1}_x =0$. Moreover, recall that in the proof of Theorem \ref{thm:cnd} we also showed that $x^D = c^D_x$ and $n_x c^D_x =0$. So it follows that $e_x n_x =0$ and similarly we also have $n_x e_x =0$. Therefore, we get that $e^c_x n_x = n_x = n_x e^c_x$. Then using this and $e_x^c = r^c s^c$, we compute that: 
    \[ \eta^{k+1} =  s^c n_x r^c  s^c n_x r^c \hdots  s^c n_x r^c = s^c n_x e^c_x n_x \hdots e^c n_x r^c = s^c n^{k+1}_x r^c = s^c 0 r^c = 0 \]
    So $\eta$ is nilpotent as desired. Now recall that by Lemma \ref{lemma:alpha1}, it follows that $r \alpha s = x e_x = x x^D x$. In other words, $r \alpha s = c_x$. On the other hand, from $e^c_x n_x = n_x = n_x e^c_x$, it also follows that $r^c \eta s^c = n_x$. Therefore, by the core-nilpotent decomposition, we finally compute that 
    \[ p \begin{bmatrix} \alpha & 0 \\ 0 & \eta \end{bmatrix} p^{-1} = \begin{bmatrix}
        r & r^c
    \end{bmatrix} \begin{bmatrix} \alpha & 0 \\ 0 & \eta \end{bmatrix} \begin{bmatrix}
        s \\ s^c
    \end{bmatrix} =  \begin{bmatrix}
        r \alpha & r^c \eta
    \end{bmatrix}  \begin{bmatrix}
        s \\ s^c
    \end{bmatrix} = r \alpha s + r^c \eta s^c = c_x + n_x = x  \]
    and conclude that that $(p, \alpha, \eta)$ is a Fitting's decomposition of $x$ as desired.  

    For $(\Leftarrow)$, we first show that $x$ is Drazin by showing that $x^D$ satisfies the three Drazin inverse axioms.
     \begin{enumerate}[{\bf [D.1]}]
\item First note that for all $m \in \mathbb{N}$, we have that: 
\[ x^m = p \begin{bmatrix} \alpha^m & 0 \\ 0 & \eta^m \end{bmatrix} p^{-1}  \]
Now let $k$ be the smallest natural number such that $\eta^{k+1} =0$. Then we compute: 
\begin{gather*}
 x^{k+1+1} x^D =  p \begin{bmatrix} \alpha^{k+1+1} & 0 \\ 0 & \eta^{k+1+1} \end{bmatrix} p^{-1} p \begin{bmatrix} \alpha^{-1} & 0 \\ 0 & 0 \end{bmatrix} p^{-1} =  p \begin{bmatrix} \alpha^{k+1} \alpha & 0 \\ 0 & 0 \end{bmatrix} \begin{bmatrix} \alpha^{-1} & 0 \\ 0 & 0 \end{bmatrix} p^{-1} \\
=  p \begin{bmatrix} \alpha^{k+1} \alpha \alpha^{-1} & 0 \\ 0 & 0 \end{bmatrix} p^{-1} =  p \begin{bmatrix} \alpha^{k+1} & 0 \\ 0 & \eta^{k+1} \end{bmatrix} p^{-1} = x^{k+1}
\end{gather*}
\item We easily compute that: 
\begin{gather*} x^D x x^D =  p \begin{bmatrix} \alpha^{-1} & 0 \\ 0 & 0 \end{bmatrix} p^{-1} p \begin{bmatrix} \alpha & 0 \\ 0 & \eta \end{bmatrix} p^{-1}  p \begin{bmatrix} \alpha^{-1} & 0 \\ 0 & 0 \end{bmatrix} p^{-1} \\ = p \begin{bmatrix} \alpha^{-1} & 0 \\ 0 & 0 \end{bmatrix} \begin{bmatrix} \alpha & 0 \\ 0 & \eta \end{bmatrix}  \begin{bmatrix} \alpha^{-1} & 0 \\ 0 & 0 \end{bmatrix} p^{-1} = p \begin{bmatrix} \alpha^{-1} \alpha \alpha^{-1} & 0 \\ 0 & 0 \eta 0 \end{bmatrix} p^{-1} = p \begin{bmatrix} \alpha^{-1} & 0 \\ 0 & 0 \end{bmatrix} p^{-1} = x^D \end{gather*}
\item We easily compute that: 
\begin{gather*} x x^D =   p \begin{bmatrix} \alpha & 0 \\ 0 & \eta \end{bmatrix} p^{-1}  p \begin{bmatrix} \alpha^{-1} & 0 \\ 0 & 0 \end{bmatrix} p^{-1}  = p \begin{bmatrix} \alpha & 0 \\ 0 & \eta \end{bmatrix}  \begin{bmatrix} \alpha^{-1} & 0 \\ 0 & 0 \end{bmatrix} p^{-1}= p \begin{bmatrix} \alpha^{-1} \alpha  & 0 \\ 0 & \eta 0 \end{bmatrix} p^{-1} \\
 = p \begin{bmatrix}  \alpha \alpha^{-1} & 0 \\ 0 & 0 \eta  \end{bmatrix} p^{-1} =  p \begin{bmatrix} \alpha^{-1} & 0 \\ 0 & 0 \end{bmatrix}  \begin{bmatrix} \alpha & 0 \\ 0 & \eta \end{bmatrix} p^{-1} =    p \begin{bmatrix} \alpha^{-1} & 0 \\ 0 & 0 \end{bmatrix} p^{-1} p \begin{bmatrix} \alpha & 0 \\ 0 & \eta \end{bmatrix} p^{-1} = x^D x
\end{gather*}
\end{enumerate}
So $x$ is Drazin. Next, we compute its induced idempotent: 
\[ e_x = x x^D =  p \begin{bmatrix} 1_I & 0 \\ 0 & 0 \end{bmatrix} p^{-1} = \begin{bmatrix}
        r & r^c
    \end{bmatrix} \begin{bmatrix} 1_I & 0 \\ 0 & 0 \end{bmatrix} \begin{bmatrix}
        s \\ s^c
    \end{bmatrix} =  \begin{bmatrix}
        r & 0
    \end{bmatrix} \begin{bmatrix}
        s \\ s^c
    \end{bmatrix} = rs   \]
So $e_x =rs$. We compute its complement to be: 
\[ e^c_x = 1_A - e_x = p p^{-1} =   \begin{bmatrix}
        r & r^c
    \end{bmatrix} \begin{bmatrix}
        s \\ s^c
    \end{bmatrix} - e_x = rs + r^c s^c - e_x = e_x + r^c s^c - e_x = r^c s^c \]
 So $e^c_x = r^c s^c$. Now note that we also have that: 
  \[  \begin{bmatrix}
        1_I & 0 \\ 0 & 1_K
    \end{bmatrix} = 1_{I \oplus K}=  p^{-1} p =\begin{bmatrix}
        s \\ s^c
    \end{bmatrix} \begin{bmatrix}
        r & r^c
    \end{bmatrix} =  \begin{bmatrix}
        sr & s r^c \\ s^c r & s^c r^c
    \end{bmatrix}  \]
 In particular, this gives us that $sr = 1_I$ and $s^c r^c = 1_K$. So both $e_x$ and $e^c_x$ split, and we conclude that $x$ is Drazin decomposable as desired. 
\end{proof} 

\begin{corollary} In an additive category $\mathbb{X}$ with finite biproducts, if $x: A \to A$ has a Fitting decomposition $(p, \alpha, \eta)$, and is therefore Drazin, then the core and nilpotent-part of $x$ are determined by:
\begin{align*}
    c_x \colon =  p \begin{bmatrix} \alpha & 0 \\ 0 & 0 \end{bmatrix} p^{-1} &&  n_x \colon = p \begin{bmatrix} 0 & 0 \\ 0 & \eta \end{bmatrix} p^{-1}
\end{align*}
\end{corollary}
\begin{proof} We compute: 
\begin{gather*} c_x = x x^D x =  p \begin{bmatrix} \alpha & 0 \\ 0 & \eta \end{bmatrix} p^{-1} p \begin{bmatrix} \alpha^{-1} & 0 \\ 0 & 0 \end{bmatrix} p^{-1}  p \begin{bmatrix} \alpha & 0 \\ 0 & \eta \end{bmatrix} p^{-1} = p \begin{bmatrix} \alpha & 0 \\ 0 & \eta \end{bmatrix} \begin{bmatrix} \alpha^{-1} & 0 \\ 0 & 0 \end{bmatrix} \begin{bmatrix} \alpha & 0 \\ 0 & \eta \end{bmatrix} p^{-1}\\
=  p \begin{bmatrix} \alpha \alpha^{-1} \alpha & 0 \\ 0 & \eta 0 \eta \end{bmatrix} p^{-1} = p \begin{bmatrix} \alpha & 0 \\ 0 & 0\end{bmatrix} p^{-1} \end{gather*}
So the desired equality holds. By the definition of $n_x$, the other equality holds as well. 
\end{proof}

Theorem \ref{thm:iso-nil} also implies that if an additive category with finite biproducts is also idempotent complete, then an endomorphism is Drazin if and only if it has a Fitting decomposition. Therefore: 

\begin{corollary} An additive category $\mathbb{X}$ with finite biproducts is Drazin complete if and only if every endomorphism has a Fitting decomposition.
\end{corollary}


\subsection{Image-Kernel Decomposition} \label{ssec:image-kernel} In Section \ref{sec:modules}, we mentioned that for a ring $R$ and an $R$-module $M$, an $R$-linear endomorphism $f: M \to M$ was Drazin if and only if $f$ had a ``Fitting decomposition'' in the sense that $M = \mathsf{im}(f^k) \oplus \mathsf{ker}(f^k)$ for some $k \geq 1$ \cite[Lemma 2.1.(4)]{wang2017class}. We generalize this sort of decomposition in an Abelian category and show that it is indeed equivalent to being Drazin. In an Abelian category $\mathbb{X}$, denote the kernel and image of an endomorphism $x: A \to A$, as follows: 
\begin{align*}
\xymatrixcolsep{5pc}\xymatrix{ \mathsf{ker}(x) \ar@{>->}[r]^-{\kappa}\ar@/^3pc/[rrr]^-{0} & A \ar[rr]^-{x}  \ar@{->>}[dr]_-{\epsilon}  && A   \\
& & \mathsf{im}(x) \ar@{>->}[ur]_-{\iota} } 
\end{align*}
where recall that $\kappa$ and $\iota$ are monic, and $\epsilon$ is epic. Now define $\psi$ as the canonical map:
\[ \psi = \begin{bmatrix}
    \iota \\
    \kappa
\end{bmatrix}: \mathsf{im}(x) \oplus \mathsf{ker}(x) \to A  \]

\begin{definition} In an Abelian category $\mathbb{X}$, a map $x: A \to A$ has an {\bf image-kernel decomposition} in case the map $\psi$ as defined above is an isomorphism.
\end{definition}

\begin{theorem} \label{thm:image-kernel}
In an Abelian category $\mathbb{X}$, $x: A \to A$ is Drazin if and only if there is some $k \in \mathbb{N}$ such that $x^{k+1}$ has an image-kernel decomposition. 
\end{theorem}

\begin{proof} For ($\Rightarrow$), suppose that $x$ is Drazin with $\mathsf{ind}(x) = k$. Our objective is to show that $\psi: \mathsf{im}(x^{k+1}) \oplus \mathsf{ker}(x^{k+1}) \to A$ is an isomorphism. To construct its inverse, we need to first construct maps $A \to \mathsf{im}(x^{k+1})$ and $A \to \mathsf{ker}(x^{k+1})$. First note that by idempotency and \textbf{[D.3]}, we get that $e_x= e_x^{k+1} = (x^D)^{k+1} x^{k+1}$. So by using the universal property of the image, there is a monic map $\mathsf{im}(e_x) \to \mathsf{im}(x^{k+1})$, which then allows us to build a map $\phi_1: A \to \mathsf{im}(x^{k+1})$ such that $\phi_1 \iota = e_x$. On the other hand, recall that in the proof of Theorem \ref{characterization-by-iteration}, we showed that $e_x x^{k+1} = x^{k+1}$. Then it follows that $e^c_x x^{k+1} = 0$. So by the universal property of the kernel, let $\phi_2: A \to \mathsf{ker}(x^{k+1})$ be the unique map such that $\phi_2 \kappa^\circ_{x^{k+1}} = e^c_x$. Then define $\phi$ as:
\[ \phi = \begin{bmatrix}
    \phi_1 & 
    \phi_2
\end{bmatrix}  : A \to \mathsf{im}(x^{k+1}) \oplus \mathsf{ker}(x^{k+1})\]
So we first compute that: 
\[ \phi \psi = \begin{bmatrix}
    \phi_1 &
    \phi_2
\end{bmatrix} \begin{bmatrix}
    \iota \\
    \kappa
\end{bmatrix}  = \phi_1 \iota + \phi_2 \kappa = e_x + e^c_x = 1_A \]
To prove that $\phi \psi$ is the identity, we must work out what $\iota \phi_1$, $\iota \phi_2$, $\kappa \phi_1$, and $\kappa \phi_2$ are. First, recall again in the proof of Theorem \ref{characterization-by-iteration}, we showed that $x^{k+1} e_x =x^{k+1}$, which also implies that $x^{k+1} e^c_x = 0$. So we first compute that: 
\[ \epsilon \iota \phi_1 \iota = x^{k+1} e_x = x^{k+1} = \epsilon \iota \]
\[ \epsilon \iota \phi_2 \kappa = x^{k+1} e^c_x = 0 = \epsilon 0 \iota  \]
Since $\epsilon$ is epic and $\iota$ is monic, it follows that $\iota \phi_1 = 1_{\mathsf{im}(x^{k+1})}$ and $\iota \phi_2 =0$. Next, using idempotency and \textbf{[D.3]}, $e_x = x^{k+1} (x^D)^{k+1}$, then we have that $\kappa e_x =0$. So we compute: 
\[ \kappa \phi_1 \iota = \kappa e_x = 0 = 0 \iota  \]
\[ \kappa \phi_2 \kappa = \kappa e^c_x = \kappa (1_A -e_x) = \kappa - \kappa e_x = \kappa - 0 = \kappa   \]
Since $\iota$ and $\kappa$ are monic, it follows that $\kappa \phi_1 = 0$ and $\kappa \phi_2 = 1_{\mathsf{ker}(x^{k+1})}$. So finally we get that:   
\[ \psi \phi = \begin{bmatrix}
    \iota \\
    \kappa
\end{bmatrix}  \begin{bmatrix}
    \phi_1 &
    \phi_2
\end{bmatrix}  =  \begin{bmatrix}
  \iota \phi_1 &\iota \phi_2
\kappa \phi_1 & \kappa \phi_2
\end{bmatrix} =   \begin{bmatrix}
1_{\mathsf{im}(x^{k+1})} & 0\\ 
0 & 1_{\mathsf{ker}(x^{k+1})}
\end{bmatrix} = 1_{\mathsf{im}(x^{k+1}) \oplus \mathsf{ker}(x^{k+1})} \]
So we conclude that $\psi: \mathsf{im}(x^{k+1}) \oplus \mathsf{ker}(x^{k+1}) \to A$ is an isomorphism. 

For $(\Leftarrow)$, suppose that $\psi: \mathsf{im}(x^{k+1}) \oplus \mathsf{ker}(x^{k+1}) \to A$ is an isomorphism. To show that $x$ is Drazin, we will show that $\psi$ is part of a Fitting decomposition of $x$. So let $\phi_1: A \to \mathsf{im}(x^{k+1})$ and $\phi_2: A \to \mathsf{ker}(x^{k+1})$ be the components of $\psi^{-1}: A \to \mathsf{im}(x^{k+1}) \oplus \mathsf{ker}(x^{k+1})$. Then from $\psi \psi^{-1} = 1_{\mathsf{im}(x^{k+1}) \oplus \mathsf{ker}(x^{k+1})}$ and $\psi^{-1} \psi = 1_A$, we get the following equalities: 
\begin{gather*}
   \iota \phi_1 = 1_{\mathsf{im}(x^{k+1})} \qquad \iota \phi_2 =0 \qquad \kappa \phi_1 = 0 \qquad  \kappa \phi_2 = 1_{\mathsf{ker}(x^{k+1})} \qquad \phi_1 \iota + \phi_2 \kappa = 1_A
\end{gather*}
Now for the isomorphism part of our decomposition, define $\alpha: \mathsf{im}(x^{k+1}) \to \mathsf{im}(x^{k+1})$ to be the composite $\alpha \colon = \iota x \phi_1$. To show that $\alpha$ is an isomorphism, we will use the fact that in an Abelian category, a map is an isomorphism if and only if it is monic and epic. To show that $\alpha$ is monic and epic, we will need to first compute some useful identities. To start with, $\epsilon \iota = x^{k+1}$ and $\iota \phi_1 = 1_{\mathsf{im}(x^{k+1})}$ together imply that $x^{k+1} \phi_1 = \epsilon$. Moreover, from $\epsilon \iota = x^{k+1}$ and $\kappa x^{k+1} =0$, since $\iota$ is monic, it follows that $\kappa \epsilon = 0$. From this and $\phi_1 \iota + \phi_2 \kappa = 1_A$ we also get that $\phi_1 \iota \epsilon = \epsilon$. On the other hand, we also get: 
\[ x \epsilon = x x^{k+1} \phi_1 = x^{k+1} x \phi_1 = \epsilon \iota x \phi_1 = \epsilon \alpha\]
So, from $x \epsilon = \epsilon \alpha$, we have:
\[ \epsilon \alpha \iota = x \epsilon \iota = x x^{k+1} = x^{k+1} x = \epsilon \iota x \] 
Then since $\epsilon$ is epic, we get that $\alpha \iota = \iota x$, which also tells us that $\iota x^n = \alpha^n \iota$ for all $n \in \mathbb{N}$. Then from $\iota \phi_1 = 1_{\mathsf{im}(x^{k+1})}$, we get that $\alpha^n = \iota x^n \phi_1$. 

Now since we are in an Abelian category, to show that $\alpha$ is monic, it suffices to show that if $f \alpha =0$ then $f=0$. So suppose that $f \alpha =0$. Then we have that: 
\[ f \iota x^{k+1} = f \iota x x^k = f \alpha \iota x^k = 0 \iota x^k = 0 \]
Since $f \iota x^{k+1} =0$, by the universal property of the kernel, there exists a unique $h$ such that $f \iota = h \kappa$. Post-composing both sides by $\phi_2$, since $\kappa \phi_2 = 1_{\mathsf{ker}(x^{k+1})}$ and $\iota \phi_2 =0$, we get that $h=0$ and so $f \iota =0$. However, $\iota$ is monic, therefore $f=0$ as desired. So $\alpha$ is monic. Dually, since we are in an Abelian category, to show that $\alpha$ is epic, it suffices to show that if $\alpha g= 0$ then $g=0$. So suppose that $\alpha g=0$. Then 
\[ \epsilon g = \phi_1 \iota \epsilon g = \phi_1 \iota x^{k+1} \phi_1 g = \phi_1 \iota \alpha^{k+1} g = \phi_1 \iota \alpha^k \alpha g = \phi_1 \iota \alpha^k  0 = 0 \]
Since $\epsilon$ is epic, it follows that $g=0$, as desired. So $\alpha$ is epic. Therefore, we conclude that $\alpha$ is an isomorphism. 

For the nilpotent part of the decomposition, note that since $\kappa x^{k+1} =0$, we also have that $\kappa x x^{k+1} =0$. Therefore, by the universal property of the kernel, there exists a unique map $\eta: \mathsf{ker}(x^{k+1}) \to \mathsf{ker}(x^{k+1})$ such that $\eta \kappa = \kappa x$. However, since $\kappa \phi_2 = 1_{\mathsf{ker}(x^{k+1})}$, we get that $\eta = \kappa x \phi_2$. Moreover, we compute that: 
\[ \kappa x = \kappa x  \left(\phi_1 \iota + \phi_2 \kappa \right) = \kappa x \phi_1 \iota +  \kappa x \phi_2 \kappa = \eta \kappa \phi_1 \iota + \kappa x \phi_2 \kappa = 0 + \kappa x \phi_2 \kappa = \kappa x \phi_2 \kappa \]
Since $\kappa x = \kappa x \phi_2 \kappa$, we can prove by induction that $\eta^{m} = \kappa x^m \phi_2$. The base case $m=0$ follows from $\kappa \phi_2 = 1_{\mathsf{ker}(x^{k+1})}$. So suppose that for all $0 \leq j \leq n$, $\eta^{n} = \kappa x^n \phi_2$. 
\[ \eta^{n+1} = \eta \eta^n = \eta \kappa x^n \phi_2 = \kappa x \phi_2 \kappa x^n \phi_2 = \kappa x x^n \phi_2 = \kappa x^{n+1} \phi_2 \]
So we do have that $\eta^{m} = \kappa x^m \phi_2$ for all $m \in \mathbb{N}$. Now for $m=k+1$, we get that: 
\[ \eta^{k+1} = \kappa x^{k+1} \phi_2 = 0 \phi_2 = 0    \]
So $\eta$ is nilpotent as desired. 

Lastly, we finally compute that: 
\begin{gather*}
    \psi^{-1} \begin{bmatrix}
  \alpha & 0 \\ 
0 & \eta
\end{bmatrix} \psi = \begin{bmatrix}
    \phi_1 &
    \phi_2
\end{bmatrix} \begin{bmatrix}
  \alpha & 0 \\ 
0 & \eta
\end{bmatrix} \begin{bmatrix}
    \iota \\
    \kappa
\end{bmatrix} = \begin{bmatrix}
    \phi_1 \alpha &
    \phi_2 \eta
\end{bmatrix} \begin{bmatrix}
    \iota \\
    \kappa
\end{bmatrix} \\
= \phi_1 \alpha \iota + \phi_2 \eta \kappa = \phi_1 \iota x + \phi_2 \kappa x =  \left(\phi_1 \iota + \phi_2 \kappa \right) x = x
\end{gather*} 
So we conclude that $(\psi^{-1}, \alpha, \eta)$ is a Fitting's decomposition of $x$. So by Theorem \ref{thm:iso-nil}, we get that $x$ is Drazin. 
\end{proof}



\section{Drazin Inverses of Opposing Pairs of Maps} \label{drazin-equivalences}


Arriving with categorical eyes to the subject of Drazin inverses it is natural to want to have a Drazin inverse of an arbitrary map. However, to have a Drazin inverse of a map $f: A \to B$, one really needs an \emph{opposing} map $g: B \to A$ to allow for the iteration which is at the heart of the notion of a Drazin inverse. The Drazin inverse of an opposing pair, $(f,g)$, is itself an opposing pair whose properties we develop in this section. Throughout this section we work in an arbitrary category $\mathbb{X}$.

\subsection{The Drazin inverse of an opposing pair} We denote a pair of maps of dual type $f: A \to B$ and $g: B \to A$ by $\xymatrixcolsep{1.5pc}\xymatrix{(f,g): A  \ar@<0.5ex>[r]  & B  \ar@<0.5ex>[l] }$ and refer to it as an \textbf{opposing pair}. 

\begin{definition} \label{def:opposing-pair-drazin-inverse} A \textbf{Drazin inverse} of $\xymatrixcolsep{1.5pc}\xymatrix{(f,g): A  \ar@<0.5ex>[r]  & B  \ar@<0.5ex>[l] }$ is an opposing pair of dual type $\xymatrixcolsep{1.5pc}\xymatrix{(f^{\frac{D}{g}},g^{\frac{D}{f}}): B  \ar@<0.5ex>[r]  & A  \ar@<0.5ex>[l] }$ satisfying the following properties:
\begin{enumerate}[{\bf [DV.1]}]
\item There is a $k \in \mathbb{N}$ such that $(fg)^k f f^{\frac{D}{g}} = (fg)^k$ and $(gf)^k g g^{\frac{D}{f}} = (gf)^k$. 
\item $f^{\frac{D}{g}} f f^{\frac{D}{g}} = f^{\frac{D}{g}}$ and $g^{\frac{D}{f}} g g^{\frac{D}{f}} = g^{\frac{D}{f}}$;
\item $ f f^{\frac{D}{g}} =  g^{\frac{D}{f}} g$ and $ f^{\frac{D}{g}} f =  g g^{\frac{D}{f}}$. 
\end{enumerate}
 The map $f^{\frac{D}{g}}: B \to A$ is called the Drazin inverse of $f$ over $g$, while the map $g^{\frac{D}{f}}: A \to B$ is called the Drazin inverse of $g$ over $f$. The \textbf{Drazin index} of $(f,g)$ is the smallest $k \in \mathbb{N}$ such that {\em \textbf{[DV.1]}} holds, which we denote by $\mathsf{ind}(f,g)=k$.
\end{definition}

There is a slight redundancy in the above definition in that, thanks to \textbf{[DV.3]}, we do not need to assume both equalities for \textbf{[DV.1]}. 

\begin{lemma}\label{lemma:DV1} In a category $\mathbb{X}$, if for $\xymatrixcolsep{1.5pc}\xymatrix{(f,g): A  \ar@<0.5ex>[r]  & B  \ar@<0.5ex>[l] }$ there is a $\xymatrixcolsep{1.5pc}\xymatrix{(f^{\frac{D}{g}},g^{\frac{D}{f}}): B  \ar@<0.5ex>[r]  & A  \ar@<0.5ex>[l] }$ which satisfies {\bf [DV.3]}, then the following are equivalent: 
\begin{enumerate}[(i)]
\item {\bf [DV.1]} holds;
\item There is a $p \in \mathbb{N}$ such that $(fg)^p f f^{\frac{D}{g}} = (fg)^p$;
\item There is a $q \in \mathbb{N}$ such that $(gf)^q g g^{\frac{D}{g}} = (gf)^q$. 
\end{enumerate}
\end{lemma}
\begin{proof} We prove that $(i) \Rightarrow (ii) \Rightarrow (iii) \Rightarrow (i)$. Clearly, we already have that $(i) \Rightarrow (ii)$. For $(ii) \Rightarrow (iii)$, suppose that there is a $p \in \mathbb{N}$ such that $(fg)^p f f^{\frac{D}{g}} = (fg)^p$. Then using \textbf{[DV.3]} and that $(gf)^{p+1} = g(fg)^p f$, we compute: 
\[ (gf)^{p+1} g g^{\frac{D}{g}} = g(fg)^p f g g^{\frac{D}{f}} = g(fg)^p f f^{\frac{D}{g}} f = g(fg)^p f  = (gf)^{p+1} \]
So $(iii)$ holds. For $(iii) \Rightarrow (i)$, suppose that there is a $q \in \mathbb{N}$ such that $(gf)^q g g^{\frac{D}{g}} = (gf)^q$. Now that we also have that $(gf)^{q+1} g g^{\frac{D}{g}} = (gf)^{q+1}$. On the other hand, using \textbf{[DV.3]} and that $(fg)^{q+1} = f(gf)^q g$, we compute: 
\[ (fg)^{q+1} f f^{\frac{D}{g}} = f(gf)^q g f f^{\frac{D}{g}} = f(gf)^q g g^{\frac{D}{f}} g = f (gf)^q g = (gf)^{q+1} \]
So \textbf{[DV.1]} holds. So we conclude that $(i) \Leftrightarrow (ii) \Leftrightarrow (iii)$ as desired. 
\end{proof}

Drazin inverses of opposing pairs share many of the same properties as Drazin inverses of endomorphisms, such as the fact that they are unique (if they exist) and also absolute. Before establishing these properties, we will show how Drazin inverses of opposing pairs are linked to Drazin inverses of endomorphisms. 

\subsection{Drazin Opposing Pairs} From an opposing pair $\xymatrixcolsep{1.5pc}\xymatrix{(f,g): A  \ar@<0.5ex>[r]  & B  \ar@<0.5ex>[l] }$, we get two endomorphisms $fg: A \to A$ and $gf: B \to B$. Our objective is to show that $(f,g)$ has a Drazin inverse if and only if the composites $fg$ and $gf$ both have Drazin inverses. To do so, we first show that, surprisingly, if one of these composites is Drazin, then so is the other one: this is sometimes known as the Cline's formula \cite{Cline1965}. 

\begin{lemma}\label{lemma:Drazin-fg-gf} For an opposing pair $\xymatrixcolsep{1.5pc}\xymatrix{(f,g): A  \ar@<0.5ex>[r]  & B  \ar@<0.5ex>[l] }$, $fg: A \to A$ is Drazin if and only if $gf: B \to B$ is Drazin. 
\end{lemma}
\begin{proof} For ($\Rightarrow$), Suppose that $fg$ is Drazin. We show that $(gf)^D \colon = g (fg)^D (fg)^D f$ satisfies the three axioms of a Drazin inverse. 
     \begin{enumerate}[{\bf [D.1]}]
\item Suppose that $\mathsf{ind}(fg)=k$. Then using Lemma \ref{lemma:Drazin-basic}.(ii) for $fg$, we compute: 
\begin{gather*}
    (gf)^{k+2} (gf)^D = g (fg)^{k+1} f (gf)^D = g (fg)^{k+1} f g (fg)^D (fg)^D f \\
    = g (fg)^{k+2} ((fg)^D)^2 f = g (fg)^k f = (gf)^{k+1}
    \end{gather*}
    \item Using \textbf{[D.2]} and \textbf{[D.3]} for $fg$, we compute that: 
\begin{gather*} (gf)^D gf (gf)^D = g (fg)^D (fg)^D f gf g (fg)^D (fg)^D f = \\
g (fg)^D f g (fg)^D  (fg)^D f g (fg)^D f = g (fg)^D (fg)^D f = (gf)^D     \end{gather*}
    \item Using \textbf{[D.3]} for $fg$, we compute that: 
  \begin{gather*} (gf)^D gf = g (fg)^D (fg)^D f g f = g (fg)^D f g (fg)^D  f  = g f g (fg)^D  (fg)^D  f = gf (gf)^D  \end{gather*}
\end{enumerate}
So we conclude that $gf$ is Drazin. 

For ($\Leftarrow$), suppose that $gf$ is Drazin. By similar calculations as above, we can show that $(fg)^D \colon = f (gf)^D (gf)^D g$ satisfies the three axioms of a Drazin inverse. So we conclude that $fg$ is Drazin. 
\end{proof}

This suggests the following definition: 

\begin{definition} A \textbf{Drazin opposing pair} is an opposing pair $\xymatrixcolsep{1.5pc}\xymatrix{(f,g): A  \ar@<0.5ex>[r]  & B  \ar@<0.5ex>[l] }$ such that either $fg$ or $gf$ is Drazin (so, by Lemma \ref{lemma:Drazin-fg-gf},  both are Drazin). 
\end{definition}

Below we will show that being a Drazin opposing pair is equivalent to having a Drazin inverse for the pair. In order to prove this, it will be useful to establish the following identities for Drazin opposing pairs: 

\begin{lemma}\label{lemma:Drazin-fg-gf-2} If $\xymatrixcolsep{1.5pc}\xymatrix{(f,g): A  \ar@<0.5ex>[r]  & B  \ar@<0.5ex>[l] }$ is a Drazin opposing pair, then we have that $(fg)^D f = f (gf)^D$ and $(gf)^D g = g (fg)^D$. 
\end{lemma}
\begin{proof} Suppose that $(f,g)$ is a Drazin opposing pair. Then by Lemma \ref{lemma:Drazin-fg-gf} we have that $(gf)^D = g (fg)^D (fg)^D f$ and $(fg)^D = f (gf)^D (gf)^D g$. Then using \textbf{[D.2]} and \textbf{[D.3]} for $gf$, we compute that:
\[ (fg)^D f = f (gf)^D (gf)^D g f = f (gf)^D g f (gf)^D = f (gf)^D \]
So $(fg)^D f = f (gf)^D$. Similarly, we can also show that $(gf)^D g = g (fg)^D$. 
\end{proof}

\begin{theorem}\label{Drazin-equiv-inverse} $\xymatrixcolsep{1.5pc}\xymatrix{(f,g): A  \ar@<0.5ex>[r]  & B  \ar@<0.5ex>[l] }$ has a Drazin inverse if and only if $(f,g)$ is a Drazin opposing pair.
\end{theorem}
\begin{proof} For ($\Rightarrow$), suppose that $(f,g)$ has a Drazin inverse $(f^{\frac{D}{g}},g^{\frac{D}{f}})$ with $\mathsf{ind}(f,g)=k$. We first show that $(fg)^D \colon = g^{\frac{D}{f}} f^{\frac{D}{g}}$ satisfies the three axioms of a Drazin inverse. 
     \begin{enumerate}[{\bf [D.1]}]
\item Note that $(fg)^{k+1}=f (gf)^k g$ and also that $f (gf)^{k}= (fg)^{k} f$. Then using \textbf{[DV.1]}, we compute that: 
\begin{gather*}
    (fg)^{k+1} (fg)^D \!=\! f (gf)^k g (fg)^D \!=\! f (gf)^{k} g g^{\frac{D}{f}} f^{\frac{D}{g}} = f (gf)^{k}  f^{\frac{D}{g}} = (fg)^{k} f f^{\frac{D}{g}} \!=\! (fg)^{k}
    \end{gather*}
    \item Using \textbf{[DV.2]} and \textbf{[DV.3]}, we compute that: 
    \begin{gather*}
        (fg)^D fg (fg)^D \!=\! g^{\frac{D}{f}} f^{\frac{D}{g}} fg g^{\frac{D}{f}} f^{\frac{D}{g}} = g^{\frac{D}{f}} g g^{\frac{D}{f}} g g^{\frac{D}{f}} f^{\frac{D}{g}}  = g^{\frac{D}{f}} g g^{\frac{D}{f}} f^{\frac{D}{g}} = g^{\frac{D}{f}} f^{\frac{D}{g}} = (fg)^D 
    \end{gather*}
    \item Using \textbf{[D.3]} for $fg$, we compute that: 
  \begin{gather*}
        (fg)^D fg  \!=\! g^{\frac{D}{f}} f^{\frac{D}{g}} fg = g^{\frac{D}{f}} g g^{\frac{D}{f}} g = g^{\frac{D}{f}} g = f f^{\frac{D}{g}} = f f^{\frac{D}{g}} f f^{\frac{D}{g}} = f g g^{\frac{D}{f}} f^{\frac{D}{g}} = fg (fg)^D 
    \end{gather*}
\end{enumerate}
So we conclude that $fg$ is Drazin with $\mathsf{ind}(fg) \leq \mathsf{ind}(f,g)$. Then by Lemma \ref{lemma:Drazin-fg-gf}, we also get that $gf$ is Drazin, and it is easy to check that $\mathsf{ind}(gf) \leq \mathsf{ind}(f,g)$ as well. Now using the formula from Lemma \ref{lemma:Drazin-fg-gf}, we can also check using \textbf{[DV.2]} and \textbf{[DV.3]} that: 
\begin{gather*}
    (gf)^D = g (fg)^D (fg)^D f = g g^{\frac{D}{f}} f^{\frac{D}{g}} g^{\frac{D}{f}} g g^{\frac{D}{f}}= f^{\frac{D}{g}} g^{\frac{D}{f}}
\end{gather*}
So we conclude that $(f,g)$ is a Drazin opposing pair.  

For ($\Leftarrow$), Suppose that $(f,g)$ is a Drazin opposing pair. Note by Lemma \ref{lemma:Drazin-fg-gf-2} that setting $f^{\frac{D}{g}} \colon = g (fg)^D = (gf)^D g$ and $g^{\frac{D}{f}} \colon = f (gf)^D = (fg)^D f$ is well-defined. So now we must show that $(f^{\frac{D}{g}},g^{\frac{D}{f}})$ satisfies the three axioms of a Drazin inverse. We start with the third one: 
 \begin{enumerate}[{\bf [DV.1]}]
 \setcounter{enumi}{2}
\item Using \textbf{[D.3]} for $fg$, we compute that: 
\begin{gather*}
    f f^{\frac{D}{g}} = f g (fg)^D = (fg)^D fg = g^{\frac{D}{f}} g
\end{gather*}
So $ f f^{\frac{D}{g}} =  g^{\frac{D}{f}} g$, and similarly we can show that $f^{\frac{D}{g}} f =  g g^{\frac{D}{f}}$. 
\end{enumerate}
         \begin{enumerate}[{\bf [DV.1]}]
\item Let $k = \mathsf{max}(\mathsf{ind}(fg), \mathsf{ind}(gf))$. Then by Lemma \ref{lemma:Drazin-basic}.(\ref{lemma:Drazin-basic.1}) for $fg$, we compute that: 
\begin{gather*}
    (fg)^k f f^{\frac{D}{g}} = (fg)^k f g (fg)^D = (fg)^{k+1} (fg)^D = (fg)^k 
\end{gather*}
As such, since \textbf{[DV.3]} holds, by Lemma \ref{lemma:DV1} we then get that \textbf{[DV.1]} holds. 
\item Using \textbf{[D.2]} for $gf$, we compute that: 
\begin{gather*}
  f^{\frac{D}{g}}  f f^{\frac{D}{g}} =  (gf)^D g f (gf)^D g = (gf)^D g = f^{\frac{D}{g}}
\end{gather*}
So $ f^{\frac{D}{g}}  f f^{\frac{D}{g}} = f^{\frac{D}{g}}$, and similarly we can show that $g^{\frac{D}{f}}  g g^{\frac{D}{f}}  =g$ as well. 
\end{enumerate}
So we conclude that $(f,g)$ has a Drazin inverse and $\mathsf{ind}(f,g) \!\leq\! \mathsf{max}(\mathsf{ind}(fg), \mathsf{ind}(gf))$. \end{proof}

Therefore, from now on we will simply say that an opposing pair $\xymatrixcolsep{1.5pc}\xymatrix{(f,g): A  \ar@<0.5ex>[r]  & B  \ar@<0.5ex>[l] }$ is \textbf{Drazin} if $(f,g)$ is a Drazin opposing pair or equivalently if $(f,g)$ has a Drazin inverse.

\subsection{Properties of Drazin Opposing Pairs} Now that we have related Drazin inverses of opposing pairs to usual Drazin inverses of endomorphisms, we can use this to access properties of Drazin opposing pairs which are analogues for those of Drazin endomorphisms. We begin with the all-important property of uniqueness: 

\begin{proposition}\label{Drazin-opposing-unique} If $\xymatrixcolsep{1.5pc}\xymatrix{(f,g): A  \ar@<0.5ex>[r]  & B  \ar@<0.5ex>[l] }$ has a Drazin inverse, then it is unique. 
\end{proposition}
\begin{proof} Suppose that $(f,g)$ has two Drazin inverses $(h,k)$ and $(p,q)$. By Theorem \ref{Drazin-equiv-inverse}, $(f,g)$ is also Drazin, so $fg$ and $gf$ are both Drazin, and we may use $(h,k)$ and $(p,q)$ to build their Drazin inverses. Since Drazin inverses are unique, we must have that $qp = (fg)^D = kh$ and $pq = (gf)^D = hk$. Now using \textbf{[DV.2]} and \textbf{[DV.3]}, we compute: 
\begin{gather*}
k = k g k = k h f = qp f = q g q = q
\end{gather*}
So $k=q$, and similarly we can also show that $h=p$. So we conclude that Drazin inverses of opposing pairs are unique. 
\end{proof}

This allows us to talk of \emph{the} Drazin inverse of  $\xymatrixcolsep{1.5pc}\xymatrix{(f,g): A  \ar@<0.5ex>[r]  & B  \ar@<0.5ex>[l] }$
and to write it as: $\xymatrixcolsep{1.5pc}\xymatrix{(f,g)^D \colon = (f^{\frac{D}{g}},g^{\frac{D}{f}}): B  \ar@<0.5ex>[r]  & A  \ar@<0.5ex>[l] }$. Moreover, this implies that the formulas of the Drazin inverses in the proof of Theorem \ref{Drazin-equiv-inverse} are canonical, and the Drazin index of the opposing pair is equal to the max of the Drazin indexes of the composites. 

\begin{corollary}\label{cor:Drazinpairformulas} If $\xymatrixcolsep{1.5pc}\xymatrix{(f,g): A  \ar@<0.5ex>[r]  & B  \ar@<0.5ex>[l] }$ is Drazin then:
\begin{enumerate}[(i)]
\item $(fg)^D = g^{\frac{D}{f}} f^{\frac{D}{g}}$ and $(gf)^D = f^{\frac{D}{g}}g^{\frac{D}{f}}$;
\item $f^{\frac{D}{g}} = g (fg)^D = (gf)^D g$ and $g^{\frac{D}{f}} = f (gf)^D = (fg)^D f$
\end{enumerate} 
Moreover, $\mathsf{ind}(f,g) = \mathsf{max}(\mathsf{ind}(fg), \mathsf{ind}(gf))$. 
\end{corollary}

It is also straightforward to see that if an opposing pair is Drazin, then its symmetric pair will also be Drazin: 

\begin{lemma}\label{lemma:sym} $\xymatrixcolsep{1.5pc}\xymatrix{(f,g): A  \ar@<0.5ex>[r]  & B  \ar@<0.5ex>[l] }$ is Drazin if and only if $\xymatrixcolsep{1.5pc}\xymatrix{(g,f): B  \ar@<0.5ex>[r]  & A  \ar@<0.5ex>[l] }$ is Drazin. 
\end{lemma}
\begin{proof} By definition, $(f,g)$ is Drazin if and only if $fg$ is Drazin  or $gf$ is Drazin if and only if $(g,f)$ is Drazin. 
\end{proof}

We can also recover the Drazin inverses of an endomorphism from the Drazin inverse of the opposing pair consisting of the endomorphism paired with the identity.

\begin{lemma} $x: A \to A$ is Drazin if and only if $\xymatrixcolsep{1.5pc}\xymatrix{(x,1_A): A  \ar@<0.5ex>[r]  & A  \ar@<0.5ex>[l] }$ is Drazin, or equivalently if $\xymatrixcolsep{1.5pc}\xymatrix{(1_A,x): A  \ar@<0.5ex>[r]  & A  \ar@<0.5ex>[l] }$ is Drazin. 
\end{lemma}
\begin{proof} For ($\Rightarrow$), Suppose that $x$ is Drazin. Then trivially $x 1_A = x = 1_A x$ is Drazin. So $(x,1_A)$ (or equivalently $(1_A, x)$) is Drazin. Applying Corollary \ref{cor:Drazinpairformulas}, we get that $ x^{\frac{D}{1_A}} = 1_A (1_A x)^D = x^D$ and $1^{\frac{D}{x}} = x (1_A x)^D = x x^D$. For ($\Leftarrow$), suppose that $(x,1_A)$ is Drazin. So $x 1_A= x = 1_A x$ is Drazin. Applying Corollary \ref{cor:Drazinpairformulas}, we compute that: $x^D = (1_A x)^D = x^{\frac{D}{1_A}} 1^{\frac{D}{x}}$ and $x^D = (x1_A)^D = 1^{\frac{D}{x}} x^{\frac{D}{1_A}}$. 
\end{proof}

Thus, we have yet another equivalent characterization of when a category is Drazin. 

\begin{corollary} A category is Drazin if and only if every opposing pair in it is Drazin. 
\end{corollary}
\begin{proof}  If $\mathbb{X}$ is Drazin, then for an opposing pair $\xymatrixcolsep{1.5pc}\xymatrix{(f,g): A  \ar@<0.5ex>[r]  & B  \ar@<0.5ex>[l] }$, $fg$ and $gf$ are both Drazin, so $(f,g)$ is Drazin. Conversely, if every opposing pair is Drazin, then by Lemma \ref{Drazin-equiv-inverse}, every endomorphism must be Drazin. 
\end{proof}

Drazin inverses of opposing pairs are also absolute: 

\begin{proposition}\label{Drazin-opposing-absolute} Let $\mathsf{F}: \mathbb{X} \to \mathbb{Y}$ be a functor and let $\xymatrixcolsep{1.5pc}\xymatrix{(f,g): A  \ar@<0.5ex>[r]  & B  \ar@<0.5ex>[l] }$ be Drazin in $\mathbb{X}$, then $\xymatrixcolsep{1.5pc}\xymatrix{(\mathsf{F}(f),\mathsf{F}(g)): \mathsf{F}(A)  \ar@<0.5ex>[r]  & \mathsf{F}(B)  \ar@<0.5ex>[l] }$ is Drazin.
\end{proposition}
\begin{proof} Suppose that $(f,g)$ is Drazin, which in particular means that $fg$ and $gf$ are Drazin. By Proposition \ref{Drazin-absolute}, Drazin inverses are absolute, so $\mathsf{F}(fg)= \mathsf{F}(f) \mathsf{F}(g)$ and $\mathsf{F}(gf)= \mathsf{F}(g) \mathsf{F}(f)$ are also Drazin. As such, we get that $(\mathsf{F}(f),\mathsf{F}(g))$ is Drazin. Moreover, applying Corollary \ref{cor:Drazinpairformulas}, we compute that: 
\begin{gather*}
    \mathsf{F}(f)^{\frac{D}{\mathsf{F}(g)}} = \mathsf{F}(g) \mathsf{F}(fg)^D =  \mathsf{F}(g) \mathsf{F}( (fg)^D ) = \mathsf{F}( g (fg)^D) = \mathsf{F}(\mathsf{F}(f^{\frac{D}{g}}))
\end{gather*}
So $\mathsf{F}(f)^{\frac{D}{\mathsf{F}(g)}} = \mathsf{F}(f^{\frac{D}{g}})$, and similarly we can also compute that $\mathsf{F}(g)^{\frac{D}{\mathsf{F}(f)}} = \mathsf{F}(g^{\frac{D}{f}})$. 
\end{proof}

Much as before, the Drazin inverse of an opposing pair commutes with everything with which the opposing pair commutes with: 

\begin{lemma} Suppose that $\xymatrixcolsep{1.5pc}\xymatrix{(f,g): A  \ar@<0.5ex>[r]  & B  \ar@<0.5ex>[l] }$ and $\xymatrixcolsep{1.5pc}\xymatrix{(v,w): A^\prime  \ar@<0.5ex>[r]  & B^\prime  \ar@<0.5ex>[l] }$ are Drazin. Then if the diagram on the left commutes, then the diagram on the right commutes: 
 \[\begin{array}[c]{c} 
  \xymatrixcolsep{5pc}\xymatrix{A \ar[d]_-f \ar[r]^-a & A' \ar[d]^-v \\ B \ar[d]_-g \ar[r]^-b	 & B' \ar[d]^-w  \\ A \ar[r]_-a & A' }   
\end{array}  \Rightarrow \begin{array}[c]{c}    \xymatrixcolsep{5pc}\xymatrix{A \ar[d]_-{g^{\frac{D}{f}}} \ar[r]^-a & A' \ar[d]^-{w^{\frac{D}{v}}} \\ B \ar[d]_-{f^{\frac{D}{g}}} \ar[r]^-b	 & B' \ar[d]^-{v^{\frac{D}{w}}}  \\ A \ar[r]_-a & A' } 
 \end{array}   \]
\end{lemma}

\begin{proof} If $(f,g)$ and $(v,w)$ are Drazin, we have that $fg$ and $vw$ are Drazin. So by Lemma \ref{Drazin-commuting}, the outer square of the diagram on the left implies that $(fg)^D a = a (vw)^D$. Now by using this and Corollary \ref{cor:Drazinpairformulas}, we compute that:  
\[ g^{\frac{D}{f}} b = (fg)^D f b = (fg)^Da v = a (vw)^D v = a w^{\frac{D}{v}}\]
So $g^{\frac{D}{f}} b = a w^{\frac{D}{v}}$, and similarly we can also show that $f^{\frac{D}{g}} a = b v^{\frac{D}{w}}$. So the diagram on the right commutes as desired. 
\end{proof}

\subsection{Opposing Pairs of Isomorphisms} We now show that an opposing pair has a Drazin index of zero precisely when it is a pair of opposing isomorphisms. The Drazin inverse of an opposing pair of isomorphisms is the opposing pair of their inverses. 

\begin{lemma} $\xymatrixcolsep{1.5pc}\xymatrix{(f,g): A  \ar@<0.5ex>[r]  & B  \ar@<0.5ex>[l] }$ is Drazin with $\mathsf{ind}(f,g) =0$ if and only if $f$ and $g$ are isomorphisms. 
\end{lemma}

\begin{proof} For $(\Rightarrow)$, Suppose that $(f,g)$ is Drazin with $\mathsf{ind}(f,g) =0$. Then \textbf{[DV.1]} tells us that $f f^{\frac{D}{g}} = 1_A$ and $g g^{\frac{D}{f}} = 1_B$. Now \textbf{[DV.3]} gives us that $f^{\frac{D}{g}} f = g g^{\frac{D}{f}} = 1_B$ and $g^{\frac{D}{f}} g= f f^{\frac{D}{g}} = 1_A$. So $f$ and $g$ are isomorphisms with $f^{-1} = f^{\frac{D}{g}}$ and $g^{\frac{D}{f}}= g^{-1}$. 

For $(\Leftarrow)$, Suppose that $f$ and $g$ are isomorphisms. Then the composites $fg$ and $gf$ are also isomorphisms. By Lemma \ref{lem:Drazin-0}, we get that $fg$ and $gf$ are Drazin with $\mathsf{ind}(fg) = 0$ and $\mathsf{ind}(gf) = 0$. As such, $(f,g)$ is Drazin. Applying Corollary \ref{cor:Drazinpairformulas} and Lemma \ref{lem:Drazin-0}, we compute that: 
\[ f^{\frac{D}{g}} = g (fg)^D = g (fg)^{-1} = g g^{-1} f^{-1} = f^{-1} \]
    So $f^{\frac{D}{g}} =  f^{-1}$, and similarly we also get that $g^{\frac{D}{f}} =  g^{-1}$. Moreover, we also have that $\mathsf{ind}(f,g) =0$. 
\end{proof}


\subsection{Drazin Opposing Pairs and Idempotents} In Section \ref{drazin-idempotents} we saw how to equivalently characterize having a Drazin inverse in terms of isomorphisms in the idempotent completion. The objective of this section is to do the same for the Drazin inverse of an opposing pair. We first recall from Lemma \ref{lemma:e_x} that every Drazin endomorphism $x$ induces an idempotent $e_x = xx^D = x^D x$. Then for a Drazin opposing pair $(f,g)$, we get two idempotents $e_{fg} = fg(fg)^D$ and $e_{gf} = gf(gf)^D = (gf)^D gf$. Moreover, we can also use the Drazin inverse to get isomorphisms between these two idempotents. 

\begin{lemma}\label{lemma:e_gf+e_fg-iso} If $\xymatrixcolsep{1.5pc}\xymatrix{(f,g): A  \ar@<0.5ex>[r]  & B  \ar@<0.5ex>[l] }$ is Drazin, then: 
\begin{enumerate}[(i)]
\item\label{e_fg+e_gf} $e_{fg} = f f^{\frac{D}{g}} = g^{\frac{D}{f}}g$ and $e_{gf} = f^{\frac{D}{g}} f = g g^{\frac{D}{f}}$ are induced idempotents;
\item $f f^{\frac{D}{g}} f: (A, e_{fg}) \to (B, e_{gf})$ is an isomorphism in $\mathsf{Split}(\mathbb{X})$ whose inverse is given by $f^{\frac{D}{g}}: (B, e_{gf}) \to (A, e_{fg})$;
\item $g g^{\frac{D}{f}} g: (B, e_{gf}) \to (A, e_{fg})$ is an isomorphism in $\mathsf{Split}(\mathbb{X})$ whose inverse is given by $g^{\frac{D}{f}}: (A, e_{gf}) \to  (B, e_{fg}) $. 
\end{enumerate}
\end{lemma}
\begin{proof}
\begin{enumerate}[(i)]
\item Recall from Corollary \ref{cor:Drazinpairformulas} that $f^{\frac{D}{g}} = g (fg)^D = (gf)^D g$ and $g^{\frac{D}{f}} = f (gf)^D = (fg)^D f$. Then combining these with the definition of the induced idempotents, we immediately get that $e_{fg} = f f^{\frac{D}{g}} = g^{\frac{D}{f}}g$ and $e_{gf} = f^{\frac{D}{g}} f = g g^{\frac{D}{f}}$ as desired. 
\item We must first show that $f f^{\frac{D}{g}} f$ and $f^{\frac{D}{g}}$ are maps in the idempotent completion. First by using (\ref{e_fg+e_gf}) and \textbf{[DV.2]} we compute that: 
\[ e_{gf} f^{\frac{D}{g}} e_{fg} = f^{\frac{D}{g}} f   f^{\frac{D}{g}} f f^{\frac{D}{g}} =  f^{\frac{D}{g}} f f^{\frac{D}{g}} = f^{\frac{D}{g}} \]
So $e_{gf} f^{\frac{D}{g}} e_{fg} = f^{\frac{D}{g}}$. On the other hand,
\[ e_{fg} f f^{\frac{D}{g}} f e_{gf} =  f f^{\frac{D}{g}}  f f^{\frac{D}{g}} f  f^{\frac{D}{g}} f =  f f^{\frac{D}{g}} f  f^{\frac{D}{g}} f =  f f^{\frac{D}{g}} f   \]
So $e_{fg} f f^{\frac{D}{g}} f e_{gf} = f f^{\frac{D}{g}} f$ as well. As such, we get  $f f^{\frac{D}{g}} f: (A, e_{fg}) \to (B, e_{gf})$ and $f^{\frac{D}{g}}: (B, e_{gf}) \to (A, e_{fg})$  are indeed maps in  $\mathsf{Split}(\mathbb{X})$. Now we must show that they are inverses of each other in $\mathsf{Split}(\mathbb{X})$. However, using (\ref{e_fg+e_gf}) again, we easily see that: 
\begin{align*}
  f f^{\frac{D}{g}} f  f^{\frac{D}{g}} = e_{fg} e_{fg} = e_{fg} =1_{(A, e_{fg})} &&   f^{\frac{D}{g}} f f^{\frac{D}{g}} f = e_{gf} e_{gf} = e_{gf} = 1_{(B, e_{gf})}
\end{align*}
So we conclude that $f f^{\frac{D}{g}} f: (A, e_{fg}) \to (B, e_{gf})$ is an isomorphism in $\mathsf{Split}(\mathbb{X})$ with inverse $f^{\frac{D}{g}}: (B, e_{gf}) \to (A, e_{fg})$, as desired. 
\item This is shown by  similar argument. 
\end{enumerate}
\end{proof}

We now prove the analogue of Theorem \ref{characterization-by-iteration} for Drazin inverses of opposing pairs:

\begin{theorem} In a category $\mathbb{X}$, $\xymatrixcolsep{1.5pc}\xymatrix{(f,g): A  \ar@<0.5ex>[r]  & B  \ar@<0.5ex>[l] }$ is Drazin if and only if there are idempotents $e_A: A \to A$ and $e_B: B \to B$, and a $k \in \mathbb{N}$, such that both of the maps $(fg)^kf: (A,e_A) \to (B,e_B)$ and $(gf)^kg: (B,e_B) \to (A,e_A)$ are isomorphisms\footnote{It is important to note that this statement does not say that $(fg)^kf$ and $(gf)^kg$ are inverses in $\mathsf{Split}(\mathbb{X})$} in $\mathsf{Split}(\mathbb{X})$. 
\end{theorem}

\begin{proof} For ($\Rightarrow$), suppose that $(f,g)$ is Drazin. So $fg$ and $gf$ are Drazin. As our idempotents, we take the induced ones $e_{fg}: A \to A$ and $e_{gf}: B \to B$. Now we first note that by Lemma \ref{lemma:Drazin-basic}.(\ref{lemma:Drazin-basic.1}) and using the same arguments as in the first half of the proof of Theorem \ref{characterization-by-iteration}, we have that for a $x: A \to A$ which is Drazin, for all $\mathsf{ind}(x) \leq j$ that $x^{j+1}: (A,e_x) \to (A,e_x)$ is an isomorphism in $\mathsf{Split}(\mathbb{X})$. So let $k = \mathsf{max}\left( \mathsf{ind}(fg), \mathsf{ind}(gf) \right)$. Then we get that $(fg)^{k+1}: (A,e_{fg}) \to (A,e_{fg})$ and $(gf)^{k+1}: (A,e_{gf}) \to (A,e_{gf})$ are isomorphisms. Now by Lemma \ref{lemma:e_gf+e_fg-iso}, we also have that $f^{\frac{D}{g}}: (B, e_{gf}) \to (A, e_{fg})$ and $g^{\frac{D}{f}}: (A, e_{gf}) \to  (B, e_{fg})$ are isomorphisms. So the composites, $(fg)^{k+1} g^{\frac{D}{f}}: (A, e_{gf}) \to  (B, e_{fg})$ and $(gf)^{k+1} f^{\frac{D}{g}}: (A,e_{gf}) \to (A,e_{gf})$ are also isomorphisms. However, applying Corollary \ref{cor:Drazinpairformulas} and Lemma \ref{lemma:Drazin-basic}.(\ref{lemma:Drazin-basic.1}), we get that: 
  \[ (fg)^{k+1} g^{\frac{D}{f}} = (fg)^{k+1} (fg)^D f = (fg)^k f   \]
  So $(fg)^{k+1} g^{\frac{D}{f}} = (fg)^k f$, and similarly we can also show that $(gf)^{k+1} f^{\frac{D}{g}} =  (gf)^kg$. So we conclude that $(fg)^kf: (A,e_A) \to (B,e_B)$ and $(gf)^kg: (B,e_B) \to (A,e_A)$ are isomorphisms in $\mathsf{Split}(\mathbb{X})$ as desired. 

  For ($\Leftarrow$), suppose for some $e_A: A \to A$ and $e_B: B \to B$, and a $k \in \mathbb{N}$, that both  $(fg)^kf: (A,e_A) \to (B,e_B)$ and $(gf)^kg: (B,e_B) \to (A,e_A)$ are isomorphisms in $\mathsf{Split}(\mathbb{X})$, with respective inverses $v: (B, e_B) \to (A,e_A)$ and $u: (A,e_A) \to (B,e_B)$. Explicitly this means we have the following identities: 
    \begin{align*}
        (fg)^kf e_B = (fg)^kf = e_A (fg)^kf && e_B (gf)^kg = (gf)^kg = (gf)^kg e_A \\
        v e_A = v = e_B v && e_A u = u = u e_B \\ 
        (fg)^kf v = e_A = u (gf)^kg && v (fg)^kf  = e_B = (gf)^kg u
    \end{align*}
We will now prove that $fg$ is Drazin. To do so, we will first show that $(fg)^{k+1}$ is an isomorphism in the idempotent splitting. Using the above identity, we first compute: 
\[ e_A (fg)^{k+1} = e_A (fg)^k fg = (fg)^k fg = (fg)^{k+1} \]
\[ (fg)^{k+1} e_A = f (gf)^k g e_A =  f (gf)^k g = (fg)^{k+1} \]
So we get that $(fg)^{k+1}: (A, e_A) \to (A,e_A)$ is indeed a map in $\mathsf{Split}(\mathbb{X})$. On the other hand, since $u: (A,e_A) \to (B,e_B)$ and $v: (B, e_B) \to (A,e_A)$ are maps in $\mathsf{Split}(\mathbb{X})$, it follows that $u (gf)^k v: (A, e_A) \to (A, e_A)$ is also a map $\mathsf{Split}(\mathbb{X})$. To show that this map is the inverse of $(fg)^{k+1}$, we first need to do some calculations:
\begin{gather*}
 (gf)^k v = (gf)^k e_B v = (gf)^k (gf)^k g u v = (gf)^k g (fg)^k u v = e_B  (gf)^k g (fg)^k u v \\
 = e_B (gf)^k (gf)^k g u v = e_B (gf)^k e_B v = e_B (gf)^k v
\end{gather*}
So $(gf)^k v = e_B (gf)^k v$, and similarly we can also show that $u (gf)^k =u (gf)^k e_A$. With these identities in hand, we can observe that: 
\begin{gather*}
   (fg)^{k+1} u (gf)^k v = f (gf)^k g u (gf)^k v = f e_B (gf)^k v = f(gf)^k v = (fg)^k f v = e_A = 1_{(A,e_A)} 
\end{gather*}
\begin{gather*}
      u (gf)^k v  (fg)^{k+1} =  u  (gf)^k v (fg)^k f g = u (gf)^k e_A g = u (gf)^k g = e_A = 1_{(A,e_A)} 
\end{gather*}
So we get that $(fg)^{k+1}: (A, e_A) \to (A,e_A)$ is an isomorphism in $\mathsf{Split}(\mathbb{X})$. Then by Theorem \ref{characterization-by-iteration}, $fg$ is Drazin. Therefore we conclude that $(f,g)$ is Drazin.
\end{proof}

\subsection{Drazin Opposing Pairs and Binary Idempotents} Recall from Section \ref{sec:binary-idempotent} that a binary idempotent is, using the terminology from this section, an opposing pair $\xymatrixcolsep{1.5pc}\xymatrix{(f,g): A  \ar@<0.5ex>[r]  & B  \ar@<0.5ex>[l] }$ such that $fgf=f$ and $gfg= g$, and this implies that $fg$ and $gf$ are idempotents. Here we explain how binary idempotents are linked to Drazin opposing pairs. Moreover, binary idempotents will play a key role in the rest of the sections. 

We first observe that similar to an idempotent being its own Drazin inverse, a binary idempotent regarded as an opposing pair is its own Drazin inverse. 

\begin{lemma} $\xymatrixcolsep{1.5pc}\xymatrix{(f,g): A  \ar@<0.5ex>[r]  & B  \ar@<0.5ex>[l] }$ is a binary idempotent if and only if $(g,f)$ is its Drazin inverse, so $(f,g)^D = (g,f)$ (so that $f^{\frac{D}{g}} = g$ and $g^{\frac{D}{f}} = f$).
\end{lemma}
\begin{proof} For $(\Rightarrow)$, Since $(f,g)$ is a binary idempotent, $fg$ and $gf$ are idempotents, and so by Lemma \ref{lemma:e-drazin}, are Drazin and their own Drazin inverses. Therefore $(f,g)$ is Drazin. Applying Corollary \ref{cor:Drazinpairformulas}, we compute:
\[ f^{\frac{D}{g}} = g(fg)^D = gfg = g \]
So $f^{\frac{D}{g}} = g$, and similarly, we can also show that $g^{\frac{D}{f}} = f$. 

For $(\Leftarrow)$, if $(f,g)$ is Drazin with itself as a Drazin inverse, then {\bf [DV.2]} gives us that 
$gfg = g$ and $fgf =f$.  Thus, $(f,g)$ is a binary idempotent.
\end{proof}

On the other hand, from any Drazin opposing pair we get \emph{two} binary idempotents. 

\begin{proposition} \label{prop:iterated-equivs} If $\xymatrixcolsep{1.5pc}\xymatrix{(f,g): A  \ar@<0.5ex>[r]  & B  \ar@<0.5ex>[l] }$ is Drazin then: 
\begin{enumerate}[(i)]
\item\label{prop:iterated-equivs.1} $(f,g)^D \colon = ( f^{\frac{D}{g}},g^{\frac{D}{f}})$ is Drazin with the Drazin inverse $(f,g)^{DD}\colon =(f f^{\frac{D}{g}}f, g g^{\frac{D}{f}}g)$ and $\mathsf{ind}((f,g)^D) \leq 1$; 
\item\label{prop:iterated-equivs.2} $(f,g)^{DD}\colon =(f f^{\frac{D}{g}}f, g g^{\frac{D}{f}}g)$ is Drazin with the Drazin inverse $(f,g)^D$; 
\item \label{prop:binary-idem.3} $(f f^{\frac{D}{g}}f,f^{\frac{D}{g}})$ and 
$(g g^{\frac{D}{f}}g, g^{\frac{D}{f}})$ are binary idempotents.
\end{enumerate}
\end{proposition}
\begin{proof} (i) Recall from Corollary \ref{cor:Drazinpairformulas} we have that $g^{\frac{D}{f}} f^{\frac{D}{g}} = (gf)^D$ and $f^{\frac{D}{g}} g^{\frac{D}{f}} = (fg)^D$. However by Lemma \ref{lem:Drazin-inverse1}.(\ref{Drazin-inverse-inverse}), a Drazin inverse is always Drazin, so $g^{\frac{D}{f}} f^{\frac{D}{g}}$ and $f^{\frac{D}{g}} g^{\frac{D}{f}}$ are Drazin. As such, $( f^{\frac{D}{g}},g^{\frac{D}{f}})$ is Drazin. Now applying the formulas from Lemma \ref{lem:Drazin-inverse1}.(\ref{Drazin-inverse-inverse}) and Corollary \ref{cor:Drazinpairformulas}, as well as using \textbf{[DV.2]} and \textbf{[DV.3]}, we compute:  
\begin{gather*}
    g^{\frac{D}{f}} ( f^{\frac{D}{g}} g^{\frac{D}{f}})^D = g^{\frac{D}{f}} ((gf)^D)^D =  g^{\frac{D}{f}} gf (gf)^D gf = g^{\frac{D}{f}} gf f^{\frac{D}{g}} g^{\frac{D}{f}} gf \\
    = f f^{\frac{D}{g}} f f^{\frac{D}{g}} f f^{\frac{D}{g}} f = f f^{\frac{D}{g}} f f^{\frac{D}{g}} f = f f^{\frac{D}{g}} f f^{\frac{D}{g}} f = f f^{\frac{D}{g}} f
\end{gather*}
So the Drazin inverse of $f^{\frac{D}{g}}$ over $g^{\frac{D}{f}}$ is $f f^{\frac{D}{g}} f$. Similarly, we can also compute that the Drazin inverse of $g^{\frac{D}{f}}$ over $f^{\frac{D}{g}}$ is $g g^{\frac{D}{f}}g$. \\

\noindent (ii) By (\ref{prop:iterated-equivs.1}), we know that $(f f^{\frac{D}{g}}f, g g^{\frac{D}{f}}g)$ is a Drazin inverse. So by (\ref{prop:iterated-equivs.1}) again, we get that $(f f^{\frac{D}{g}}f, g g^{\frac{D}{f}}g)$ has a Drazin inverse, which by \textbf{[DV.2]} is precisely $( f^{\frac{D}{g}},g^{\frac{D}{f}})$. \\

\noindent (iii) Using {\bf [DV.2]}, we compute that: 
\begin{gather*}
    (ff^{\frac{D}{g}}f) f^{\frac{D}{g}} (ff^{\frac{D}{g}}f) = f(f^{\frac{D}{g}}f f^{\frac{D}{g}}) f (f^{\frac{D}{g}}ff^{\frac{D}{g}})f = f(f^{\frac{D}{g}} f f^{\frac{D}{g}}) f = f f^{\frac{D}{g}}f \end{gather*}
    \begin{gather*}
    f^{\frac{D}{g}} (f f^{\frac{D}{g}} f) f^{\frac{D}{g}} = (f^{\frac{D}{g}} f f^{\frac{D}{g}}) f f^{\frac{D}{g}} = f^{\frac{D}{g}} f f^{\frac{D}{g}} = f^{\frac{D}{g}}
\end{gather*}
So $(ff^{\frac{D}{g}}f,f^{\frac{D}{g}})$ is a binary idempotent. Similarly we can also show that $(g g^{\frac{D}{f}}g, g^{\frac{D}{f}})$ is a binary idempotent.
\end{proof}

Using binary idempotents, we can also characterize when an opposing pair is the Drazin inverse of another opposing pair. To do so, it will be first useful to observe the following: 

\begin{lemma}\label{lemma:a-drazin-inverse} $\xymatrixcolsep{1.5pc}\xymatrix{(f,g): A  \ar@<0.5ex>[r]  & B  \ar@<0.5ex>[l] }$ is a Drazin inverse if and only if there is an opposing pair of dual type $\xymatrixcolsep{1.5pc}\xymatrix{(f^{\frac{D}{g}},g^{\frac{D}{f}}): B  \ar@<0.5ex>[r]  & A  \ar@<0.5ex>[l] }$ which satisfy {\bf [DV.1$^\prime$]} $f f^{\frac{D}{g}} f = f$ and $g g^{\frac{D}{f}} g = g$, {\bf [DV.2]} and {\bf [DV.3]}.
\end{lemma} 
\begin{proof}For ($\Rightarrow$), suppose that $(f,g)$ is the Drazin inverse of $(h,k)$, so $(f,g) = (h^{\frac{D}{k}}, k^{\frac{D}{h}})$. Then by Proposition \ref{prop:iterated-equivs}.(\ref{prop:iterated-equivs.1}), we get that $(f,g)$ is Drazin with Drazin inverse $( f^{\frac{D}{g}},g^{\frac{D}{f}})= ( h h^{\frac{D}{k}} h ,  k k^{\frac{D}{h}} k ) = (h f h, k g k)$. This means that {\bf [DV.2]} and {\bf [DV.3]} hold. For {\bf [DV.1$^\prime$]}, since $f= h^{\frac{D}{k}}$ and $f^{\frac{D}{g}} = h h^{\frac{D}{k}} h$, using {\bf [DV.2]} twice we compute that: 
\[ f f^{\frac{D}{g}} f =  h^{\frac{D}{k}} h h^{\frac{D}{k}} h h^{\frac{D}{k}} = h^{\frac{D}{k}} h h^{\frac{D}{k}} =  h^{\frac{D}{k}} = f \]
So $f f^{\frac{D}{g}} f = f$, and similarly we can also show that $g g^{\frac{D}{f}} g = g$. 

For ($\Leftarrow$), we show that $(f^{\frac{D}{g}},g^{\frac{D}{f}})$ is the Drazin inverse of $(f,g)$. By assumption, {\bf [DV.2]} and {\bf [DV.3]} hold, so it remains to show {\bf [DV.1]}. Using {\bf [DV.1$^\prime$]} and {\bf [DV.3]}, we show that {\bf [DV.1]} holds for $k=1$:
\begin{align*}
   fg ff^{\frac{D}{g}} = fg g^{\frac{D}{f}}g = fg && gf g g^{\frac{D}{f}} = gf f^{\frac{D}{g}}f= gf
\end{align*}
So $(f,g)$ is Drazin. Now by {\bf [DV.1$^\prime$]} we have that $(f,g) = (f f^{\frac{D}{g}}f, g g^{\frac{D}{f}}g)$. Then Proposition \ref{prop:iterated-equivs}.(\ref{prop:iterated-equivs.1}) tells us that $(f,g) = (f f^{\frac{D}{g}}f, g g^{\frac{D}{f}}g)$ is the Drazin inverse of $( f^{\frac{D}{g}},g^{\frac{D}{f}})$. 
\end{proof}

\begin{proposition}\label{prop:a-drazin-inverse} $\xymatrixcolsep{1.5pc}\xymatrix{(f,g): A  \ar@<0.5ex>[r]  & B  \ar@<0.5ex>[l] }$ is a Drazin inverse if and only if $(f,g)$ is Drazin and both
$(f,f^{\frac{D}{g}})$ and $(g^{\frac{D}{f}},g)$ are binary idempotents.
\end{proposition}
\begin{proof} For ($\Rightarrow$), this is simply Lemma \ref{prop:iterated-equivs}. For ($\Leftarrow$), since by assumption $(f,g)$ is Drazin (so {\bf [DV.2]} and {\bf [DV.3]} hold), by Lemma \ref{lemma:a-drazin-inverse} we need only show that {\bf [DV.1$^\prime$]} holds. However this is immediate since $(f,f^{\frac{D}{g}})$ and $(g^{\frac{D}{f}},g)$ are binary idempotents. So by Lemma \ref{lemma:a-drazin-inverse} we get that $(f,g)$ is a Drazin inverse. 
\end{proof}

\begin{corollary}\label{cor:a-drazin-inv-is-drazin} If $\xymatrixcolsep{1.5pc}\xymatrix{(f,g): A  \ar@<0.5ex>[r]  & B  \ar@<0.5ex>[l] }$ is a Drazin inverse if and only if $(f,g)$ is Drazin, ${\sf ind}((f,g)) \leq 1$,and satisfies {\bf [DV.1$^\prime$]} $f f^{\frac{D}{g}} f = f$ and $g g^{\frac{D}{f}} g = g$. 
\end{corollary}
\begin{proof} For ($\Rightarrow$), by Proposition \ref{prop:iterated-equivs}.(\ref{prop:iterated-equivs.1}), we get that $(f,g)$ is Drazin with ${\sf ind}((f,g)) \leq 1$. While that $(f,g)$ and its Drazin inverse satisfy {\bf [DV.1$^\prime$]} follows Proposition \ref{prop:a-drazin-inverse} that $(f,f^{\frac{D}{g}})$ and $(g^{\frac{D}{f}},g)$ are binary idempotents. For ($\Leftarrow$), since by assumption $(f,g)$ is Drazin (so {\bf [DV.2]} and {\bf [DV.3]} hold), and also that {\bf [DV.1$^\prime$]} holds, by Lemma \ref{lemma:a-drazin-inverse} we get that $(f,g)$ is a Drazin inverse. 
\end{proof}

\subsection{Being Drazin in a Dagger Category}\label{sec:dagger} In this section, we consider Drazin inverses in \emph{dagger} categories. In particular, in the next section we will relate Drazin inverses to another sort of generalized inverse that arises in dagger categories, called the \emph{Moore-Penrose inverse}. 

Recall that a \textbf{dagger category} is a category $\mathbb{X}$ equipped with a contravariant functor $\dagger: \mathbb{X}^{op} \to \mathbb{X}$ which is the identity on objects and involutive. Explicitly, this means that for every map $f: A \to B$, there is a map of dual type $f^\dagger: B \to A$ called the \textbf{adjoint} of $f$ such that $(fg)^\dagger = g^\dagger f^\dagger$, $1_A^\dagger = 1_A$, and $f^{\dagger \dagger} = f$.  

For a more in-depth introduction to dagger categories, we refer the reader to \cite{HeunenVicary}.  There are plenty of examples of dagger categories: in particular, the category of opposing maps itself is the cofree dagger category of a category \cite[Definition 3.1.16]{heunen2009categorical}. Thus we have already secretly been working with dagger categories! 

Since the adjoint of an endomorphism is an endomorphism, it makes sense to ask whether the adjoint of a Drazin endomorphism is Drazin. It turns out that the Drazin inverse of the adjoint is the adjoint of the Drazin inverse:

\begin{lemma}\label{lemma:Drazin-dagger} In a dagger category $\mathbb{X}$, if $x: A \to A$ is Drazin then $x^\dagger: A \to A$ is also Drazin where $(x^\dagger)^D = (x^D)^\dagger$. 
\end{lemma}

\begin{proof} Clearly, if $x$ is Drazin in $\mathbb{X}$ then $x$ is Drazin in $\mathbb{X}^{op}$ with $x^D$ as its Drazin inverse. Since Drazin inverses are absolute by Proposition \ref{Drazin-absolute}, it follows that $x^\dagger$ is Drazin and $(x^\dagger)^D = (x^D)^\dagger$ as desired. 
\end{proof}

In a dagger category, for any arbitrary map $f: A \to B$, we get the opposing pair $\xymatrixcolsep{1.5pc}\xymatrix{(f,f^\dagger): A  \ar@<0.5ex>[r]  & B  \ar@<0.5ex>[l] }$, which we call an \textbf{adjoint opposing pair}. When an adjoint opposing pair $(f, f^\dagger)$ is Drazin, it turns out that the Drazin inverse of the map over its adjoint, that is $f^{\frac{D}{f^\dagger}}$, is the adjoint of the Drazin inverse of the adjoint over the map: 

\begin{lemma} \label{Lemma:Drazin-inverse-dagger-2}In a dagger category $\mathbb{X}$, if $\xymatrixcolsep{1.5pc}\xymatrix{(f,f^\dagger): A  \ar@<0.5ex>[r]  & B  \ar@<0.5ex>[l] }$ has a Drazin inverse then:
\begin{enumerate}[(i)]
\item \label{Drazin-inverse-dagger-2} $f^{\frac{D}{f^\dagger}}= ((f^\dagger)^{\frac{D}{f}})^\dagger$ and so $(f,f^\dagger)^D$ is an adjoint pair;
\item \label{Drazin-inverse-dagger-3} The induced idempotents are \textbf{$\dagger$-idempotents}, that is, $e_{ff^\dagger}^\dagger = e_{ff^\dagger}$ and $e_{f^\dagger f}^\dagger = e_{f^\dagger f}$. 
\end{enumerate}
\end{lemma}

\begin{proof} (i) If $(f,f^\dagger)$ has a Drazin inverse, by Theorem \ref{Drazin-equiv-inverse} $(f,f^\dagger)$ is Drazin, so $ff^\dagger$ and $f^\dagger f$ are Drazin. Then using Corollary \ref{cor:Drazinpairformulas} and Lemma \ref{lemma:Drazin-dagger}, we compute:  
\[ ((f^\dagger)^{\frac{D}{f}})^\dagger = ( f(f^\dagger f)^D )^\dagger = ((f^\dagger f)^D)^\dagger f^\dagger = ((f^\dagger f)^\dagger)^D f^\dagger = (f^\dagger f)^D f^\dagger = f^{\frac{D}{f^\dagger}} \]
So $f^{\frac{D}{f^\dagger}}= ((f^\dagger)^{\frac{D}{f}})^\dagger$ as desired. \\

\noindent (ii) By Lemma \ref{lemma:e_gf+e_fg-iso}.(\ref{e_fg+e_gf}), we know that $e_{ff^\dagger}= (f^\dagger)^{\frac{D}{f}} f^{\frac{D}{f^\dagger}}$ and $e_{f^\dagger f} =  f^{\frac{D}{f^\dagger}} (f^\dagger)^{\frac{D}{f}}$. However by (\ref{Drazin-inverse-dagger-2}), we then get that $e_{ff^\dagger}= \left(f^{\frac{D}{f^\dagger}} \right)^\dagger f^{\frac{D}{f^\dagger}}$ and $e_{f^\dagger f} =  f^{\frac{D}{f^\dagger}} \left(f^{\frac{D}{f^\dagger}} \right)^\dagger $ -- which implies that $e_{ff^\dagger}$ and $e_{f^\dagger f}$ are $\dagger$-idempotents. 
\end{proof}

We can also precisely say when an adjoint opposing pair is the Drazin inverse of another adjoint opposing pair. 

\begin{lemma}\label{lemma:a-drazin-inverse-dagger} In a dagger category $\mathbb{X}$, $\xymatrixcolsep{1.5pc}\xymatrix{(f,f^\dagger): A  \ar@<0.5ex>[r]  & B  \ar@<0.5ex>[l] }$ is a Drazin inverse if and only if $(f,f^\dagger)$ is Drazin with $\mathsf{ind}(f,f^\dagger) \leq 1$ and $f f^{\frac{D}{f^\dagger}} f = f$. 
\end{lemma}
\begin{proof} For ($\Rightarrow$), this is immediate from Lemma \ref{lemma:a-drazin-inverse}. For ($\Leftarrow$), if $f f^{\frac{D}{f^\dagger}} f = f$, then Lemma \ref{Lemma:Drazin-inverse-dagger-2}.(\ref{Drazin-inverse-dagger-2}) tell us that the dagger of this identity is precisely $f^\dagger (f^{\frac{D}{f^\dagger}})^\dagger f^\dagger = f^\dagger$. Then by applying Lemma \ref{lemma:a-drazin-inverse}, we get that $(f,f^\dagger)$ is a Drazin inverse. \end{proof}

\subsection{The Moore-Penrose Inverse}
The remainder of this section is dedicated to discussing the relationship between Drazin opposing pairs and Moore-Penrose inverses. For a more in-depth introduction to Moore-Penrose inverses and examples, we invite the reader to see \cite{campbell2009generalized,EPTCS384.10}.

In a dagger category $\mathbb{X}$, a \textbf{Moore-Penrose inverse} \cite[Def 2.3]{EPTCS384.10} of a map $f: A \to B$ is a map of dual type $f^\circ: B \to A$ such that: 
 \[ \textbf{[MP.1]}~ f f^\circ f = f; ~~~~~\textbf{[MP.2]}~ f^\circ f f^\circ = f^\circ; ~~~~~\textbf{[MP.3]}~ (ff^\circ)^\dagger = ff^\circ;  ~~~~~\textbf{[MP.4]}~ (f^\circ f)^\dagger = f^\circ f. \]   
 
Like Drazin inverses, Moore-Penrose inverses are unique \cite[Lemma 2.4]{EPTCS384.10}. Moreover, if $f$ has a Moore-Penrose inverse, then so does $f^\dagger$, where $f^{\dagger \circ} = f^{\circ \dagger}$ \cite[Lemma 2.5.(ii)]{EPTCS384.10}. We also observe that having a Moore-Penrose inverse gives a binary idempotent:

\begin{lemma} \label{lem:binary-idm-of-MP}
    In a dagger category $\mathbb{X}$, if $f: A \to B$ has a Moore-Penrose inverse if and only if $(f,f^\circ)$ is a binary idempotent 
    whose induced idempotents are $\dagger$-idempotents.
\end{lemma}

\begin{proof}
    {\bf [MP.1]} and {\bf [MP.2]} ensure that it is a binary idempotent while {\bf [MP.3]} and {\bf [MP.4]} ensure that the induced idempotents are $\dagger$-idempotents.
\end{proof}

Having a Moore-Penrose inverse can also be characterized in terms of Drazin inverses. Specifically a map has a Moore-Penrose inverse if and only if its induced adjoint opposing pair \emph{is} a Drazin inverse. 

\begin{theorem}\label{thm:MP=Drazin} In a dagger category $\mathbb{X}$, $f: A \to B$ has a Moore-Penrose inverse if and only if $\xymatrixcolsep{1.5pc}\xymatrix{(f,f^\dagger): A  \ar@<0.5ex>[r]  & B  \ar@<0.5ex>[l] }$ is a Drazin inverse.
\end{theorem}

Note that \emph{being} a Drazin inverse is a considerably stronger property than \emph{having} a Drazin inverse.

\begin{proof} For $(\Rightarrow)$, we show that $(f^{\frac{D}{f^\dagger}}, (f^\dagger)^{\frac{D}{f}}) \colon = (f^\circ, f^{\circ \dagger})$ satisfies the requirements from Lemma \ref{lemma:a-drazin-inverse}. 
\begin{description}
\item[{\bf [DV.1$^\prime$]}] Using \textbf{[MP.1]}, we have: 
\begin{gather*}
 f^\dagger (f^\dagger)^{\frac{D}{f}} f^\dagger  = f^\dagger f^{\circ \dagger}  f^\dagger = ( f f^\circ f)^\dagger = f^\dagger \\
    f f^{\frac{D}{f^\dagger}} f = f f^\circ f = f 
\end{gather*}
\item[{\bf [DV.2]}] Using \textbf{[MP.2]}, we have: 
\begin{gather*}
 f^{\frac{D}{f^\dagger}} f f^{\frac{D}{f^\dagger}} = f^\circ f f^\circ = f^\circ = f^{\frac{D}{f^\dagger}} \\ (f^\dagger)^{\frac{D}{f}} f^\dagger (f^\dagger)^{\frac{D}{f}} = f^{\circ \dagger}  f^\dagger f^{\circ \dagger} = (f^\circ f f^\circ)^\dagger =   f^{\circ \dagger} =  (f^\dagger)^{\frac{D}{f}}
\end{gather*}
\item[{\bf [DV.3]}] Using \textbf{[MP.3]} and \textbf{[MP.4]}, we have: 
\begin{gather*}
    f f^{\frac{D}{f^\dagger}} = f f^\circ = (f f^\circ)^\dagger =  f^{\circ \dagger} f^\dagger =  (f^\dagger)^{\frac{D}{f}} f^\dagger \\ 
    f^{\frac{D}{f^\dagger}} f =  f^\circ f = ( f^\circ f )^\dagger =  f^\dagger f^{\circ \dagger} =  f^\dagger(f^\dagger)^{\frac{D}{f}} 
\end{gather*}
\end{description}
So by Lemma \ref{lemma:a-drazin-inverse}, $(f,f^\dagger)$ is a Drazin inverse. In fact, $(f,f^\dagger)$ is Drazin with Drazin inverse $(f^{\circ \dagger}, f^\circ)$ and $\mathsf{ind}(f,g) \leq 1$. 

For $(\Leftarrow)$, by Corollary \ref{cor:a-drazin-inv-is-drazin}, $(f,f^\dagger)$ is Drazin. So we will show that $f^\circ = f^{\frac{D}{f^\dagger}}$ satisfies the four axioms of a Moore-Penrose inverse. However note that {\bf [MP.1]} is precisely {\bf [DV.1$^\prime$]}, while \textbf{[MP.2]} is precisely \textbf{[DV.3]}. Now for \textbf{[MP.3]}, using Lemma \ref{Lemma:Drazin-inverse-dagger-2} and \textbf{[DV.3]}, we compute: 
\begin{gather*}
    (f f^\circ)^\dagger = (f^\circ)^\dagger f^\dagger = (f^{\frac{D}{f^\dagger}})^\dagger f^\dagger = (f^\dagger)^{\frac{D}{f}} f^\dagger = f f^{\frac{D}{f^\dagger}} = f f^\circ 
\end{gather*}
While \textbf{[MP.4]} is shown using similar calculations. So we conclude that $f$ has a Moore-Penrose inverse.
\end{proof}

In general, while not every map in a dagger category has a Moore-Penrose inverse, for an adjoint opposing pair which is Drazin, the Drazin inverse of the map over its adjoint always has a Moore-Penrose inverse. 

\begin{lemma}\label{lemma:Drazin-MP} In a dagger category $\mathbb{X}$, if $\xymatrixcolsep{1.5pc}\xymatrix{(f,f^\dagger): A  \ar@<0.5ex>[r]  & B  \ar@<0.5ex>[l] }$ is Drazin then $f^{\frac{D}{f^\dagger}}$ has a Moore-Penrose inverse, where $\left(f^{\frac{D}{f^\dagger}} \right)^\circ = f f^{\frac{D}{f^\dagger}} f$. 
\end{lemma}
\begin{proof} If $(f,f^\dagger)$ is Drazin, by Lemma \ref{Lemma:Drazin-inverse-dagger-2}.(\ref{Drazin-inverse-dagger-2}), we have that $(f^{\frac{D}{f^\dagger}},(f^{\frac{D}{f^\dagger}})^\dagger)$ is its Drazin inverse. In other words, using Lemma \ref{lemma:sym}, we get that $((f^{\frac{D}{f^\dagger}})^\dagger, f^{\frac{D}{f^\dagger}})$ is the Drazin inverse of the adjoint opposing pair $(f^\dagger, f)$. Then by Theorem \ref{thm:MP=Drazin}, $f^{\frac{D}{f^\dagger}}$ has a Moore-Penrose inverse given by $\left(f^{\frac{D}{f^\dagger}} \right)^\circ = f f^{\frac{D}{f^\dagger}} f$.
\end{proof}

Now a \textbf{Moore-Penrose dagger category} \cite[Def 2.3]{EPTCS384.10} is a dagger category such that every map has a Moore-Penrose inverse. Then by Theorem \ref{thm:MP=Drazin}, we obtain a novel equivalent description of when a dagger category is Moore-Penrose. 

\begin{corollary}\label{cor:MPcat} A dagger category $\mathbb{X}$ is Moore-Penrose if and only if every adjoint opposing pair is a Drazin inverse. 
\end{corollary}

\subsection{When Moore-Penrose is Drazin} We conclude with the natural question of when the Drazin inverse and Moore-Penrose inverse coincide. For complex square matrices, this happens precisely when the matrix is an EP matrix \cite[Thm 7.3.4]{campbell2009generalized}, which means that it commutes with its Moore-Penrose inverse \cite[Def 4.3.1]{campbell2009generalized}. Here we show that the same holds true in an arbitrary dagger category. 

\begin{definition} In a dagger category $\mathbb{X}$, a map $f: A \to B$ is said to be \textbf{EP} if it has a Moore-Penrose inverse and it commutes with it, that is, $ff^\circ =f^\circ f$. 
\end{definition}

\begin{definition} In a dagger category $\mathbb{X}$, an endomorphism $x: A \to A$ is said to be \textbf{Moore-Penrose-Drazin} if $x$ is Drazin such that its Drazin inverse is also its Moore-Penrose inverse, $x^D = x^\circ$. 
\end{definition}

\begin{theorem}\label{lemma:EP} In a dagger category $\mathbb{X}$, for an endomorphism $x: A \to A$, the following are equivalent: 
\begin{enumerate}[(i)]
\item $x$ is Moore-Penrose-Drazin;
\item $x$ is Drazin with $\mathsf{ind}(x) \leq 1$ and the induced idempotent $e_x$ is a $\dagger$-idempotent; 
\item $x$ is EP. 
\end{enumerate}
\end{theorem}
\begin{proof} For $(i) \Rightarrow (ii)$, by definition $x$ is Drazin. Using \textbf{[DV.3]} and \textbf{[MP.1]}, we compute:  
\[ x^{2} x^D = x x x^D = x x^D x = x x^\circ x = x  \]
So $\mathsf{ind}(x) \leq 1$. Now using \textbf{[MP.3]}, we compute that: 
\[ e_x^\dagger = (xx^D)^\dagger = (x x^\circ)^\dagger = x x^\circ  = e_x \]
   So $e_x^\dagger$ is a $\dagger$-idempotent.  

For $(ii) \Rightarrow (iii)$, we show that $x^D$ satisfies the four Moore-Penrose inverse axioms. However, note that since $\mathsf{ind}(x) \leq 1$, $x^D$ is a group inverse of $x$. Then \textbf{[MP.1]} and \textbf{[MP.2]} are precisely the same as \textbf{[G.1]} and \textbf{[G.2]} respectively. Now since $e_x^\dagger$ is a $\dagger$-idempotent, meaning $e_x^\dagger = e_x$, which gives us precisely that $(xx^D)^\dagger = x x^D$ and $(x^D x)^\dagger = x^D x$, which are \textbf{[MP.3]} and \textbf{[MP.4]}. So $x^\circ = x^D$ is a Moore-Penrose inverse of $x$, and by \textbf{[D.3]}, we also have that $x x^\circ = x^\circ x$.

For $(iii) \Rightarrow (i)$, we show that setting $x^D = x^\circ$ satisfies the three group inverse axioms. However, note that \textbf{[G.1]} and \textbf{[G.2]} are the same as \textbf{[MP.1]} and \textbf{[MP.2]}. While \textbf{[G.3]} is precisely the extra assumption $x x^\circ = x^\circ x$. So $x^\circ$ is a group inverse of $x$, which implies that $x$ is Drazin, whose Drazin inverse is its Moore-Penrose inverse. Thus, $x$ is Moore-Penrose-Drazin as desired. 
\end{proof}

In particular, for any map with a Moore-Penrose inverse, we always get two Moore-Penrose-Drazin endomorphisms. 

\begin{lemma} In a dagger category $\mathbb{X}$, if $f: A \to B$ has a Moore-Penrose inverse, then $ff^\dagger$ and $f^\dagger f$ are Moore-Penrose-Drazin. 
\end{lemma}
\begin{proof} By Theorem \ref{thm:MP=Drazin} and Corollary \ref{cor:a-drazin-inv-is-drazin}, $(f,f^\dagger)$ is Drazin with $\mathsf{ind}(f,f^\dagger) \leq 1$, which tells us that $ff^\dagger$ and $f^\dagger f$ are Drazin with $\mathsf{ind}(ff^\dagger) \leq 1$ and $\mathsf{ind}(f^\dagger f) \leq 1$ (where the index inequalities comes from the proof of Theorem \ref{Drazin-equiv-inverse}). While Lemma \ref{Lemma:Drazin-inverse-dagger-2}.(\ref{Drazin-inverse-dagger-3}) tells us that $e_{ff^\dagger}$ and $e_{ff^\dagger}$ are $\dagger$-idempotents. So by Lemma \ref{lemma:EP}, we conclude that $ff^\dagger$ and $f^\dagger f$ are Moore-Penrose-Drazin. 
\end{proof}



\bibliographystyle{plain}      
\bibliography{drazin}   
\end{document}